\documentclass[reqno,english]{amsart}%
\usepackage{amsfonts,amsmath,latexsym,verbatim,amscd,mathrsfs,color,array}
\usepackage[colorlinks=true]{hyperref}
\usepackage{amsmath,amssymb,amsthm,amsfonts,graphicx,color}
\usepackage{amssymb}
\usepackage{pdfsync}
\usepackage{epstopdf}
\usepackage{cite}
\usepackage{graphicx}
\usepackage{amsmath}
\usepackage{amsfonts}%
\usepackage[a4paper]{geometry}
\providecommand{\U}[1]{\protect\rule{.1in}{.1in}}

\numberwithin{equation}{section}

\newtheorem{theorem}{Theorem}[section]
\newtheorem{corollary}{Corollary}[section]
\newtheorem{lemma}{Lemma}[section]
\newtheorem{proposition}{Proposition}[section]
\newtheorem{remark}{Remark}[section]
\newtheorem{definition}{Definition}[section]

\numberwithin{equation}{section}

\newcommand{\bbr}{\mathbb{R}}

\newcommand{\bbn}{\mathbb{N}}
\newcommand{\ve}{\varepsilon}

\newcommand{\bd}{\begin{definition}}
\newcommand{\ed}{\end{definition}}

\newcommand{\br}{\begin{remark}}
\newcommand{\er}{\end{remark}}

\newcommand{\be}{\begin{equation}}
\newcommand{\ee}{\end{equation}}

\newcommand{\bc}{\begin{corollary}}
\newcommand{\ec}{\end{corollary}}
\begin{document}

\title[The Struwe decomposition]{A refined blow-up analysis of the Brezis-Nirenberg equation and its application: The one-bubble case for $N\geq4$}

\author[R. He]{Rui He}
\address{\noindent  School of Mathematics, Yunnan Normal University, Kunming, 650500, P.R. China\\
Yunnan Key Laboratory of Modern Analytical Mathematics and Applications, Kunming, 650500, P.R. China}
\email{xxxxxx@ynnu.edu.cn}

\author[X. Liu]{Xiangqing Liu}
\address{\noindent  School of Mathematics, Yunnan Normal University, Kunming, 650500, P.R. China\\
Yunnan Key Laboratory of Modern Analytical Mathematics and Applications, Kunming, 650500, P.R. China}
\email{liuxiangqing@ynnu.edu.cn}

\author[J. Wei]{Juncheng Wei}
\address{\noindent Department of Mathematics, Chinese University of Hong Kong,
Shatin, NT, Hong Kong}
\email{wei@math.cuhk.edu.hk}

\author[Y. Wu]{Yuanze Wu}
\address{\noindent  School of Mathematics, Yunnan Normal University, Kunming, 650500, P.R. China\\
Yunnan Key Laboratory of Modern Analytical Mathematics and Applications, Kunming, 650500, P.R. China}
\email{yuanze.wu@ynnu.edu.cn}

\thanks{R. He and X. Liu are supported by NSFC (No. 12161093) and Modern Applied Mathematics and Life Sciences Interdisciplinary Research Team in Yunnan (No. 202405AS350003).  J. Wei is partially supported by GRF of Hong Kong entitled "New frontiers in singularity formations in nonlinear partial differential equations".   Y. Wu is supported by NSFC (No. 12171470), Modern Applied Mathematics and Life Sciences Interdisciplinary Research Team in Yunnan (No. 202405AS350003) and Yunnan Fundamental Research Projects (No. 202601CJ070001). }

\begin{abstract}
In this paper, we consider the famous Brezis-Nirenberg equation
\begin{eqnarray*}
\left\{
\aligned
&-\Delta u=\lambda u+|u|^{\frac{4}{N-2}}u,\quad&\mbox{in}\,\, \Omega,\\
&u=0,\quad&\mbox{on}\,\, \partial\Omega,
\endaligned
\right.
\end{eqnarray*}
where $N\geq3$ is the dimension, $\Omega\subset\mathbb{R}^N$ is a bounded domain with smooth boundary $\partial\Omega$ and $\lambda>0$ is a parameter.  By developing a refined blow-up analysis based on the inverse reduction argument developed in \cite{WW2019,WW2019-2}, we classify, for the fist time, the Struwe decomposition of the Brezis-Nirenberg equation in the one-bubble case as the parameter $\lambda$ varies for $N\geq4$.  As applications, we prove that the $4d$ Brezis-Nirenberg equation has a nontrivial solution (least energy solution) for $\lambda\in\sigma(-\Delta)$ in general bounded domains, where $\sigma(-\Delta)$ is the spectrum of $-\Delta$ in $H^1_0(\Omega)$.  Our result completes the existence theory of the Brezis-Nirenberg equation for $N\geq4$ in \cite{AP2025,CFP1985,CFS,CSS1986,CW2005,CSZ2012,SWW2009,TYZ2022} since 1984.

\vspace{3mm} \noindent{\bf Keywords:} Brezis-Nirenberg equation; Struwe decomposition; Refined blow-up analysis; Sign-changing solution; Inverse reduction argument.

\vspace{3mm}\noindent {\bf AMS} Subject Classification 2020: 35B33; 35B40; 35B44; 35J15.%
\end{abstract}

\date{}

\maketitle

\section{Introduction}
\subsection{Background}
In this paper, we consider the famous Brezis-Nirenberg equation
\begin{eqnarray}\label{BN}
\left\{
\aligned
&-\Delta u=\lambda u+|u|^{\frac{4}{N-2}}u,\quad&\mbox{in}\,\, \Omega,\\
&u=0,\quad&\mbox{on}\,\, \partial\Omega,
\endaligned
\right.
\end{eqnarray}
where $\lambda>0$ is a parameter, $\Omega\subset\bbr^N$ is a bounded domain with smooth boundary $\partial\Omega$ and $N\geq3$.

\vskip0.12in

This famous model was introduced by Brezis and Nirenberg in their celebrated paper \cite{BN} as an analogous example of the Yamabe problem.  The Yamabe problem concerns the existence of a metric $g$ conformal to $g_0$ in an $N$-dimensional Riemannian manifold $(\mathcal{M}, g_0)$ whose scalar curvature $R_g$ is equal to a prescribed function $R$.  It boils down to showing the existence of a positive solution $u$, serving as the conformal factor $g=u^{\frac{N-2}{2}}g_0$, to the nonlinear partial differential equation
\begin{eqnarray}\label{Yamabe}
-\Delta_{g_0}u+\frac{N-2}{4(N-1)}R_{g_0} u= \frac{N-2}{4(N-1)}R u^{2^*-1}\quad\text{in}\quad\mathcal{M},
\end{eqnarray}
where $-\Delta_{g_0}$ is the Laplace-Beltrami operator in $\mathcal{M}$ and $N\geq3$.  Thus, the Brezis-Nirenberg equation~\eqref{BN} is a generalization of the Yamabe problem~\eqref{Yamabe} by introducing a free parameter $\lambda$ if $\mathcal{M}$ is a smoothly bounded domain in $\mathbb{R}^N$.  By considering the following variational problem
\begin{eqnarray}\label{GS}
S_\lambda=\inf_{u\in H^1_0(\Omega)\backslash\{0\}}\frac{\|\nabla u\|_{L^2(\Omega)}^2-\lambda\|u\|_{L^2(\Omega)}^2}{\|u\|_{L^{2^*}(\Omega)}^{2^*}}
\end{eqnarray}
where we use the notation $2^*=\frac{2N}{N-2}$ to denote the critical Sobolev exponent as usual,
Brezis and Nirenberg gave a detailed study on the existence theory of positive solutions of \eqref{BN}:
\begin{enumerate}
\item[$(i)$]\quad \eqref{BN} only has the trivial solution when $\Omega$ is star-shaped about an inner point of $\Omega$ and $\lambda\leq 0$;
\item[$(ii)$]\quad \eqref{BN} has no positive solutions for $\lambda\geq\lambda_1$ where $\lambda_1$ is the first eigenvalue of $-\Delta$ in $H_0^1(\Omega)$;
\item[$(iii)$]\quad \eqref{BN} admits a positive solution for every $\lambda\in (0, \lambda_1)$ if $N\geq 4$;
\item[$(iv)$]\quad \eqref{BN} admits a positive solution for every $\lambda\in (\lambda^*, \lambda_1)$ if $N=3$, where $0<\lambda_*<\lambda_1$ is given by $\lambda_*=\inf\left\{\lambda\in\bbr\mid S_{\lambda}<S\right\}$ and $S$ is the optimal Sobolev constant;
\item[$(v)$]\quad \eqref{BN} admits a positive solution if and only if $\lambda\in (\lambda^*, \lambda_1)$ for $N=3$ and $\Omega=B_1$ the unit ball, where $\lambda^*=\frac{\lambda_1}{4}$ for $\Omega=B_1$.
\end{enumerate}
Bahri and Coron \cite{BC} presented an existence result for a positive solution to \eqref{BN} for $\Omega$ with a nontrivial topology and $\lambda=0$ by the same variational method, which further completes the very deep study on the existence theory of positive solutions to \eqref{BN} in \cite{BN}.
Rey \cite{R} introduced the Lyapunov-Schmidt reduction argument to establish the existence theory of positive solutions of \eqref{BN} for $N\geq5$ and $\lambda>0$ sufficiently small by proving the nondegeneracy of the Aubin-Talenti bubbles \cite{A,T} given by
\begin{eqnarray}\label{Talenti}
U_{\mu, \xi}(x)=\frac{\alpha_N\mu^{\frac{N-2}{2}}}{\left(\mu^2+|x-\xi|^2\right)^{\frac{N-2}{2}}},
\end{eqnarray}
where $\alpha_N=(N(N-2))^{\frac{N-2}{4}}$ and we say that $U_{\mu, \xi}$ is nondegenerate in the sense that the only bounded solutions to the linearization of the Yamabe equation in $\mathbb{R}^N$ at $U_{\mu, \xi}$ are spanned by $\partial_{\xi_i}U_{\mu, \xi}$, $1\leq i\leq N$, and $\partial_{\mu}U_{\mu, \xi}$.  The Lyapunov-Schmidt reduction argument was improved by Musso and Pistoia \cite{MP} for the multiplicity of positive solutions of equation~\eqref{BN}, and was generalized by del Pino, Dolbeault and Musso \cite{dPDM2006}, Musso and Salazar \cite{MS2018} and Pistoia, Rocci and Vaira \cite{PMV2025} for $\lambda\to\lambda_*$ in the case $N=3$ and for $\lambda\to0$ in the case $N\geq4$, respectively.  The positive solutions constructed by the Lyapunov-Schmidt reduction all blow up at stable critical points of the Kirchhoff-Routh function of $\Omega$, which is given by
\begin{eqnarray*}\label{eqn0002}
\mathcal{K}(\pmb{\mu},\pmb{x})=\frac12\left(\sum_{j=1}^kH(x_j,x_j)\mu_j^{N-2}-\sum_{i,j=1;i\not=j}^kG(x_i,x_j)\mu_j^{\frac{N-2}{2}}\mu_i^{\frac{N-2}{2}}\right)-\sum_{j=1}^k\frac{B_N}{2}f_N(\mu_j),
\end{eqnarray*}
where
\begin{eqnarray}\label{Green}
G(x, y)=\gamma_N\left(\frac{1}{|x-y|^{N-2}}-H(x, y)\right)
\end{eqnarray}
is the Green function of the Laplace operator $-\Delta$ at the boundary $\partial\Omega$, with $\gamma_N=\frac{1}{(N-2)\omega_N}$ and $\omega_N$ the surface area of the unit sphere in $\bbr^N$,
$H(x, y)$ is the regular part of $G(x, y)$, that is, $H(x, y)$ is the unique solution of the following equation
\begin{eqnarray}\label{regular}
\left\{\begin{aligned}
&-\Delta H(x, y)=0,\quad&\mbox{in}\,\, \Omega,\\
&H(x, y)=\frac{1}{|x-y|^{N-2}}, \quad&\mbox{on}\,\, \partial\Omega,
\end{aligned}\right.
\end{eqnarray}
$B_N$ is a constant only depending on the dimension $N$, $k\in\mathbb{N}$, $\{x_j\}\subset\Omega$ and $\{\mu_j\}\subset\mathbb{R}_{+}$ are the number, the locations and the heights of the bubbles, respectively,  and $f_N(t)$ is a function only depending on the dimension $N$ (For $k=1$, the Kirchhoff-Routh function is reduced to the Robin function).  These results provide a rather complete understanding in the existence theory of positive solutions of the Brezis-Nirenberg equation~\eqref{BN}.  We also point out that we shall frequently use the following notation.

\vskip0.06in

{\bf Notation:}\quad We denote the spectrum of $-\Delta$ in $H_0^1(\Omega)$ by
$\sigma(-\Delta)=\{\lambda_i\}_{i\geq1}$.

\vskip0.12in

Besides the existence theory of positive solutions of \eqref{BN}, which is almost completely understood, the existence theory of sign-changing solutions of \eqref{BN} is also established under some assumptions on $N$ and $\lambda$.  To our best knowledge, the first results on the existence theory of sign-changing solutions of \eqref{BN} are obtained by Capozzi, Fortunato and Palmieri \cite{CFP1985} by minimization methods similar to Brezis and Nirenberg \cite{BN} and Bahri and Coron \cite{BC}, which can be stated as follows.
\begin{theorem}\label{Thm-CFP1984}
Let $N\geq4$ and $\lambda>0$.  Then the Brezis-Nirenberg equation~\eqref{BN} has a sign-changing solution, provided
\begin{enumerate}
\item[$(1)$]\quad $N=4$ and $\lambda\not\in\sigma(-\Delta)$.
\item[$(2)$]\quad $N\geq5$ and $\lambda>0$.
\end{enumerate}
\end{theorem}

This existence theory has been continuously improved in the past 40 years.  Cerami, Solimini and Struwe \cite{CSS1986} improved Theorem~\ref{Thm-CFP1984} by proving the existence of least energy sign-changing solutions of \eqref{BN} for $N\geq7$ and $0<\lambda<\lambda_1$ via the method of the Nehari manifold, where we say that $u$ is a least energy sign-changing solution of \eqref{BN} if it is a sign-changing solution which has least energy among all sign-changing solutions of \eqref{BN}.  Clapp and Weth \cite{CW2005} improved Theorem~\ref{Thm-CFP1984} by proving the following multiplicity result via the Lusternik–Schnirelmann theory:
 \begin{enumerate}
\item[$(1)$]\quad \eqref{BN} has $\left[\frac{N+1}{2}\right]$ pairs of sign-changing solutions for $N\geq 4$ and $\lambda\not\in\sigma(-\Delta)$.
\item[$(2)$]\quad \eqref{BN} has $\left[\frac{N+2}{2}\right]$ pairs of sign-changing solutions for $N\geq 4$ and $0<\lambda<\lambda_1$.
\item[$(3)$]\quad \eqref{BN} has $\left[\frac{N+1-m}{2}\right]$ pairs of sign-changing solutions for $N\geq 4$ and $\lambda\in\sigma(-\Delta)$, where $m$ is the multiplicity of $\lambda$ as an eigenvalue of $-\Delta$ in $H_0^1(\Omega)$.
\end{enumerate}
Szulkin, Weth and Willem \cite{SWW2009} improved Theorem~\ref{Thm-CFP1984} by proving the existence of least energy sign-changing solutions of \eqref{BN} via the method of the Pankov-Nehari manifold, provided
\begin{enumerate}
\item[$(1)$]\quad $N=4$ and $\lambda>\lambda_1$ with $\lambda\not\in\sigma(-\Delta)$,
\item[$(2)$]\quad $N\geq5$ and all $\lambda>0$.
\end{enumerate}
Chen, Shioji and Zou \cite{CSZ2012} improved Theorem~\ref{Thm-CFP1984} by proving that \eqref{BN} has $\left[\frac{N+1}{2}\right]$ pairs of least energy sign-changing solutions for $N\geq 5$ and $\lambda\geq\lambda_1$ via improving the estimates in the Lusternik–Schnirelmann theory used in \cite{CW2005}.  Chen, Lin and Zou \cite{CLZ2014} and Tavares, You and Zou \cite{TYZ2022} improved Theorem~\ref{Thm-CFP1984} by proving the existence of least energy sign-changing solutions of \eqref{BN} for $N\geq6$ and $N\geq4$, respectively, with $0<\lambda<\lambda_1$ via constructing new test functions in the method of the Nehari manifold used in \cite{CSS1986}.  There are also many other works devoted to the multiplicity of sign-changing solutions of \eqref{BN} by using the variational methods, which we refer the readers to 
\cite{AGGS2008,CC2003,CFS,CSS1986,CW2005,DS2002,DS2003,FJ,GG,S,SZ2010} and the references therein, where we remark that infinitely many sign-changing solutions of equation~\eqref{BN} have been constructed for all $\lambda>0$ and $N\geq7$ in \cite{DS2002,SZ2010}.  It is worth pointing out that all the above mentioned results are all devoted to $N\geq4$, and a significant result on the existence theory of sign-changing solutions of equation~\eqref{BN} via the variational methods is obtained very recently by Ali and Premoselli \cite{AP2025}, which proved, for the first time, that there exists $\lambda_{i,*}\in[\lambda_i, \lambda_{i+1})$ for every $i\geq1$ such that the $3d$ Brezis-Nirenberg equation~\eqref{BN} has a least energy sign-changing solution for $\lambda\in(\lambda_{i,*}, \lambda_{i+1})$.  Ali and Premoselli's proof is based on the discovery, for the first time, of the positive interval in $[\lambda_i, \lambda_{i+1})$ for every $i\geq1$ such that the positive mass condition holds true in it.  Besides the variational methods, the Lyapunov-Schmidt reduction argument is also used to construct sign-changing solutions of \eqref{BN} in \cite{BMP,CC,IV,IV1,LVWW,LVWW2026,MiPi,MP2010,MRV2024,PV,V2015} and the references therein, where different kinds of sign-changing solutions are obtained under different assumptions on dimension $N$ and parameter $\lambda$.  We also remark that another significant result on the existence theory of sign-changing solutions of equation~\eqref{BN} via the Lyapunov-Schmidt reduction argument is obtained very recently by Sun, Wei and Yang \cite{SWY2025}, where infinitely many sign-changing solutions of the $3d$ Brezis-Nirenberg equation~\eqref{BN} in the unit ball have been constructed for all $\lambda>0$ which solves Brezis' first open question proposed in \cite{B}, which had been open for 40 years.  Sun-Wei-Yang's construction is based on a new idea of using  del Pino-Musso-Pacard-Pistoia's crown solution as building blocks, whose nondegeneracy is proved by Musso and Wei \cite{MW2015}, and inverting it at a zero point via the Kelvin transformation.  The same idea has also been used in \cite{SWY2026} to construct infinitely many sign-changing solutions of the $3d$ Brezis-Nirenberg equation~\eqref{BN} with less symmetry of $\Omega$ for $\lambda>0$ sufficiently small.  We also point out that if $\Omega$ is a unit ball then the existence and multiplicity of sign-changing solutions of \eqref{BN} is also obtained by the ODE methods in \cite{ABP, AY, AGGPV,EGPV,GG, IP} and the references therein.

\subsection{Main results.}
Although various existence and multiplicity results on the sign-changing solutions of \eqref{BN} have been obtained yet in the past 40 years, several problems still remain open.  For example, since the multiplicity of $\lambda_i$ is unknown in general bounded domains (It is well known that $\lambda_1$ is simple in any general bounded domain and $\lambda_i$ are simple for $i\geq2$ in generic bounded domains), it is still not known whether the $4d$ Brezis-Nirenberg equation~\eqref{BN} always has sign-changing solutions for $\lambda=\lambda_i$, which has been open for a long time since \cite{CFS}.  In this paper, we shall answer this open question by proving the following theorem.
\begin{theorem}\label{existence}
Let $N=4$ and $\partial\Omega$ is smooth.  Then the Brezis-Nirenberg equation~\eqref{BN} has a least energy sign-changing solution for all $\lambda>0$.  Moreover, the least energy of sign-changing solutions of \eqref{BN}, denoted by $m_{sg}(\lambda)$, is continuous and strictly decreasing as a function of $\lambda\in(\lambda_i, \lambda_{i+1}]$ for all $i\geq0$ such that
\begin{eqnarray*}
\left\{\aligned
&\lim_{\lambda\to0^+}m_{sg}(\lambda)=\frac{1}{2}S^2,\\
&\lim_{\lambda\to\lambda_i^+}m_{sg}(\lambda)\in\left(0, \frac{1}{4}S^2\right],\quad i\geq1
\endaligned\right.
\quad\text{and}\quad
\left\{\aligned
&\lim_{\lambda\to\lambda_1^-}m_{sg}(\lambda)\in\left(0, \frac{1}{4}S^2\right],\\
&\lim_{\lambda\to\lambda_i^-}m_{sg}(\lambda)=0,\quad i\geq2,
\endaligned\right.
\end{eqnarray*}
where for the sake of simplicity, we denote $\lambda_0=0$.
\end{theorem}

\vskip0.12in

Let us briefly sketch the proof of Theorem~\ref{existence}.  It has been proved by Cerami, Solimini and Struwe \cite{CSS1986}, Chen, Lin and Zou \cite{CLZ2014},  Chen, Shioji and Zou \cite{CSZ2012}, Clapp and Weth \cite{CW2005}, Tavares, You and Zou \cite{TYZ2022}, and Szulkin, Weth and Willem \cite{SWW2009} that the least energy sign-changing solutions of \eqref{BN}, denoted by $u_{\lambda, sg}$, satisfies
\begin{eqnarray*}
E_\lambda(u_{\lambda,sg})=\left\{\aligned
&\inf_{\mathcal{N}_{sg}}E_\lambda(u)<\frac{2}{N}S^{\frac{N}{2}},\quad0<\lambda<\lambda_1,\\
&\inf_{\mathcal{P}_i}E_\lambda(u)<\frac{1}{N}S^{\frac{N}{2}},\quad\lambda_i<\lambda<\lambda_{i+1}\text{ and }i\geq1,
\endaligned\right.
\end{eqnarray*}
where $\lambda_i$ are the eigenvalues of $-\Delta$ in $H_0^1(\Omega)$, 
\begin{eqnarray*}
E_\lambda(u)=\frac{1}{2}\|u\|^2_{H_0^1(\Omega)}-\frac{\lambda}{2}\|u\|^2_{L^2(\Omega)}-\frac{1}{p+1}\|u\|^{p+1}_{L^{p+1}(\Omega)}
\end{eqnarray*}
is the corresponding energy functional of the Brezis-Nirenberg equation~\eqref{BN}, 
\begin{eqnarray*}
\mathcal{N}_{sg}=\left\{u\in H_0^1(\Omega)\backslash\{0\}\mid E'_\lambda(u^{\pm})u^{\pm}=0\text{ and }u^\pm\not=0\right\}
\end{eqnarray*}
and
\begin{eqnarray*}
\mathcal{P}_i=\left\{u\in H_0^1(\Omega)\backslash\{0\}\mid E'_\lambda(u)u=0\quad\text{and}\quad Q_jE'_\lambda(u)=0, 1\leq j\leq i\right\}
\end{eqnarray*}
are its associated sign-changing Nehari manifold and Pankov-Nehari manifold, respectively, with $Q_j$ the projection from $H_0^1(\Omega)$ onto $\Xi_j=\bbr e_{j,1}\oplus\bbr e_{j,2}\oplus\cdots\oplus\bbr e_{j,m_j}$ with
$m_j\in\bbn$ being the multiplicity of $\lambda_{j}$ and $\{e_{j,i}\}$ being its orthogonal system such that $\|e_{j,i}\|_{H^1_0(\Omega)}=1$.  Since it is known \cite{AP2025, SWW2009} that 
\begin{eqnarray}\label{engy}
\frac{1}{N}S^{\frac{N}{2}}\geq\inf_{\mathcal{P}_i}E_{\lambda_i}(u),\quad \forall i\geq1,
\end{eqnarray}
for any sequence $u_{\lambda^j, sg}$ in the set
\begin{eqnarray*}
\mathcal{G}_{sg}=\left\{u_{\lambda, sg}\mid u_{\lambda, sg}\text{ is a least energy sign-changing solution of \eqref{BN} for }\lambda\not\in\sigma(-\Delta)\right\},
\end{eqnarray*}
we have
\begin{eqnarray*}\label{energy-condition}
\sup_{\lambda^j}\left\|u_{\lambda^j, sg}\right\|_{H^1_0(\Omega)}^2\leq S^{\frac{N}{2}}
\end{eqnarray*}
if $\lambda^j\to\lambda_i$ as $j\to\infty$ for some $i\geq1$.  Our purpose is to show that $\left\{u_{\lambda^j, sg}\right\}$ is compact in $H^1_0(\Omega)$.  Thus, the limit, say $u_{\lambda_i, sg}$, must be a least energy sign-changing solution of \eqref{BN} for $\lambda=\lambda_i$.  To prove this, we assume the contrary that $\left\{u_{\lambda^j, sg}\right\}$ is noncompact in $H^1_0(\Omega)$.  Thus, it follows from the well-known Struwe decomposition \cite{Str1984} that there exist $\varrho=\pm1$, $\xi_\ve\in\Omega$ and $\mu_\ve\to0$ such that 
\begin{eqnarray}\label{Struwe0}
\left\|\nabla u_\ve-\varrho\nabla U_{\xi_\ve, \mu_\ve}\right\|_{L^2(\bbr^N)}\to0
\end{eqnarray}
as $\ve=\lambda^j-\lambda_i\to0$ up to a subsequence, where $U_{\xi_\ve, \mu_\ve}(x)$ is the Aubin-Talenti bubble given by \eqref{Talenti} and we denote $u_\ve=u_{\lambda^j, sg}$ for the sake of simplicity.  By \eqref{engy}, the Struwe decomposition~\eqref{Struwe0} is not good enough to solve the compactness issue of $\{u_\ve\}$ in passing to the limit.

\vskip0.12in

To go further, we shall derive a refined expansion of \eqref{Struwe0}.  For this purpose, we go to a general setting by assuming that $u_\ve$ is a solution (sign-constant or sign-changing) of the Brezis-Nirenberg equation~\eqref{BN} such that
\begin{eqnarray}\label{condition}
\sup_{\ve}\left\|\nabla u_\ve\right\|_{L^2(\Omega)}^2\leq S^{\frac{N}{2}}\quad\text{and}\quad u_{\ve}(\xi_{\ve})=\max_{x\in\Omega}u_{\ve}(x)\to+\infty
\end{eqnarray}
as $\ve=\lambda-\overline{\lambda}$ for some $\overline{\lambda}>0$.  Thus, by the well-known Struwe decomposition (\cite{Str1984}) and the standard blow-up arguments (\cite{H1991,R1989}), 
\begin{eqnarray}\label{Struwe}
\left\|u_\ve-\mathcal{W}_{\mu_\ve, \xi_\ve}\right\|_{H_0^1(\Omega)}\to0\quad\text{and}\quad\frac{d(\xi_{\ve},\partial\Omega)}{\mu_\ve}\to+\infty
\end{eqnarray}
as $\ve\to0$, where 
\begin{eqnarray}\label{mu-ve}
\mu_{\ve}=\left(\frac{1}{u_{\ve}(\xi_{\ve})}\right)^{\frac{2}{N-2}}\to0\quad\text{as }\ve\to0
\end{eqnarray}
and $\mathcal{W}_{\mu, \xi}(x)$ is the projection of the Aubin-Talenti bubble $U_{\mu, \xi}(x)$ into $H^1_0(\Omega)$, that is, $\mathcal{W}_{\mu, \xi}(x)$ is the unique solution of the following equation
\begin{eqnarray}\label{Projection}
\left\{\aligned
&-\Delta w=\mathcal{U}_{\mu, \xi}^{p},\quad&\mbox{in}\,\, \Omega,\\
&w=0,\quad&\mbox{on}\,\, \partial\Omega.
\endaligned\right.
\end{eqnarray}
As usual \cite{BC, CKW2025, E2004, R1989}, we solve the following variational problem
\begin{equation}\label{orthogonal-pro}
\inf_{(\mu, \xi)\in (0,\infty)\times \Omega}\left\|u_\ve-\mathcal{W}_{\mu,\xi}\right\|^2_{H^1_0(\Omega)}
\end{equation}
to orthogonally decompose $u_\ve$.  It is easy to prove that \eqref{orthogonal-pro} is achieved by $(\mu_{\ve,*}, \xi_{\ve,*})\in(0,\infty)\times \Omega$ such that $\frac{d(\xi_{\ve,*},\partial\Omega)}{\mu_{\ve,*}}\to+\infty$ as $\ve\to0$.  Under the condition~\eqref{condition}, by the standard blow-up arguments (\cite{H1991,R1989}), it is also easy to see that $\mu_{\ve,*}\sim\mu_{\ve}$ and $\frac{\left|\xi_{\ve,*}-\xi_{\ve}\right|}{\mu_{\ve,*}}\lesssim1$ as $\ve\to0$.  Thus, without loss of generality, we may assume that $(\mu_{\ve,*}, \xi_{\ve,*})=(\mu_{\ve}, \xi_{\ve})$.  We introduce some necessary notations before the further discussions. For every $\mu>0$ and $\xi \in \Omega$, we denote
\begin{eqnarray}\label{kernel}
\Phi_{\mu,\xi}^i=\left\{\aligned
&\mu \frac{\partial U_{\mu, \xi}}{\partial \mu},\quad i=0,\\
&\mu \frac{\partial U_{\mu, \xi}}{\partial \xi_i},\quad 1\leq i\leq N.
\endaligned\right.
\end{eqnarray}
Then it is well known (\cite{BE1991, R}) that all the bounded solutions of the following linear equation,
\begin{eqnarray}\label{linear}
-\Delta \Phi_{\mu,\xi}=pU_{\mu, \xi}^{p-1}\Phi_{\mu,\xi}\quad\text{in }\mathbb{R}^N,
\end{eqnarray}
is spanned by $\{\Phi_{\mu,\xi}^i\}_{0\leq i\leq N}$, where for the sake of simplicity, we denote $p=2^*-1$.  As usual, we denote the projection of $\Phi_{\mu,\xi}^i$ onto $H_0^1(\Omega)$ by $\mathcal{Z}_{\mu,\xi}^i$, that is, $\mathcal{Z}_{\mu,\xi}^i$ is the unique solution of the following linear equation
\begin{eqnarray}\label{projkernel}
\left\{\aligned&-\Delta w=pU_{\mu, \xi}^{p-1}\Phi_{\mu,\xi}^i,\quad\text{in }\Omega,\\
&w=0,\quad\text{on }\partial\Omega.
\endaligned\right.
\end{eqnarray}
Now, let 
\begin{eqnarray}\label{v}
v_\ve=u_\ve-\mathcal{W}_{\mu_\ve, \xi_\ve}.
\end{eqnarray}
Then by \eqref{projkernel} and \eqref{orthogonal-pro}, 
\begin{eqnarray}\label{orthogonality}
\langle v_\ve, \mathcal{Z}_{\mu_\ve, \xi_\ve}^i\rangle=\int_{\Omega}U_{\mu_\ve, \xi_\ve}^{p-1}\Phi_{\mu_\ve,\xi_\ve}^i v_\ve dx=0
\end{eqnarray}
for all $0\leq i\leq N$.  To refine the expansion~\eqref{Struwe}, the first step is to 
use the inverse reduction argument introduced by Wang and Wei \cite{WW2019, WW2019-2} to improve \eqref{Struwe} up to the next order term.  The main finding in our improvement is that the Sobolev norm of $v_\ve$ is much larger than the vanishing rate of the Aubin-Talenti bubble $U_{\mu_\ve, \xi_\ve}$ away from the concentration point $\xi_\ve$ as $\ve\to0$, that is, 
\begin{eqnarray*}
\lim_{\ve\to0}\frac{\|v_\ve\|_{H_0^1(\Omega)}}{\mu_\ve^{\frac{N-2}{2}}}=+\infty
\end{eqnarray*}
for $\overline{\lambda}>0$.  This finding suggests that the limit equation of \eqref{BN} as $\ve\to0$, which is away from the concentration point $\xi_\ve$, is the eigenvalue equation.  Thus, the limit value $\overline{\lambda}$, if it is positive, must be an eigenvalue of $-\Delta$ in $H_0^1(\Omega)$, which generates additional kernels in the analysis of the compactness of $\{u_\ve\}$ as $\ve\to0$.
To well control the remaining term $v_\ve$ in any reasonable sense, the second step is to remove projections of $v_\ve$ on all kernels such that the new remaining term, say $v_{\ve}^*$, is orthogonal to all kernels in $H_0^1(\Omega)$, and then expand the orthogonal conditions of $v_{\ve}^*$ up to the second order terms.  We remark that for $N\geq5$, the precise expansion of the orthogonal conditions only needs the $H_0^1$-norm estimate of $v_{\ve}^*$ while, for 
$N=4$, the $H_0^1$-norm estimate of $v_{\ve}^*$ must be replaced by the $L^\infty$-norm estimate of $v_{\ve}^*$.  Our main observation in this refined expansion is that, if $\ve\to0^+$ then the refined expansion is nondegenerate in the sense that there is no cancelation up to the second order terms.  To state our classification results, we introduce the following definition
\begin{definition}
We say that $x_0\in\mathbb{R}^N$ is a singular point of the function $v(x)\in C^1\left(\mathbb{R}^N\right)$ if $v(x_0)=0$ and $\nabla v(x_0)=0$.
\end{definition}
Now, our first classification result can be stated as follows.
\begin{theorem}\label{zero weak limit+}
Let $N\geq4$.  Suppose that $u_\ve$ is a solution of the Brezis-Nirenberg equation~\eqref{BN} satisfying the condition~\eqref{condition} for $\overline{\lambda}>0$ as $\ve=\lambda-\overline{\lambda}\to0^+$.  Then we must have $N=4,5$, $\overline{\lambda}=\lambda_k$ for some $k\geq1$ and
\begin{eqnarray*}
u_\ve=\mathcal{W}_{\mu_\ve,\xi_\ve}+\sum_{i=1}^{m_{k}}\tau_{k,i,\ve}e_{k,i}+\varphi_{\ve},
\end{eqnarray*}
where $\sum_{i=1}^{m_{k}}\tau_{k,i,\ve}e_{k,i}(\xi_\ve)<0$, $\xi_\ve\to\xi_0\in\overline{\Omega}$, $\mu_{\ve}\to0$, $\beta_\ve^*=\max_{1\leq i\leq m_{k}}|\tau_{k,i,\ve}|\to0$ and 
\begin{eqnarray*}\label{equa0015}
\frac{\|\varphi_{\ve}\|_{H_0^1(\Omega)}}{\max_{1\leq i\leq m_{k}}|\tau_{k,i,\ve}|}\to0
\end{eqnarray*}
as $\ve\to0^+$.  Moreover, 
\begin{enumerate}
\item[$(a)$]\quad if $N=5$ then $\xi_0$ is a singular point of $\mathcal{E}_0^*(x)=\sum_{i=1}^{m_{k}}t_{k,i,0}e_{k,i}(x)$
with $t_{k,i,0}=\lim_{\ve\to0}\frac{\tau_{k,i,\ve}}{\beta_{\ve}^*}$.
\item[$(b)$]\quad If $N=4$ then either $\xi_0$ is a singular point of $\mathcal{E}_0^*(x)=\sum_{i=1}^{m_{\kappa}}t_{k,i,0}e_{k,i}(x)$ or
\begin{enumerate}
\item[$(b_1)$]\quad $\xi_0\in\Omega$, $\mathcal{E}_0^*\left(\xi_0\right)\not=0$ and $\nabla \mathcal{E}_0^*\left(\xi_0\right)=0$ with
\begin{eqnarray*}
\beta_\ve\sim\frac{1}{\ve}e^{-\frac{1}{\ve}}\quad\text{and}\quad\mu_\ve\sim e^{-\frac{1}{\ve}}
\end{eqnarray*}
as $\ve\to0^+$.
\item[$(b_2)$]\quad $\xi_0\in\partial\Omega$ and $\partial_{\nu}\mathcal{E}_0^*\left(\xi_0\right)\not=0$ with
\begin{eqnarray*}
\mu_\ve\sim e^{-\frac{1}{\sqrt{\ve}}},\quad \beta_\ve\sim \frac{1}{\ve^{\frac{3}{4}}}e^{-\frac{1}{\sqrt{\ve}}}\quad\text{and}\quad d(\xi_\ve, \partial\Omega)\sim\ve^{\frac{1}{4}}
\end{eqnarray*}
as $\ve\to0^+$, where $\partial_{\nu}\mathcal{E}_0^*\left(\xi_0\right)$ is the normal derivative of $\mathcal{E}_0^*$ at $\xi_0$.
\end{enumerate}
All the asymptotic behaviors described in $(b_1)$ and $(b_2)$ can be precisely computed up to the leading order term.
\end{enumerate}
\end{theorem}

\begin{remark}
The sign-changing solution described in $(b_1)$ of Theorem~\ref{zero weak limit+}, which will concentrate at inner critical points of the eigenfunctions as $\lambda$ close to the related eigenvalues from above, has been constructed in \cite{IV, LVWW}.  Moreover, $(b_2)$ of Theorem~\ref{zero weak limit+} suggests that the $4d$ Brezis-Nirenberg equation may have sign-changing solutions, which will concentrate at the boundary $\partial\Omega$ as $\lambda$ close to the eigenvalue $\lambda_k$ for some $k\geq1$ from above, if there exists $\pmb{t}_0$ such that $\partial\Omega_{\pmb{t}_0, k, -}\cap\partial\Omega\not=\emptyset$, where $\pmb{t}_0=(t_{1,0}, t_{2,0}, \cdots, t_{m_k,0})$ and
\begin{eqnarray*}
\Omega_{\pmb{t}_0, k, -}=\left\{x\in\Omega\mid \sum_{i=1}^{m_{\kappa}}t_{i,0}e_{k,i}(x)<0\right\}.
\end{eqnarray*}
\end{remark}

\vskip0.12in

If $\ve\to0^-$ then our refined expansion may degenerate in the sense that there is a cancelation in the leading order terms if
\begin{eqnarray*}
\|\mathcal{E}_0^*\|_{L^{p+1}(\Omega)}^{p+1}+\|\mathcal{E}_0^*\|_{L^{2}(\Omega)}^{2}\lim_{\ve\to0^-}\frac{\ve}{\left(\beta_\ve^*\right)^{p-1}}=0,
\end{eqnarray*}  
where $\mathcal{E}_0^*(x)=\sum_{i=1}^{m_{k}}t_{k,i,0}e_{k,i}(x)$
with $t_{k,i,0}=\lim_{\ve\to0}\frac{\tau_{k,i,\ve}}{\beta_{\ve}^*}$ and $\beta_\ve^*=\max_{1\leq i\leq m_{k}}|\tau_{k,i,\ve}|\to0$.
Thus, under the current refined expansion of the orthogonal conditions of $v_\ve^*$, we can obtain a partial classification of the Struwe decomposition under the condition~\eqref{condition} as $\ve\to0^-$ if there is no cancelations in passing to the limit.
\begin{theorem}\label{zero weak limit-1}
Let $N\geq4$ and $u_\ve$ is a solution of the Brezis-Nirenberg equation~\eqref{BN} satisfying the condition~\eqref{condition} for $\overline{\lambda}>0$ as $\ve=\lambda-\overline{\lambda}\to0^-$.  If
\begin{eqnarray*}
\|\mathcal{E}_0^*\|_{L^{p+1}(\Omega)}^{p+1}+\|\mathcal{E}_0^*\|_{L^{2}(\Omega)}^{2}\lim_{\ve\to0^-}\frac{\ve}{\left(\beta_\ve^*\right)^{p-1}}\not=0,
\end{eqnarray*}
then we must have $N=4$, $\overline{\lambda}=\lambda_k$ for some $k\geq1$ and 
\begin{eqnarray*}
u_\ve=\mathcal{W}_{\mu_\ve, \xi_\ve}+\sum_{i=1}^{m_{k}}\tau_{k,i,\ve}e_{k,i}+\varphi_{\ve},
\end{eqnarray*}
where $\xi_\ve\to\xi_0\in\partial\Omega$, $\mu_{\ve}\to0$, $\beta_\ve^*=\max_{1\leq i\leq m_{k}}|\tau_{k,i,\ve}|\to0$ and 
\begin{eqnarray*}\label{equa00151}
\frac{\|\varphi_{\ve}\|_{H_0^1(\Omega)}}{\max_{1\leq i\leq m_{k}}|\tau_{k,i,\ve}|}\to0
\end{eqnarray*}
as $\ve\to0^-$.  Moreover, $\xi_0\in\partial\Omega$ is a singular point of $\mathcal{E}_0^*(x)=\sum_{i=1}^{m_{k}}t_{k,i,0}e_{k,i}(x)$
with $t_{k,i,0}=\lim_{\ve\to0}\frac{\tau_{k,i,\ve}}{\beta_{\ve}^*}$.
\end{theorem}

In the degenerate case
\begin{eqnarray*}
\|\mathcal{E}_0^*\|_{L^{p+1}(\Omega)}^{p+1}+\|\mathcal{E}_0^*\|_{L^{2}(\Omega)}^{2}\lim_{\ve\to0^-}\frac{\ve}{\left(\beta_\ve^*\right)^{p-1}}=0,
\end{eqnarray*}
we need further refine the expansion of the orthogonal conditions of $v_\ve^*$.  For this purpose, we construct a new perturbation near the limit profile to remove this cancelation (see Proposition~\ref{new-decomp} for more details).  It can be achieved since the degeneracy is generated by a nonlinear term $\|\mathcal{E}_0^*\|_{L^{p+1}(\Omega)}^{p+1}$ and a linear term $\|\mathcal{E}_0^*\|_{L^{2}(\Omega)}^{2}$.  Thus, it contains a secondary nondegenerate condition.  After doing this, we shall 
orthogonally decompose $v_\ve^*$ once more to remove the additional projections on all kernels in $H_0^1(\Omega)$, according to the construction of the new perturbation.  Then we expand the orthogonal conditions of another new remaining term, say $v_\ve^{**}$, up to the third order terms to complete the classification of the Struwe decomposition under the condition~\eqref{condition} as $\ve\to0^-$, which can be stated as follows.
\begin{theorem}\label{zero weak limit-2}
Let $N\geq4$ and $u_\ve$ is a solution of the Brezis-Nirenberg equation~\eqref{BN} satisfying the condition~\eqref{condition} for $\overline{\lambda}>0$ as $\ve=\lambda-\overline{\lambda}\to0^-$.  If
\begin{eqnarray*}
\|\mathcal{E}_0^*\|_{L^{p+1}(\Omega)}^{p+1}+\|\mathcal{E}_0^*\|_{L^{2}(\Omega)}^{2}\lim_{\ve\to0^-}\frac{\ve}{\left(\beta_\ve^*\right)^{p-1}}=0,
\end{eqnarray*}
then $\overline{\lambda}=\lambda_k$ for some $k\geq1$ and 
\begin{eqnarray*}
u_\ve=\mathcal{W}_{\mu_\ve, \xi_\ve}+\sum_{i=1}^{m_{k}}\tau_{k,i,\ve}e_{k,i}+\varphi_{\ve},
\end{eqnarray*}
where $\mathcal{E}_0^*(x)=\sum_{i=1}^{m_{k}}t_{k,i,0}e_{k,i}(x)$
with $t_{k,i,0}=\lim_{\ve\to0}\frac{\tau_{k,i,\ve}}{\beta_{\ve}^*}$, $\xi_\ve\to\xi_0\in\overline{\Omega}$, $\mu_{\ve}\to0$, $\beta_\ve^*=\max_{1\leq i\leq m_{k}}|\tau_{k,i,\ve}|\to0$ and 
\begin{eqnarray*}\label{equa00151}
\frac{\|\varphi_{\ve}\|_{H_0^1(\Omega)}}{\max_{1\leq i\leq m_{k}}|\tau_{k,i,\ve}|}\to0
\end{eqnarray*}
as $\ve\to0^-$.  Moreover, either
\begin{enumerate}
\item[$(a)$]\quad the matrix
\begin{eqnarray}\label{matrix}
\mathcal{M}=\left(\int_{\Omega}\left(p\left|\widetilde{\mathcal{E}}_{\ve,0}^*\right|^{p-1}-1\right)e_{k,l}e_{k,t}dx\right)_{m_k\times m_k}
\end{eqnarray}
is singular, where $\widetilde{\mathcal{E}}_{\ve,0}^*=\sum_{l=1}^{m_k}\left(\frac{\|\mathcal{E}_0^*\|_{L^{2}(\Omega)}^{2}}{\|\mathcal{E}_0^*\|_{L^{p+1}(\Omega)}^{p+1}}\right)^{\frac{1}{p-1}}t_{k,l,0} e_{k,l}$, or
\item[$(b)$]\quad $N=4,5$ and
\begin{enumerate}
\item[$(b_1)$]\quad if $N=4$ then $\xi_0$ is a singular point of $\mathcal{E}_0^*(x)$;
\item[$(b_2)$]\quad if $N=5$ then either $\xi_0$ is a singular point of $\mathcal{E}_0^*(x)$ or $\xi_0\in\Omega$, $\mathcal{E}_0^*\left(\xi_0\right)\not=0$ and $\nabla \mathcal{E}_0^*\left(\xi_0\right)=0$ with
\begin{eqnarray*}
\beta_\ve^*\sim|\ve|^{\frac{3}{4}}\quad\text{and}\quad\mu_\ve\sim|\ve|^{\frac{3}{2}}.
\end{eqnarray*}
as $\ve\to0^-$.  The asymptotic behaviors described above can be precisely computed up to the leading order term.
\end{enumerate}
\end{enumerate}
\end{theorem}

\vskip0.12in

\begin{remark}
Since we have renormalized $\{e_{j,i}\}$ such that $\|e_{j,i}\|_{H^1_0(\Omega)}=1$, the matrix~\eqref{matrix} is just the Hessian of the function
\begin{eqnarray*}
\mathcal{H}(\pmb{t}_{k})=\frac{1}{2}\left\|\sum_{n=1}^{m_k}t_{k,n}e_{k,n}\right\|_{L^2(\Omega)}^{2}-\frac{1}{2^*}\left\|\sum_{n=1}^{m_k}t_{k,n}e_{k,n}\right\|_{L^{p+1}(\Omega)}^{p+1}+\frac{m_k}{2}
\end{eqnarray*}
at the point
\begin{eqnarray*}
\hat{\pmb{t}}_{k,0}=\left(\left(\frac{\|\mathcal{E}_0^*\|_{L^{2}(\Omega)}^{2}}{\|\mathcal{E}_0^*\|_{L^{p+1}(\Omega)}^{p+1}}\right)^{\frac{1}{p-1}}t_{k,1,0}, \left(\frac{\|\mathcal{E}_0^*\|_{L^{2}(\Omega)}^{2}}{\|\mathcal{E}_0^*\|_{L^{p+1}(\Omega)}^{p+1}}\right)^{\frac{1}{p-1}}t_{k,2,0}, \cdots, \left(\frac{\|\mathcal{E}_0^*\|_{L^{2}(\Omega)}^{2}}{\|\mathcal{E}_0^*\|_{L^{p+1}(\Omega)}^{p+1}}\right)^{\frac{1}{p-1}}t_{k,m_k,0}\right).
\end{eqnarray*}
Moreover, if $k=1$ or $k\geq2$ and $\Omega$ is generic in the sense of Uhlenbeck \cite{U1976}, then $\lambda_k$ is simple.  In these cases, the matrix~\eqref{matrix} is reduced to $(p-1)\|\mathcal{E}_0^*\|_{L^{2}(\Omega)}^{2}$ which is automatically regular.  Thus, the classification result stated in Theorem~\ref{zero weak limit-2} holds true if $\lambda_k$ is simple.  Such kind of blow-up sign-changing solutions have been constructed in \cite{IV, LVWW} for $N=5$.
\end{remark}

\vskip0.12in

Theorems~\ref{zero weak limit+}, \ref{zero weak limit-1} and \ref{zero weak limit-2} give {\bf a first complete classification} of the Struwe decomposition under the condition~\eqref{condition}, which we believe will be useful in understanding many other critical equations.  Going back to the proof of the compactness of $\{u_\ve\}$, since the blow-up phenomenon is classified, the blow-up phenomenon of $\{u_\ve\}$ for $\ve=\lambda-\lambda_1\to0^-$ will contradict the regularity of the principal eigenfunction (see the Step~1 in the proof of Theorem~\ref{existence} for more details).  Thus, $4d$ Brezis-Nirenberg equation~\eqref{BN} has a least energy sign-changing solution for $\lambda=\lambda_1$.  For $\lambda=\lambda_k$ with $k\geq2$, we consider the blow-up phenomenon of $\{u_\ve\}$ for $\ve=\lambda-\lambda_k\to0^+$.  Then the classification of the blow-up phenomenon implies that
\begin{eqnarray*}
E_{\lambda}\left(u_{\ve}\right)\geq\frac{1}{4}S^2+o(\mu_\ve^2|\log\mu_\ve|)
\end{eqnarray*}
as $\ve\to0^+$, where $\mu_\ve$ is the blow-up rate of $u_{\ve}$.  However, since $u_\ve$ is a least energy sign-changing solution, it has a variational characterization
\begin{eqnarray*}
E_\lambda(u_{\ve})=\inf_{v\in\mathbb{Y}_k\atop v\not=0}\max_{w\in\mathbb{X}_k\atop t>0}E_\lambda(tv+w)
\end{eqnarray*}
where $\mathbb{X}_k=\oplus_{j=1}^{k}\Xi_j$ and $\mathbb{Y}_k=\oplus_{j=k+1}^{\infty}\Xi_j$.
We then construct a test function of $E_\lambda(u_{\ve})$, based on the same blow-up rate $\mu_\ve$ but new location $\eta_\ve$ such that $\sum_{l=1}^{m_k}s^*_{k,l,\ve}e_{k,l}(\eta_\ve)=0$ and $d(\eta_\ve, \partial\Omega)\gtrsim1$, by this variational characterization, which leads to
\begin{eqnarray*}
E_{\lambda}\left(u_{\ve}\right)-\frac{1}{4}S^2\lesssim-\mu_\ve^2|\log\mu_\ve|.
\end{eqnarray*}
Here, $\sum_{l=1}^{m_k}s^*_{k,l,\ve}e_{k,l}(x)$ is the next order of the constructed test function given in \eqref{test-10}.
We remark that this construction only works for $k\geq2$ since we need the property that $\sum_{l=1}^{m_k}s_{k,l}e_{k,l}$ is sign-changing in $\Omega$ for any $\pmb{s}_{k}\not=\pmb{0}\in\mathbb{R}^{m_k}$.  Thus, the blow-up phenomenon is impossible according to the energy expansion up to the second order term in passing to the limit.  It follows that $4d$ Brezis-Nirenberg equation~\eqref{BN} has a least energy sign-changing solution for $\lambda=\lambda_k$ with all $k\geq2$.

\vskip0.12in

We also remark that, as pointed out above, Theorem~\ref{existence} {\bf completes the existence theory} of the Brezis-Nirenberg equation~\eqref{BN} for $N\geq4$ in the sense that it asserts that  the $4d$ Brezis-Nirenberg equation~\eqref{BN} always has sign-changing solutions for $\lambda=\lambda_i$ with $i\geq2$, which has been open for a long time since \cite{CFS}.  For $\lambda=\lambda_1$, the existence theory of the $4d$ Brezis-Nirenberg equation~\eqref{BN} has already been established by Clapp and Weth \cite{CW2005}.  The novelty of Theorem~\ref{existence} is that it asserts that  one of the nontrivial solutions of the $4d$ Brezis-Nirenberg equation~\eqref{BN} for $\lambda=\lambda_1$ must be a least energy one.

\subsection{Further remarks}
As pointed out above, the crucial point in proving Theorem~\ref{existence} is the classification of the one-bubble blow-up phenomenon of solutions of \eqref{BN}.  Such kind of studies was implicitly initiated by Brezis' discussion on the open question~7 in \cite{B1986} and Han made the first attempt in this direction.  He proved in \cite{H1991} by the blow-up analysis and the local Pohozaev identity that as $\lambda\to0$ for $N\geq4$, the positive solution of \eqref{GS}, say $u_\lambda$, will concentrate and blow up whose behavior likes $U_{\mu_\lambda, \xi_\lambda}$ such that $\mu_{\lambda}\to0$ with a precise dependence on $\lambda$ and $\xi_\lambda\to\xi_0$ with $\xi_0$ being a global minimal point of the Robin function of $\Omega$.  The same result is obtained by Rey in \cite{R1989} via orthogonally decomposing the remaining term.  The case $\lambda\to\lambda_*$ in $N=3$ is much more complicated and the same result of the positive solution of \eqref{GS} is established by Druet \cite{D2002} more than ten years later.  The proof is much more involved, which required by further adapting the iteration technique for the decay rate away from the concentration point and the harmonic analysis near the concentration point.  Esposito gave a simple proof of Druet's result in \cite{E2004} by orthogonally decomposing the remaining term and estimating the quadratic form of the functional in \eqref{GS} at its positive solution.  By further developing the idea in \cite{D2002} with the application of the Harnack inequality in the neck region of each blow-up point, Druet and Laurain proved in \cite{DL2010} the stability of the Pohozaev obstruction of \eqref{BN} in a more general setting, which asserted that blow-up points of the positive solutions of \eqref{BN} with finite energy are all isolated.  Recently, in the similar spirit, Konig and Laurain \cite{KL2022-1,KL2024} derived the precise profile of all positive solutions of \eqref{BN} as $\lambda\to0$ for $N\geq4$ and $\lambda\to\lambda_*$ for $N=3$ by further applying the refined analysis.  Moreover, Cao, Luo and Peng \cite{CLP2021} proved, by applying the local Pohozaev identity, the local uniqueness of blow-up positive solutions as $\lambda\to0$ for $N\geq6$, if all concentration points are nondegenerate critical points of the Robin function of $\Omega$.  Besides, in the recent papers \cite{FKK2020,FKK2021,FKK2024}, by nice and surprisingly elegant development  of the estimation on the quadratic form of the functional in \eqref{GS} in \cite{E2004} with very little regularity assumptions, Frank, Konig and Kovarik improved the results of Druet \cite{D2002}, Esposito \cite{E2004}, Han \cite{H1991} and Rey \cite{R1989} in the following two fronts:  
\begin{enumerate}
\item[$(1)$]\quad The same conclusions hold true for minimizing sequences of \eqref{GS} satisfying a mild additional condition.
\item[$(2)$]\quad The expansion of minimizing sequences could be precisely up to the second order term.
\end{enumerate}
These results give a rather complete understanding of the classification of positively blow-up solutions of \eqref{BN}.  It is also worth pointing out that the above a priori analysis on the profile of nontrivial solutions of \eqref{BN} as $\lambda\to0$ for $N\geq4$ and $\lambda\to\lambda_*$ for $N=3$ are all based on the following crucial fact: All the positive solutions of the Yamabe equation in $\mathbb{R}^N$, that is, the limit equation in the blow-up analysis of \eqref{BN}, are the Aubin-Talenti bubbles $U_{\mu,\xi}$ (proved by Caffarelli, Gidas and Spruck \cite{CGS}), moreover, the Aubin-Talenti bubbles are all nondegenerate in the sense that the only bounded solutions to the linearization of the Yamabe equation in $\mathbb{R}^N$ at $U_{\mu,\xi}$ are spanned by $\partial_{\xi_i}U_{\mu,\xi}$, $1\leq i\leq N$ and $\partial_{\mu}U_{\mu,\xi}$ (proved by Rey \cite{R} and Bianchi and Egnell \cite{BE1991}).

\vskip0.12in

The classification on the profile of sign-changing solutions of \eqref{BN}, as the parameter $\lambda$ close to a concentration value, is also considered in the literatures.  Ben Ayed, Mehdi and Pacella made the first attempt in this direction in \cite{BEP2006-2, BEP2006-1} by considering the least energy sign-changing solutions for $N=3$, $\lambda\to\overline{\lambda}^*$ and $N\geq4$, $\lambda\to0$, respectively, where $\overline{\lambda}^*\geq\lambda_*$ is defined by
\begin{eqnarray*}
\overline{\lambda}^*=\inf\left\{\lambda>0\mid \text{\eqref{BN} has a least energy sign-changing solution for }N=3\right\}.
\end{eqnarray*}
They proved, by further assuming that the blow-up rates of the positive and negative part of $u_\lambda^*$ are the same in the case $N\geq4$, that the least energy sign-changing solution $u_\lambda^*$ will concentrate and blow up at two different points of $\Omega$ as $\lambda\to\overline{\lambda}^*$ in the case $N=3$ and $\lambda\to0$ in the case $N\geq4$.  It is worth pointing out that it is not yet known what is the value $\overline{\lambda}^*$.  
Iacopetti \cite{Iac} and Iacopetti and Pacella \cite{IP2015} have observed that 
the further assumptions in \cite{BEP2006-1} are not necessary for $N\geq7$ and seems to be necessary for $4\leq N\leq 6$.  They proved in \cite{Iac} that least energy sign-changing solution $u_\lambda^*$ will concentrate and blow up at two different points of $\Omega$ as $\lambda\to0$ with different blow-up rates of the positive and negative part  such that their limit profile is a tower of two bubbles for $N\geq7$ and in \cite{IP2015} that this kind of solutions do not exist for $4\leq N\leq 6$.  We remark that since the profile of the blow-up sign-changing solutions of \eqref{BN} may be also sign-changing, and sign-changing solutions of the Yamabe equation in $\mathbb{R}^N$, that is, the limit equation in the blow-up analysis of \eqref{BN}, are far away from well understood, see, for example, the construction of  del Pino, Musso, Pacard and Pistoia \cite{dPMPP2011}, Ding \cite{D1986}, Medina, Musso and Wei \cite{MMW2019} and Medina and Musso \cite{MM2021} on infinitely many sign-changing solutions of the Yamabe equation in $\mathbb{R}^N$, the profile of blow-up sign-changing solutions of \eqref{BN} are also far away from well understood, except the radial case.  In the radial setting, that is, $\Omega$ is a unit ball, the profile of blow-up sign-changing solutions of \eqref{BN} are completely understood by making a detailed ODE analysis in \cite{ABP, AY, AGGPV, GG, EGPV, IP}, which asserts that 
\begin{enumerate}
\item[$(1)$]\quad $\left(\frac{2m-1}{2}\pi\right)^2$ are all concentration values of \eqref{BN} for $N=3$.
\item[$(2)$]\quad Eigenvalues $\lambda_i$ are all concentration values of \eqref{BN} for $N=4,5$.
\item[$(3)$]\quad All concentration values of \eqref{BN} for $N=6$ are $\overline{\lambda}_m$ such that \eqref{BN} for $N=6$ and $\lambda=\overline{\lambda}_m$ has a unique radial solution which has $m-1$ nodal domains.
\item[$(4)$]\quad $\lambda=0$ is the only concentration value of \eqref{BN} for $N\geq7$.
\end{enumerate}
Moreover, the only concentration point in the radial setting is the origin with one standard Aubin-Talenti bubble for $3\leq N\leq 6$ and arbitrary number of standard Aubin-Talenti bubbles in an alternative way on the sign for $N\geq7$.
Thus, to our best knowledge, Theorems~\ref{zero weak limit+}, \ref{zero weak limit-1} and \ref{zero weak limit-2} are the first results on a complete classification of one-bubble sign-changing solutions of \eqref{BN} in a general bounded domain, which may be of independent interest.  Since similar blow-up analysis of sign-changing solutions is carried out for the Schr\"odinger equation on Riemannian manifolds, see, for example, 
\cite{DMW2019,MPV2009,PV2013,P2024,PR2025,PT2018,PV2022,PV2022-1,RV2013,RV2015,RV2023} and the references therein, we believe that the refinement of the inverse reduction argument developed in this paper will also be useful in the Riemannian setting to deeper understand the compactness of sign-changing solutions of the Schr\"odinger equation on Riemannian manifolds.

\section{Preliminaries}
We recall that we have denoted the projection of $U_{\mu,\xi}$ onto $H_0^1(\Omega)$ by $\mathcal{W}_{\mu,\xi}$.  Then it is also well known (\cite{MRV2024,R}) that 
\begin{equation}\label{exp-ker}
\mathcal{Z}_{\mu,\xi}^i(x)=\left\{\aligned &\Phi_{\mu,\xi}^0(x)-\frac{\alpha_N(N-2)}{2}\mu^{\frac{N-2}{2}}H(x,\xi)+O\left(\frac{\mu^{\frac{N+2}{2}}}{d(\xi,\partial \Omega)^{N}}\right),\quad\text{for }i=0,\\
&\Phi_{\mu,\xi}^i(x)-\alpha_N \mu^{\frac{N}{2}}\frac{\partial H(x,\xi)}{\partial \xi_i }+O\left(\frac{\mu^{\frac{N+4}{2}}}{d(\xi,\partial \Omega)^{N+1}}\right), \quad\text{for }i=1,\cdots,N\endaligned\right.
\end{equation}
and
\begin{equation}\label{exp-bub}
\mathcal{W}_{\mu,\xi}(x)= U_{\mu,\xi}(x)-\alpha_N\mu^{\frac{N-2}{2}}H(x,\xi)+O\left(\frac{\mu^{\frac{N+2}{2}}}{d(\xi,\partial \Omega)^{N}}\right),
\end{equation}
where $d(\xi,\partial \Omega)$ is the distance of $\xi$ and $\partial\Omega$, and $H(x,\xi)$ is the regular part of the Green function given by \eqref{regular}.  Clearly, by applying the maximum principle to \eqref{Projection} and \eqref{linear}, 
\begin{eqnarray}\label{comp}
0<\mathcal{W}_{\mu,\xi}\leq U_{\mu,\xi}\quad\text{and}\quad \left|\mathcal{Z}_{\mu,\xi}^0\right|\lesssim \mathcal{W}_{\mu,\xi}
\end{eqnarray}
in $\Omega$.
Moreover, it is also known (\cite{LVWW}) that
\begin{eqnarray*}\label{ker-gre}
\mathcal{Z}_{\mu,\xi}^i(x)=\left\{\aligned \frac{(N-2)\alpha_N}{2\gamma_N} \mu^{\frac{N-2}{2}} G(x,\xi)+\omega_{\mu,\xi}^0,\quad\text{if } i=0,\\
\frac{\alpha_N}{\gamma_N} \mu^{\frac{N}{2}} \frac{\partial G(x,\xi) }{\partial \xi_i}+\omega_{\mu,\xi}^i,\quad \text{if } 1\leq i\leq N\endaligned \right.
\end{eqnarray*}
and
\begin{eqnarray*}\label{bub-gre}
\mathcal{W}_{\mu,\xi}(x)=\frac{\alpha_N}{\gamma_N} \mu^{\frac{N-2}{2}} G(x,\xi)+\omega^*_{\mu,\xi},
\end{eqnarray*}
where $G(x,\xi)$ is the Green function of $\Omega$ given by \eqref{Green},
\begin{eqnarray*}
|\omega_{\mu,\xi}^0|+|\omega_{\mu,\xi}^*| \lesssim \frac{\mu^{\frac{N+2}{2}}}{d(\xi,\partial \Omega)^{N}}+\left\{\aligned &\mu^{\frac{N-2}{2}} |x-\xi|^{-(N-2)},\quad \text{if }|x-\xi| \lesssim \mu,\\
&\mu^{\frac{-(N-2)}{2}},\quad\text{if } |x-\xi|\sim \mu,\\
&\mu^{\frac{N+2}{2}}|x-\xi|^{-N},\quad\text{if }|x-\xi| \gtrsim \mu
\endaligned\right.
\end{eqnarray*}
and
\begin{eqnarray*}
|\omega_{\mu,\xi}^i| \lesssim \frac{\mu^{\frac{N+2}{2}}}{d(\xi,\partial \Omega)^{N}}+\left\{\aligned &\mu^{\frac{N}{2}} |x-\xi|^{-(N-1)} \,\, \text{if} \,\, |x-\xi| \lesssim \mu,\\
&\mu^{\frac{-(N-2)}{2}},\,\,\, \,\text{if} |x-\xi|\sim \mu\\
&\mu^{\frac{N+4}{2}}|x-\xi|^{-N-1}, \,\, \,\,\text{if} \,\, |x-\xi| \gtrsim \mu
\endaligned\right.
\end{eqnarray*}
for $1\leq i\leq N$.  We denote the Robin function of $\Omega$ by $\varphi(x)$, that is, $\varphi(x)=H(x,x)$.  Then it is well known (\cite{CPY2021}) that for smooth $\Omega$,
\begin{eqnarray}\label{exp-rob}
\left\{\aligned
&\varphi(x)=\frac{1}{(2d(x,\partial \Omega))^{N-2}}\left(1+O\left(d(x,\partial \Omega)\right)\right) , \\
 &\nabla\varphi(x) =\frac{2(N-2)}{(2d(x,\partial \Omega))^{N-1}}\frac{x'-x}{|x'-x|}(1+O(d(x,\partial \Omega)))\endaligned\right.
\end{eqnarray}
as $d(x,\partial \Omega)\to0$, where $x' \in \partial \Omega$ is the unique point satisfying $d(x,\partial \Omega)=|x-x'|$.

\vskip0.12in

Let $\textit{i}^*:L^{\frac{p+1}{p}}(\Omega)\rightarrow H_0^1(\Omega)$ be the adjoint operator of the embedding $\textit{i}^*:H_0^1(\Omega)\rightarrow L^{p+1}(\Omega)$, namely, if $v \in L^{\frac{p+1}{p}}(\Omega)$ then $u =\textit{i}^*(v) \in H_0^1(\Omega)$ is the unique solution of the equation
\begin{eqnarray*}
\left\{\aligned  -\Delta u=v,\quad \text{in }\Omega,\\
u=0,\quad\text{on }\partial \Omega.
\endaligned\right.
\end{eqnarray*}
It follows from the continuity of the adjoint operator that
\begin{eqnarray*}
\|\textit{i}^*(v)\|_{H_0^1(\Omega)} \lesssim \|v\|_{L^{\frac{p+1}{p}}(\Omega)},\quad \forall v\in L^{\frac{p+1}{p}}(\Omega).
\end{eqnarray*}
Let 
\begin{eqnarray*}\label{ker-BN}
E_\ve=\left\{\aligned
&\text{span}\left\{Z_{\mu,\xi}^i\mid0\leq i \leq N\right\},\quad \overline{\lambda}\not=\lambda_\kappa,\\
&\text{span}\left\{Z_{\mu,\xi}^i,e_{k,l}\mid0\leq i \leq N, 1 \leq l \leq m_k\right\},\quad \overline{\lambda}=\lambda_k,
\endaligned\right.
\end{eqnarray*}
and $\Pi^\perp:H_0^1(\Omega)\rightarrow E_\ve^\perp$ is the projection, where
\begin{eqnarray*}
E_\ve^\perp=\left\{v\in H^1_0(\Omega)\mid \langle v, \phi\rangle=0, \forall\phi\in E_\ve\right\}
\end{eqnarray*}
and $\langle u, v\rangle=\int_{\Omega}\nabla u\nabla vdx$ is the usual inner product in $H_0^1(\Omega)$.  Then we have the following fundamental result.
\begin{lemma}\label{Lem2.1}
Let $N\geq4$, then
\begin{equation*}
\|\rho\|_{H_0^1(\Omega)} \lesssim \|g\|_{L^{\frac{p+1}{p}}(\Omega)}
\end{equation*}
for any $\rho \in E_\ve^\perp$ and $g\in H_0^{1}(\Omega)$ satisfying
\begin{eqnarray*}
\Pi^{\perp}\left[\rho-i^*\left(\lambda \rho+f'(\mathcal{W}_{\mu_\ve,\xi_\ve})\rho+g\right)\right]=0.
\end{eqnarray*}
\end{lemma}
\begin{proof}
Under the condition~\eqref{Struwe}, the proof of the conclusion is achieved by applying the same blow-up arguments used for \cite[Proposition~4.1]{IV} (see also \cite[Proposition~3.1]{LVWW}).  Since these blow-up arguments are well known nowadays, we omit the details here.
\end{proof}

\section{The inverse reduction and the basic estimates}
By \eqref{BN} and \eqref{Projection}, $v_\ve$, the remaining term in the Struwe decomposition given by \eqref{v}, satisfies the following equation
\begin{eqnarray}\label{v-equ}
\left\{\aligned
&-\Delta {v}_\ve= \lambda {v}_{\ve}+f(\mathcal{W}_{\mu_\ve,\xi_\ve}+v_{\ve})-f(\mathcal{W}_{\mu_\ve,\xi_\ve})+f(\mathcal{W}_{\mu_\ve,\xi_\ve})-f(U_{\mu_\ve,\xi_\ve})+\lambda \mathcal{W}_{\mu_\ve,\xi_\ve},\quad\text{in }\Omega,\\
&\langle v_\ve, \mathcal{Z}_{\mu_\ve, \xi_\ve}^i\rangle=\int_{\Omega}U_{\mu_\ve, \xi_\ve}^{p-1}\Phi_{\mu_\ve,\xi_\ve}^i v_\ve dx=0,\quad 0\leq i\leq N,\\
&v_\ve=0,\quad\text{on }\partial\Omega,
\endaligned\right.
\end{eqnarray}
where for the sake of simplicity, we denote $f(u)=|u|^{p-2}u$.
\begin{lemma}\label{L-infinity}
Let $N\geq4$ and $v_\ve$ is the solution of \eqref{v-equ}, then we have
\begin{eqnarray*}
\|v_\ve\|_{L^\infty (\Omega)} \lesssim  \mu_\ve^{-\frac{N-2}{2}}\kappa_\ve^{N-2}+\mu_\ve^{\frac{6-N}{2}}
\end{eqnarray*}
as $\ve\to0$,  where $\kappa_\ve=\frac{\mu_\ve}{d(\xi_\ve,\partial \Omega)}$.
\end{lemma}
\begin{proof}
Let
\begin{equation}\label{eqa3.2}
V_\ve(x)=\frac{\mu_\ve^{\frac{N-2}{2}}v_{\ve}(\mu_\ve x+\xi_\ve)}{\|{v}_{\ve}\|_{H_0^1(\Omega)}},
\end{equation}
then by setting $V_\ve(x)=0$ in $\mathbb{R}^N\backslash\Omega_{\mu_\ve,\xi_\ve}$ and the smoothness of $\Omega$, we know that $V_\ve$ is bounded in $D^{1,2}(\mathbb{R}^N)$, where $D^{1,2}(\bbr^d)=\dot{W}^{1,2}(\bbr^d)$ is the usual homogeneous Sobolev space (cf. \cite[Definition~2.1]{FG2021}) and
\begin{eqnarray*}
\Omega_{\mu_\ve,\xi_\ve}:=\frac{\Omega-\xi_\ve}{\mu_\ve}=\left\{x\in\mathbb{R}^N\mid \mu_\ve x+\xi_{\ve}\in\Omega\right\}.
\end{eqnarray*}
We introduce the set
\begin{eqnarray}\label{A1}
\mathcal{A}_1=\left\{x \in \Omega\mid \left|\mathcal{W}_{\mu_\ve,\xi_\ve}(x)\right|\lesssim\left|v_\ve(x)\right|\right\}.
\end{eqnarray}
For $x\in \mathcal{A}_1$,
\begin{eqnarray*}
\left|f(\mathcal{W}_{\mu_\ve,\xi_\ve}+v_{\ve})-f(\mathcal{W}_{\mu_\ve,\xi_\ve})-f'(\mathcal{W}_{\mu_\ve,\xi_\ve})v_\ve\right|
&\lesssim&|v_\ve|^{p}\\
&\lesssim&\left\{\aligned
&|v_\ve|^{p},\quad p\geq2,\\
&f''(\mathcal{W}_{\mu_\ve,\xi_\ve})|v_\ve|^2,\quad 1<p<2.
\endaligned\right.
\end{eqnarray*}
For $x\in\Omega\backslash\mathcal{A}_1$, by the Taylor expansion,
\begin{eqnarray*}
\left|f(\mathcal{W}_{\mu_\ve,\xi_\ve}+v_{\ve})-f(\mathcal{W}_{\mu_\ve,\xi_\ve})-f'(\mathcal{W}_{\mu_\ve,\xi_\ve})v_\ve\right|
\lesssim f''(\mathcal{W}_{\mu_\ve,\xi_\ve})|v_\ve|^2.
\end{eqnarray*}
Thus, we always have
\begin{eqnarray}\label{taylor-1}
\left|f(\mathcal{W}_{\mu_\ve,\xi_\ve}+v_{\ve})-f(\mathcal{W}_{\mu_\ve,\xi_\ve})-f'(\mathcal{W}_{\mu_\ve,\xi_\ve})v_\ve\right|
\lesssim\left\{\aligned
&f''(\mathcal{W}_{\mu_\ve,\xi_\ve})|v_\ve|^2+|v_\ve|^{p},\quad p\geq2,\\
&f''(\mathcal{W}_{\mu_\ve,\xi_\ve})|v_\ve|^2,\quad 1<p<2.
\endaligned\right.
\end{eqnarray}
Since by applying the maximum principle to \eqref{regular} and by the positivity of $G(x,\xi_\ve)$, we have 
\begin{eqnarray}\label{decay-H}
H(x,\xi_\ve)\lesssim \min\left\{\frac{1}{d(\xi_\ve,\partial \Omega)^{N-2}},\frac{1}{|x-\xi_\ve|^{N-2}}\right\}\quad\text{in }\Omega,
\end{eqnarray}
similarly, by \eqref{Struwe} and \eqref{exp-bub}, we also have
\begin{eqnarray}\label{taylor-2}
\left|f(\mathcal{W}_{\mu_\ve,\xi_\ve})-f(U_{\mu_\ve,\xi_\ve})-f'(U_{\mu_\ve,\xi_\ve})(\mathcal{W}_{\mu_\ve,\xi_\ve}-U_{\mu_\ve,\xi_\ve})\right|\lesssim\left\{\aligned
&\frac{\mu_\ve^{N-2}f''(U_{\mu_\ve,\xi_\ve})}{d(\xi_\ve,\partial \Omega)^{2(N-2)}}+\frac{\mu_\ve^{\frac{N+2}{2}}}{d(\xi_\ve,\partial \Omega)^{N+2}},\quad p\geq2,\\
&\frac{\mu_\ve^{N-2}f''(U_{\mu_\ve,\xi_\ve})}{d(\xi_\ve,\partial \Omega)^{2(N-2)}},\quad 1<p<2.
\endaligned\right.
\end{eqnarray}
Thus, by \eqref{v-equ} and \eqref{eqa3.2}, $V_\ve$ satisfies 
\begin{eqnarray}\label{eqa3.6}
\left\{\aligned
&-\Delta V_\ve=\lambda \mu_\ve^2V_{\ve}+(f'(U_{1,0})+\mathcal{O}(\kappa_\ve^{N-2}))V_{\ve}+\mathcal{K}_{\ve},\quad\text{in }\Omega_{\mu_\ve,\xi_\ve},\\
&V_\ve=0,\quad \quad \text{on }\partial\Omega_{\mu_\ve,\xi_\ve}.
\endaligned\right.
\end{eqnarray}
where by \eqref{taylor-1} and \eqref{taylor-2}, 
\begin{eqnarray*}
\mathcal{K}_{\ve}=O\left(\|v_\ve\|_{H_0^1(\Omega)}|V_\ve|^2+\|v_\ve\|_{H_0^1(\Omega)}^{p-1}|V_{\ve}|^p{\bf 1}_{p\geq2}\right)+\mathcal{O}\left(
\frac{\kappa_\ve^{N-2}+\mu_\ve^2 U}{\|{v}_\ve\|_{H_0^1(\Omega)}}\right).
\end{eqnarray*}
Here, we use the notation ${\bf 1}_{(A)}=1$ if $(A)$ holds and ${\bf 1}_{(A)}=0$ if $(A)$ does not hold, where $(A)$ is a condition.  Since by the choice of $\mu_\ve$ given by \eqref{mu-ve} and the definition of $v_\ve$, we must have $\|v_\ve\|_{L^{\infty}(\Omega)}\lesssim \mu_\ve^{-\frac{N-2}{2}}$, by \eqref{eqa3.2} and applying the classical Moser iteration to \eqref{eqa3.6}, we have
\begin{eqnarray}\label{eqa31.6}
\left\|V_\ve\right\|_{L^\infty\left(\Omega_{\mu_\ve,\xi_\ve}\right)}\lesssim
\frac{\kappa_\ve^{N-2}+\mu_\ve^2}{\|{v}_\ve\|_{H_0^1(\Omega)}}+\|{v}_\ve\|_{H_0^1(\Omega)}^{\frac{1\wedge(p-1)}{2}}
\end{eqnarray}
as $\ve\to0$, where $a\wedge b=\min\{a,b\}$.  Since by the condition~\eqref{condition}, we must have $\|{v}_\ve\|_{H_0^1(\Omega)}\to0$ as $\ve\to0$.  Thus, if $\frac{\kappa_\ve^{N-2}+\mu_\ve^2}{\|{v}_\ve\|_{H_0^1(\Omega)}}\gtrsim1$ then by \eqref{eqa3.2} and \eqref{eqa31.6}, we have already established the conclusion.  Otherwise, if $\frac{\kappa_\ve^{N-2}+\mu_\ve^2}{\|{v}_\ve\|_{H_0^1(\Omega)}}\to0$ as $\ve\to0$ then by \eqref{eqa31.6}, we must have $\left\|V_\ve\right\|_{L^\infty\left(\Omega_{\mu_\ve,\xi_\ve}\right)}\to0$ as $\ve\to0$.  We claim that in this case, 
\begin{eqnarray}\label{eqa3.7}
\left\|V_\ve\right\|_{L^\infty\left(\Omega_{\mu_\ve,\xi_\ve}\right)}\lesssim
\frac{\kappa_\ve^{N-2}+\mu_\ve^2}{\|{v}_\ve\|_{H_0^1(\Omega)}}.
\end{eqnarray}
Assume the contrary.  Then 
\begin{eqnarray}\label{eqa33.8}
\frac{\kappa_\ve^{N-2}+\mu_\ve^2}{\left\|V_\ve\right\|_{L^\infty\left(\Omega_{\mu_\ve,\xi_\ve}\right)}\|{v}_\ve\|_{H_0^1(\Omega)}}\rightarrow 0
\end{eqnarray}
as $\ve\to0$ up to subsequence.  Let 
\begin{eqnarray*}
\bar{V}_{\ve}=\frac{V_\ve}{\left\|V_\ve\right\|_{L^\infty\left(\Omega_{\mu_\ve,\xi_\ve}\right)}}.
\end{eqnarray*}
Then by \eqref{eqa3.6} and \eqref{eqa33.8}, $\bar{V}_{\ve}$ satisfies
\begin{eqnarray}\label{eqa33.6}
\left\{\aligned
&-\Delta \bar{V}_\ve=\lambda \mu_\ve^2\bar{V}_{\ve}+(f'(U_{1,0})+\mathcal{O}(\kappa_\ve^{N-2}))\bar{V}_{\ve}+\overline{\mathcal{K}}_{\ve},\quad\text{in }\Omega_{\mu_\ve,\xi_\ve},\\
&\bar{V}_\ve=0,\quad \quad \text{on }\partial\Omega_{\mu_\ve,\xi_\ve}.
\endaligned\right.
\end{eqnarray}
where
\begin{eqnarray*}
\overline{\mathcal{K}}_{\ve} =O\left(\|v_\ve\|_{H_0^1(\Omega)}|V_\ve|+\|v_\ve\|_{H_0^1(\Omega)}^{p-1}|V_{\ve}|^{p-1}{\bf 1}_{p\geq2}\right)\bar{V}_{\ve}+o(1).
\end{eqnarray*}
Since $\left\|V_\ve\right\|_{L^\infty\left(\Omega_{\mu_\ve,\xi_\ve}\right)}\to0$ as $\ve\to0$ and $\left\|\bar{V}_\ve\right\|_{L^\infty\left(\Omega_{\mu_\ve,\xi_\ve}\right)}=1$, by applying the standard elliptic regularity to \eqref{eqa33.6}, we have, for any $\gamma \in (0,1)$, 
\begin{equation} \label{eqa3.9}
\|\bar{V}_\ve\|_{C_{loc}^{1,\gamma}(\mathbb{R}^N)} \lesssim \mu_\ve^2+\left\|U_{1,0}\right\|_{L_{loc}^{p+1}(\mathbb{R}^N)}+o(1),
\end{equation}
which, together with $\left\|\bar{V}_\ve\right\|_{L^\infty\left(\Omega_{\mu_\ve,\xi_\ve}\right)}=1$, implies that there exists $R>0$ and $x_\ve\in\mathbb{R}^N$ such that $|x_\ve| \leq R$ and $|\bar{V}_{\ve}(x_{\ve})|=1$.  By \eqref{eqa3.9} and the Arzela-Ascoli theorem, $\bar{V}_{\ve}\rightarrow \bar{V}_{0}$ in $C_{loc}^1(\mathbb{R}^N)\cap L^\infty(\mathbb{R}^N)$ as $\ve\to0$, which, together with \eqref{Talenti}, \eqref{kernel}, \eqref{eqa3.2}, \eqref{eqa33.6} and the orthogonal condition in \eqref{v-equ}, implies that $\bar{V}_{0}$ satisfies the equation~\eqref{linear} in the weak sense with the orthogonal conditions 
\begin{eqnarray*}
\int_{\mathbb{R}^N} \nabla \bar{V}_0  \nabla \Psi^idx=0,\quad\text{for all }i=0,\cdots,N.
\end{eqnarray*}
It follows from the nondegeneracy of the Talenti bubble $U_{1,0}$ and the standard elliptic regularity theory that $\bar{V}_0=0$, which contradicts $|\bar{V}_{\ve}(x_{\ve})|=1$ for some $|x_\ve| \leq R$ as $\ve\to0$.
Thus, the claim~\eqref{eqa3.7} holds true and the conclusion follows immediately from \eqref{eqa3.2}.
 \end{proof}

With Lemma~\ref{L-infinity} in hands, we can drive the following basic estimate of $\|v_\ve\|_{H^1_0(\Omega)}$.

\begin{lemma}\label{basic-esti}
Let $N\geq4$ and $v_\ve$ is the solution of \eqref{v-equ}, then $\mu_\ve^{-\frac{N-2}{2}}\|v_{\ve}\|_{H_0^1(\Omega)}\rightarrow+\infty$ as $\ve\to0$.
\end{lemma}
\begin{proof}
Suppose the contrary that $\|v_\ve\|_{H_0^1(\Omega)} \lesssim \mu_\ve^{\frac{N-2}{2}}$ as $\ve\to0$. 
We first test the equation~\eqref{v-equ} by $Z_{\mu_\ve,\xi_\ve}^0$, which implies that
\begin{eqnarray}\label{test-0}
0&=&\int_{\Omega}\left(f\left(\mathcal{W}_{\mu_\ve,\xi_\ve}+v_\ve\right) -f\left(\mathcal{W}_{\mu_\ve,\xi_\ve}\right)\right)Z_{\mu_\ve,\xi_\ve}^0dx+\lambda\int_{\Omega}v_\ve Z_{\mu_\ve,\xi_\ve}^0dx\notag\\
&&+\lambda \int_{\Omega}\mathcal{W}_{\mu_\ve,\xi_\ve}Z_{\mu_\ve,\xi_\ve}^0dx+\int_{\Omega}\left(f\left(\mathcal{W}_{\mu_\ve,\xi_\ve}\right)-f\left(U_{\mu_\ve,\xi_\ve}\right)\right)Z_{\mu_\ve,\xi_\ve}^0dx\notag\\
&:=&I+II+III+IV.
\end{eqnarray}
Similar to \eqref{taylor-1}, we have
\begin{eqnarray*}
\left|f(\mathcal{W}_{\mu_\ve,\xi_\ve}+v_{\ve})-f(\mathcal{W}_{\mu_\ve,\xi_\ve})\right|\lesssim |v_\ve|^p,
\end{eqnarray*}
for $x\in \mathcal{A}_1$ and
\begin{eqnarray*}
\left|f(\mathcal{W}_{\mu_\ve,\xi_\ve}+v_{\ve})-f(\mathcal{W}_{\mu_\ve,\xi_\ve})-f'(\mathcal{W}_{\mu_\ve,\xi_\ve})v_\ve\right|&\lesssim& f''(\mathcal{W}_{\mu_\ve,\xi_\ve})|v_\ve|^2\\
&\lesssim&f''(\mathcal{W}_{\mu_\ve,\xi_\ve})|v_\ve|^2{\bf 1}_{p\geq2}+|v_\ve|^p
\end{eqnarray*}
for $x\in \Omega\backslash\mathcal{A}_1$,
where the set $\mathcal{A}_1$ is given by \eqref{A1}.  Thus, by \eqref{exp-ker}, \eqref{comp}, the orthogonal condition~\eqref{orthogonality}, H\"older inequality and $\|v_\ve\|_{H_0^1(\Omega)} \lesssim \mu_\ve^{\frac{N-2}{2}}$,
\begin{eqnarray}\label{equa3.7}
I&=&\int_{\Omega\backslash \mathcal{A}_1}\left(f(\mathcal{W}_{\mu_\ve,\xi_\ve}+v_{\ve}) -f(\mathcal{W}_{\mu_\ve,\xi_\ve})\right)Z_{\mu_\ve,\xi_\ve}^0dx+\int_{\mathcal{A}_1}\left(f (\mathcal{W}_{\mu_\ve,\xi_\ve}+v_{\ve}) -f(\mathcal{W}_{\mu_\ve,\xi_\ve})\right)Z_{\mu_\ve,\xi_\ve}^0dx\notag\\
&=&\int_{\Omega\backslash \mathcal{A}_1} f'(\mathcal{W}_{\mu_\ve,\xi_\ve})Z_{\mu_\ve,\xi_\ve}^0v_{\ve}dx+\mathcal{O}\left(\int_{\Omega}f''(\mathcal{W}_{\mu_\ve,\xi_\ve})|v_\ve|^2\left|Z_{\mu_\ve,\xi_\ve}^0\right|dx+\mu_\ve^{N}\right)\notag\\
&=&\int_{\Omega} f'(\mathcal{W}_{\mu_\ve,\xi_\ve})Z_{\mu_\ve,\xi_\ve}^0v_{\ve}dx+\mathcal{O}\left(\int_{\Omega}f'(U_{\mu_\ve,\xi_\ve})|v_\ve|^2dx+\mu_\ve^{N}\right)\notag\\
&=& \int_{\Omega}\left(f'(\mathcal{W}_{\mu_\ve,\xi_\ve})-f'(U_{\mu_\ve,\xi_\ve})\right)Z_{\mu_\ve,\xi_\ve}^0v_{\ve}dx
+\int_{\Omega}v_{\ve}f'(U_{\mu_\ve,\xi_\ve})\left(Z_{\mu_\ve,\xi_\ve}^0-\Phi_{\mu_\ve,\xi_\ve}^0\right)dx\notag\\
&&+\mathcal{O}\left(\mu_\ve^{N-2}\right).
\end{eqnarray}
By Lemma~\ref{App-esti}, \eqref{comp} and \eqref{decay-H},
\begin{eqnarray}
 \left|\int_{\Omega \backslash B_{d(\xi_\ve,\partial \Omega)}(\xi_\ve)}\left(f'(\mathcal{W}_{\mu_\ve,\xi_\ve})-f'(U_{\mu_\ve,\xi_\ve})\right)Z_{\mu_\ve,\xi_\ve}^0v_{\ve}dx\right| &\lesssim&\mu_\ve^{\frac{N-2}{2}}\left(\int_{\Omega \backslash B_{d(\xi_\ve,\partial \Omega)}(\xi_\ve)}U_{\mu_\ve,\xi_\ve}^{p+1}dx\right)^{\frac{p}{p+1}}\notag\\
&\lesssim&  \kappa_\ve^{\frac{N+2}{2}} \mu_\ve^{\frac{N-2}{2}}.\label{eqa4.0}
\end{eqnarray}
By Lemma~\ref{App-esti}, \eqref{exp-bub} and \eqref{comp},
\begin{eqnarray*}
&&\left|\int_{B_{d(\xi_\ve,\partial \Omega)}(\xi_\ve)}\left(f'(\mathcal{W}_{\mu_\ve,\xi_\ve})-f'(U_{\mu_\ve,\xi_\ve})\right)Z_{\mu_\ve,\xi_\ve}^0v_{\ve}dx\right|\notag\\
&\lesssim&\mu_\ve^{\frac{N-2}{2}}\int_{B_{d(\xi_\ve,\partial \Omega)}(\xi_\ve)}\left(H(x,\xi_\ve)+\frac{\mu_\ve^2}{d(\xi_\ve, \partial\Omega)^{N}}\right)U_{\mu_\ve,\xi_\ve}^{p-1}\left|v_{\ve}\right|dx\notag\\
&\lesssim&\kappa_\ve^{N-2}\left(\int_{B_{d(\xi_\ve,\partial \Omega)}(\xi_\ve)}U_{\mu_\ve,\xi_\ve}^{\frac{p^2-1}{p}}dx\right)^{\frac{p}{p+1}}\notag\\
&\lesssim&\left\{\aligned
&\kappa_\ve^{N-2} \mu_\ve^{\frac{N-2}{2}},\quad N=4,5,\\
&\kappa_\ve^4\left|\log \kappa_\ve\right|^{\frac{2}{3}} \mu_\ve^{2},\quad N=6\,,\\
&\kappa_\ve^{\frac{N+2}{2}}\mu_\ve^{\frac{N-2}{2}},\quad N\geq7,
\endaligned\right.
\end{eqnarray*}
which, together with \eqref{eqa4.0}, implies that
\begin{eqnarray}\label{I1-0}
\left|\int_{\Omega}\left(f'(\mathcal{W}_{\mu_\ve,\xi_\ve})-f'(U_{\mu_\ve,\xi_\ve})\right)Z_{\mu_\ve,\xi_\ve}^0v_{\ve}dx\right|
&\lesssim&\left\{\aligned
&\kappa_\ve^{N-2} \mu_\ve^{\frac{N-2}{2}},\quad N=4,5,\\
&\kappa_\ve^4\left|\log \kappa_\ve\right|^{\frac{2}{3}} \mu_\ve^{2},\quad N=6\,,\\
&\kappa_\ve^{\frac{N+2}{2}}\mu_\ve^{\frac{N-2}{2}},\quad N\geq7.
\endaligned\right.
\end{eqnarray}
Similarly, by \eqref{exp-ker}, we also have 
\begin{eqnarray}\label{I2-0}
\left|\int_{\Omega}v_{\ve}f'(U_{\mu_\ve,\xi_\ve})\left(Z_{\mu_\ve,\xi_\ve}^0-\Phi_{\mu_\ve,\xi_\ve}^0\right)dx\right|
&\lesssim&\left\{\aligned
&\kappa_\ve^{N-2} \mu_\ve^{\frac{N-2}{2}},\quad N=4,5,\\
&\kappa_\ve^4\left|\log \kappa_\ve\right|^{\frac{2}{3}} \mu_\ve^{2},\quad N=6\,,\\
&\kappa_\ve^{\frac{N+2}{2}}\mu_\ve^{\frac{N-2}{2}},\quad N\geq7.
\endaligned\right.
\end{eqnarray}
Inserting \eqref{I1-0} and \eqref{I2-0} into \eqref{equa3.7}, we have
\begin{eqnarray}\label{I}
|I|=\mathcal{O}\left(\kappa_\ve^N\right)+\mathcal{O}\left(\mu_\ve^{N-2}\right).
\end{eqnarray}
Since $\|v_\ve\|_{H_0^1(\Omega)} \lesssim \mu_\ve^{\frac{N-2}{2}}$, by Lemma~\ref{App-esti}, the H\"older inequality and \eqref{exp-ker},
\begin{eqnarray}\label{II}
|II|\lesssim\mu_\ve^{\frac{N-2}{2}}\left(\int_{\Omega}\left(Z_{\mu_\ve,\xi_\ve}^0\right)^{\frac{p+1}{p}}dx\right)^{\frac{p}{p+1}}=\left\{\aligned
&\mathcal{O}\left(\mu_\ve^{N-2}\right),\quad N=4,5,\\
&\mathcal{O}\left(\mu_\ve^{4}\left|\log\mu_\ve\right|^{\frac{2}{3}}\right),\quad N=6,\\
&\mathcal{O}\left(\mu_\ve^{\frac{N+2}{2}}\right),\quad N\geq7.
\endaligned\right.
\end{eqnarray}
Since $H(x,\xi_\ve)\lesssim \frac{1}{d(\xi_\ve,\partial \Omega)^{N-2}}$ by \eqref{decay-H}, by \eqref{exp-ker}, \eqref{exp-bub} and Lemma~\ref{App-esti},
\begin{eqnarray}\label{III-5}
III&=&\lambda \int_{\Omega}\mathcal{W}_{\mu_\ve,\xi_\ve}Z_{\mu_\ve,\xi_\ve}^0dx\notag\\
&=&\lambda \int_{B_{d(\xi_\ve,\partial \Omega)}(\xi_\ve)}U_{\mu_\ve,\xi_\ve}\Phi_{\mu_\ve,\xi_\ve}^0dx+\mathcal{O}\left(\mu_\ve^{-\frac{N-2}{2}}\kappa_\ve^{N-2}\int_{B_{d(\xi_\ve,\partial \Omega)}(\xi_\ve)}U_{\mu_\ve,\xi_\ve}dx\right)\notag\\
&&+\mathcal{O}\left(\int_{\Omega\backslash B_{d(\xi_\ve,\partial \Omega)}(\xi_\ve)}U_{\mu_\ve,\xi_\ve}^2dx\right)\notag\\
&=&\lambda D_{N,1}\mu_\ve^2+\mathcal{O}\left(\mu_\ve^2\kappa_\ve^{N-4}\right)
\end{eqnarray}
for $N\geq5$, where $D_{N,1}=\int_{\mathbb{R}^N}U_{1,0}\Phi_{1,0}^0dx$.  For $N=4$, by \eqref{exp-ker}, \eqref{exp-bub} and Lemma~\ref{App-esti},
\begin{eqnarray}\label{III-41}
III&=&\lambda \int_{\Omega}\mathcal{W}_{\mu_\ve,\xi_\ve}Z_{\mu_\ve,\xi_\ve}^0dx\notag\\
&=&\lambda \int_{\Omega}U_{\mu_\ve,\xi_\ve}\Phi_{\mu_\ve,\xi_\ve}^0dx+\mathcal{O}\left(\mu_\ve^{-1}\kappa_\ve^{2}\int_{\Omega}U_{\mu_\ve,\xi_\ve}dx\right)\notag\\
&\geq&\lambda \int_{B_{d(\xi_\ve,\partial \Omega)}(\xi_\ve)}U_{\mu_\ve,\xi_\ve}\Phi_{\mu_\ve,\xi_\ve}^0dx+\mathcal{O}\left(\mu_\ve^2\right)\notag\\
&=&\lambda D_{4,2}\mu_\ve^2\left|\log\kappa_\ve\right|+\mathcal{O}\left(\mu_\ve^2\right)
\end{eqnarray}
and
\begin{eqnarray}\label{III-42}
III&=&\lambda \int_{\Omega}\mathcal{W}_{\mu_\ve,\xi_\ve}Z_{\mu_\ve,\xi_\ve}^0dx\notag\\
&=&\lambda \int_{\Omega}U_{\mu_\ve,\xi_\ve}\Phi_{\mu_\ve,\xi_\ve}^0dx+\mathcal{O}\left(\mu_\ve^{-1}\kappa_\ve^{2}\int_{\Omega}U_{\mu_\ve,\xi_\ve}dx\right)\notag\\
&\leq&\lambda \int_{B_R(\eta)}U_{\mu_\ve,\xi_\ve}\Phi_{\mu_\ve,\xi_\ve}^0dx+\mathcal{O}\left(\mu_\ve^2\right)\notag\\
&=&\lambda D_{4,2}\mu_\ve^2\left|\log\mu_\ve\right|+\mathcal{O}\left(\mu_\ve^2\right),
\end{eqnarray}
where $B_R(\eta)$ is the ball centered at $\eta\in\Omega$ with the radius $R$ such that $\Omega\subset B_R(\eta)$ and $D_{4,2}$ is an absolute constant.  Again, similar to \eqref{taylor-2}, by \eqref{exp-ker}, \eqref{exp-bub}, \eqref{decay-H} and Lemma~\ref{App-esti},
\begin{eqnarray}\label{iv-1}
IV&=&\int_{B_{d(\xi_\ve,\partial \Omega)}(\xi_\ve)}\left(f\left(\mathcal{W}_{\mu_\ve,\xi_\ve}\right)-f\left(U_{\mu_\ve,\xi_\ve}\right)\right)Z_{\mu_\ve,\xi_\ve}^0dx+\mathcal{O}\left(\int_{\Omega\backslash B_{d(\xi_\ve,\partial \Omega)}(\xi_\ve)}U_{\mu_\ve,\xi_\ve}^{p+1}dx\right)\notag\\
&=&-\alpha_N\mu_\ve^{\frac{N-2}{2}}\int_{B_{d(\xi_\ve,\partial \Omega)}(\xi_\ve)}f'\left(U_{\mu_\ve,\xi_\ve}\right)\Phi_{\mu_\ve,\xi_\ve}^0H(x,\xi_\ve)dx+\mathcal{O}\left(\frac{\kappa_\ve^{2N-4}}{\mu_\ve^{N-2}}\int_{B_{d(\xi_\ve,\partial \Omega)}(\xi_\ve)}U_{\mu_\ve,\xi_\ve}^{p-1}dx\right)\notag\\
&&+\mathcal{O}\left(\frac{\kappa_\ve^{N}}{\mu_\ve^{\frac{N-2}{2}}}\int_{B_{d(\xi_\ve,\partial \Omega)}(\xi_\ve)}U_{\mu_\ve,\xi_\ve}^{p}dx\right)+\mathcal{O}\left(\kappa_\ve^N\right)\notag\\
&=&-\alpha_N\mu_\ve^{\frac{N-2}{2}}\int_{B_{d(\xi_\ve,\partial \Omega)}(\xi_\ve)}f'\left(U_{\mu_\ve,\xi_\ve}\right)\Phi_{\mu_\ve,\xi_\ve}^0H(x,\xi_\ve)dx+\left\{\aligned
&\mathcal{O}\left(\kappa_\ve^N\right),\quad N\geq5,\\
&\mathcal{O}\left(\kappa_\ve^4\left|\log\kappa_\ve\right|\right),\quad N=4,
\endaligned\right.
\end{eqnarray}
which, together with
\begin{eqnarray*}\label{iv-2}
&&\int_{B_{d(\xi_\ve,\partial \Omega)}(\xi_\ve)}f'\left(U_{\mu_\ve,\xi_\ve}\right)\Phi_{\mu_\ve,\xi_\ve}^0H(x,\xi_\ve)dx\notag\\
&=&\varphi(\xi_\ve)\int_{B_{d(\xi_\ve,\partial \Omega)}(\xi_\ve)}f'\left(U_{\mu_\ve,\xi_\ve}\right)\Phi_{\mu_\ve,\xi_\ve}^0dx+\mathcal{O}\left(\frac{\kappa_\ve^N}{\mu_\ve^{N}}\int_{B_{d(\xi_\ve,\partial \Omega)}(\xi_\ve)}U_{\mu_\ve,\xi_\ve}^p|x-\xi_\ve|^2dx\right)\\
&=&D_{N,3}\mu_{\ve}^{\frac{N-2}{2}}\varphi(\xi_\ve)+\mathcal{O}\left(\mu_\ve^{\frac{N-2}{2}}\kappa_\ve^2+\frac{\kappa_\ve^N}{\mu_\ve^{\frac{N-2}{2}}}\left|\log\kappa_\ve\right|\right),
\end{eqnarray*}
implies that
\begin{eqnarray}\label{IV}
IV=-\alpha_N\mu_\ve^{N-2}D_{N,3}\varphi(\xi_\ve)+\mathcal{O}\left(\kappa_\ve^N\left|\log\kappa_\ve\right|\right),
\end{eqnarray}
where $D_{N,3}=p\int_{\mathbb{R}^N}U_{1,0}^{p-1}\Phi^0_{1,0}dx$.
Inserting \eqref{I}, \eqref{II}, \eqref{III-5}, \eqref{III-41}, \eqref{III-42} and \eqref{IV} into \eqref{test-0}, we have
\begin{eqnarray}\label{test-05}
0=\lambda D_{N,1}\mu_\ve^2-\alpha_N\mu_\ve^{N-2}D_{N,3}\varphi(\xi_\ve)+\mathcal{O}\left(\mu_\ve^2\kappa_\ve^{N-4}\right)+\mathcal{O}\left(\kappa_\ve^N\left|\log\kappa_\ve\right|\right)+o\left(\mu_\ve^{2}\right)
\end{eqnarray}
for $N\geq5$,
\begin{eqnarray}\label{test-041}
0\geq\lambda D_{4,2}\mu_\ve^2\left|\log\kappa_\ve\right|-\alpha_N\mu_\ve^{2}D_{4,3}\varphi(\xi_\ve)+\mathcal{O}\left(\kappa_\ve^4\left|\log\kappa_\ve\right|\right)+\mathcal{O}\left(\mu_\ve^2\right)
\end{eqnarray}
and
\begin{eqnarray}\label{test-042}
0\leq\lambda D_{4,2}\mu_\ve^2\left|\log\mu_\ve\right|-\alpha_N\mu_\ve^{2}D_{4,3}\varphi(\xi_\ve)+\mathcal{O}\left(\kappa_\ve^4\left|\log\kappa_\ve\right|\right)+\mathcal{O}\left(\mu_\ve^2\right)
\end{eqnarray}
for $N=4$.  Clearly, $\xi_\ve\to\xi_0\in\overline{\Omega}$ as $\ve\to0$ up to a subsequence.  Since 
\begin{eqnarray*}
\kappa_\ve=\frac{\mu_\ve}{d(\xi_\ve,\partial \Omega)}\sim\mu_\ve
\end{eqnarray*}
if $\xi_0\in\Omega$, by \eqref{exp-rob}, \eqref{test-05} and \eqref{test-041}, we must have $\xi_0\in\partial\Omega$, which, together with \eqref{test-041} and \eqref{test-042}, implies that $\left|\log\mu_\ve\right|\sim\left|\log\kappa_\ve\right|$ as $\ve\to0$ for $N=4$.  Thus, by \eqref{test-041} and \eqref{test-042} once more,
 \begin{eqnarray}\label{test-04}
0=\lambda D_{4,2}\mu_\ve^2\left|\log\mu_\ve\right|-\alpha_N\mu_\ve^{2}D_{4,3}\varphi(\xi_\ve)+\mathcal{O}\left(\kappa_\ve^4\left|\log\kappa_\ve\right|\right)+\mathcal{O}\left(\mu_\ve^2\right)
\end{eqnarray}
for $N=4$.  By \eqref{exp-rob}, \eqref{test-05} and \eqref{test-04}, we have
\begin{eqnarray}\label{k-mu}
\kappa_\ve^{N-2}\sim\left\{\aligned
&\mu_\ve^2,\quad N\geq5,\\
&\mu_\ve^2\left|\log\mu_\ve\right|,\quad N=4,
\endaligned\right.
\end{eqnarray}
which, together with \eqref{regular} and \eqref{decay-H}, implies that
\begin{eqnarray} \label{H-esti}
H(x,\xi_\ve)\lesssim \left\{\aligned 
&\mu_\ve^{4-N},\quad N\geq 5,\\
&|\log\mu_\ve|,\quad N=4
\endaligned\right.
\quad\text{and}\quad
\mu_\ve\left|\frac{\partial H(x,\xi_\ve) }{\partial \xi_{i,\ve}}\right|\lesssim \left\{\aligned 
&\mu_\ve^{4-N+\frac{2}{N-2}},\quad N\geq 5,\\
&\mu_\ve|\log\mu_\ve|^{\frac{3}{2}},\quad N=4
\endaligned\right.
\end{eqnarray}
in $\Omega$, where $\xi_\ve=(\xi_{1,\ve}, \xi_{2,\ve}, \cdots, \xi_{N,\ve})$.  We next test the equation~\eqref{v-equ} by $Z_{\mu_\ve,\xi_\ve}^1$, which implies that
\begin{eqnarray}\label{test-1}
0&=&\int_{\Omega}\left(f\left(\mathcal{W}_{\mu_\ve,\xi_\ve}+v_\ve\right) -f\left(\mathcal{W}_{\mu_\ve,\xi_\ve}\right)\right)Z_{\mu_\ve,\xi_\ve}^1dx+\lambda\int_{\Omega}v_\ve Z_{\mu_\ve,\xi_\ve}^1dx\notag\\
&&+\lambda \int_{\Omega}\mathcal{W}_{\mu_\ve,\xi_\ve}Z_{\mu_\ve,\xi_\ve}^1dx+\int_{\Omega}\left(f\left(\mathcal{W}_{\mu_\ve,\xi_\ve}\right)-f\left(U_{\mu_\ve,\xi_\ve}\right)\right)Z_{\mu_\ve,\xi_\ve}^1dx\notag\\
&:=&I_1+II_1+III_1+IV_1.
\end{eqnarray}
Similar to the estimate of $I$ in \eqref{equa3.7},
\begin{eqnarray}\label{equa3.71}
I_1&=&\int_{\Omega}\left(f'(\mathcal{W}_{\mu_\ve,\xi_\ve})-f'(U_{\mu_\ve,\xi_\ve})\right)Z_{\mu_\ve,\xi_\ve}^1v_{\ve}dx
+\int_{\Omega}v_{\ve}f'(U_{\mu_\ve,\xi_\ve})\left(Z_{\mu_\ve,\xi_\ve}^1-\Phi_{\mu_\ve,\xi_\ve}^1\right)dx\notag\\
&&+\mathcal{O}\left(\mu_\ve^{N-2}\right),
\end{eqnarray}
which, together with \eqref{exp-ker}, \eqref{comp}, \eqref{H-esti} and Lemma~\ref{App-esti}, implies that
\begin{eqnarray}\label{equa3.72}
|I_1|&\lesssim&\mu_\ve^{\frac{N-2}{2}}\int_{\Omega}\left(\max\left\{\mu_\ve\left|\frac{\partial H(x,\xi_\ve) }{\partial \xi_{i,\ve}}\right|, H(x,\xi_\ve)\right\}+\frac{\mu_\ve^2}{d(\xi_\ve,\Omega)^N}\right)U_{\mu_\ve,\xi_\ve}^{p-1}|v_{\ve}|dx+\mathcal{O}\left(\mu_\ve^{N-2}\right)\notag\\
&\lesssim&\mu_\ve^{2}\left(\int_{\Omega}U_{\mu_\ve,\xi_\ve}^{\frac{p^2-1}{p}}dx\right)^{\frac{p}{p+1}}+\mathcal{O}\left(\mu_\ve^{N-2}\right)\notag\\
&\lesssim&\left\{\aligned
&\mu_\ve^{3},\quad N=5,\\
&\mu_\ve^4|\log\mu_\ve|^{\frac{2}{3}},\quad N=6,\\
&\mu_\ve^4,\quad N\geq7.
\endaligned\right.
\end{eqnarray}
For $N=4$, we shall use Lemma~\ref{L-infinity} to replace the assumption $\|v_\ve\|_{H_0^1(\Omega)} \lesssim \mu_\ve^{\frac{N-2}{2}}$ in the estimate~\eqref{equa3.71}.  Thus, by \eqref{exp-ker}, \eqref{comp}, \eqref{k-mu}, \eqref{H-esti} and Lemma~\ref{App-esti},
\begin{eqnarray}\label{equa3.73}
I_1&=&\int_{\Omega}\left(f'(\mathcal{W}_{\mu_\ve,\xi_\ve})-f'(U_{\mu_\ve,\xi_\ve})\right)Z_{\mu_\ve,\xi_\ve}^1v_{\ve}dx
+\int_{\Omega}v_{\ve}f'(U_{\mu_\ve,\xi_\ve})\left(Z_{\mu_\ve,\xi_\ve}^1-\Phi_{\mu_\ve,\xi_\ve}^1\right)dx\notag\\
&&+\mathcal{O}\left(\int_{\Omega}|v_{\ve}|^2U_{\mu_\ve,\xi_\ve}\left|Z_{\mu_\ve,\xi_\ve}^1\right|dx+\mu_\ve^{4}\right)\notag\\
&=&\int_{\Omega}\left(f'(\mathcal{W}_{\mu_\ve,\xi_\ve})-f'(U_{\mu_\ve,\xi_\ve})\right)Z_{\mu_\ve,\xi_\ve}^1v_{\ve}dx
+\int_{\Omega}v_{\ve}f'(U_{\mu_\ve,\xi_\ve})\left(Z_{\mu_\ve,\xi_\ve}^1-\Phi_{\mu_\ve,\xi_\ve}^1\right)dx\notag\\
&&+\mathcal{O}\left(\mu_\ve^{2}|\log\mu_\ve|^2\left(\int_{B_{d(\xi_\ve,\partial\Omega)}(\xi_\ve)}U_{\mu_\ve,\xi_\ve}\left|\Phi_{\mu_\ve,\xi_\ve}^1\right|dx+\int_{\Omega\backslash B_{d(\xi_\ve,\partial\Omega)}(\xi_\ve)}\kappa_\ve U_{\mu_\ve,\xi_\ve}^2dx\right)+\mu_\ve^{4}\right)\notag\\
&\lesssim&\mu_\ve\int_{\Omega}\left(\max\left\{\mu_\ve\left|\frac{\partial H(x,\xi_\ve) }{\partial \xi_{1,\ve}}\right|, H(x,\xi_\ve)\right\}+\frac{\mu_\ve^2}{d(\xi_\ve,\partial\Omega)^4}\right)U_{\mu_\ve,\xi_\ve}^{2}|v_{\ve}|dx\notag\\
&&+\mathcal{O}\left(\mu_\ve^{4}|\log\mu_\ve|^2(1+\kappa_\ve^2|\log d(\xi_\ve,\partial\Omega)|)\right)\notag\\
&=&\mathcal{O}\left(\mu_\ve^{4}|\log\mu_\ve|^3\right),
\end{eqnarray}
where we have used the estimate 
\begin{eqnarray}\label{point-z1}
\left|Z_{\mu_\ve,\xi_\ve}^i\right|+\mu_\ve^{\frac{N}{2}}\left|\frac{\partial H(x,\xi_\ve) }{\partial \xi_{i,\ve}}\right|\lesssim\left\{\aligned
&\left|\Phi_{\mu_\ve,\xi_\ve}^i\right|,\quad x\in B_{d(\xi_\ve,\partial\Omega)}(\xi_\ve),\\
&\kappa_\ve U_{\mu_\ve,\xi_\ve},\quad x\in\Omega\backslash B_{d(\xi_\ve,\partial\Omega)}(\xi_\ve),
\endaligned\right.
\end{eqnarray}
which comes from the maximum principle, \eqref{Struwe} and \eqref{ker-gre}.
Combining \eqref{equa3.72} and \eqref{equa3.73}, we have
\begin{eqnarray}\label{I1}
|I_1|
\lesssim\left\{\aligned
&\mu_\ve^{3},\quad N=4,5,\\
&\mu_\ve^4|\log\mu_\ve|^{\frac{2}{3}},\quad N=6,\\
&\mu_\ve^4,\quad N\geq7.
\endaligned\right.
\end{eqnarray}
Similarly, by using Lemma~\ref{L-infinity} for $N=4$ and the assumption $\|v_\ve\|_{H_0^1(\Omega)} \lesssim \mu_\ve^{\frac{N-2}{2}}$ for $N\geq5$, we have, by \eqref{k-mu}, \eqref{H-esti}, \eqref{point-z1} and Lemma~\ref{App-esti}, that
\begin{eqnarray}\label{II1}
&&|II_1|\notag\\
&\lesssim&\left\{\aligned
&\mu_\ve^{-1}\kappa_\ve^2\int_{B_{d(\xi_\ve,\partial\Omega)}(\xi_\ve)} \left|\Phi_{\mu_\ve,\xi_\ve}^1\right|dx+\mu_\ve\kappa_\ve\left(\int_{\Omega\backslash B_{d(\xi_\ve,\partial\Omega)}(\xi_\ve)} U_{\mu_\ve,\xi_\ve}^{\frac43}dx\right)^{\frac34},\quad N=4,\\
&\mu_\ve^{\frac{N-2}{2}}\left(\left(\int_{B_{d(\xi_\ve,\partial\Omega)}(\xi_\ve)} \left|\Phi_{\mu_\ve,\xi_\ve}^1\right|^{\frac{p+1}{p}}dx\right)^{\frac{p}{p+1}}+\kappa_\ve\left(\int_{\Omega\backslash B_{d(\xi_\ve,\partial\Omega)}(\xi_\ve)} U_{\mu_\ve,\xi_\ve}^{\frac{p+1}{p}}dx\right)^{\frac{p}{p+1}}\right),\quad N\geq5
\endaligned\right.\notag\\
&\lesssim&\left\{\aligned
&\mu_\ve^{2}\kappa_\ve,\quad N=4,\\
&\mu_\ve^{\frac72},\quad N=5,\\
&\mu_\ve^{4}|\log\mu_\ve|^{\frac{2}{3}},\quad N=6,\\
&\mu_\ve^{\frac{N+2}{2}},\quad N\geq7.
\endaligned\right.
\end{eqnarray}
Similar to \eqref{III-5}, by \eqref{comp}, \eqref{k-mu}, \eqref{H-esti}, \eqref{point-z1}, the symmetry of $\Phi_{\mu_\ve,\xi_\ve}^1$ and Lemma~\ref{App-esti},
\begin{eqnarray}\label{III1}
III_1&=&\lambda \int_{\Omega}\mathcal{W}_{\mu_\ve,\xi_\ve}Z_{\mu_\ve,\xi_\ve}^1dx\notag\\
&=&\mathcal{O}\left(\int_{B_{d(\xi_\ve,\partial \Omega)}(\xi_\ve)}\mu_\ve^{\frac{N}{2}}\left|\frac{\partial H(x,\xi_\ve) }{\partial \xi_{i,\ve}}\right|U_{\mu_\ve,\xi_\ve}dx\right)+\mathcal{O}\left(\int_{B_{d(\xi_\ve,\partial \Omega)}(\xi_\ve)}\mu_\ve^{\frac{N-2}{2}}H(x,\xi_\ve)\left|\Phi_{\mu_\ve,\xi_\ve}^1\right|dx\right)\notag\\
&&+\mathcal{O}\left(\int_{\Omega\backslash B_{d(\xi_\ve,\partial \Omega)}(\xi_\ve)}\kappa_\ve U_{\mu_\ve,\xi_\ve}^2dx\right)\notag\\
&=&\mathcal{O}\left(\mu_\ve^{2}\kappa_\ve^{N-3}\right)+\left\{\aligned
&\mathcal{O}\left(\kappa_\ve\mu_\ve^{2}|\log d(\xi_\ve,\partial \Omega)|\right),\quad N=4,\\
&\mathcal{O}\left(\mu_\ve^{N-2}\kappa_\ve\right),\quad N\geq5.
\endaligned\right.
\end{eqnarray}
Similar to \eqref{iv-1}, by \eqref{exp-ker}, \eqref{comp}, \eqref{exp-bub}, \eqref{point-z1} and Lemma~\ref{App-esti},
\begin{eqnarray*}\label{IV1-01}
IV_1&=&\int_{B_{d(\xi_\ve,\partial \Omega)}(\xi_\ve)}\left(f\left(\mathcal{W}_{\mu_\ve,\xi_\ve}\right)-f\left(U_{\mu_\ve,\xi_\ve}\right)\right)Z_{\mu_\ve,\xi_\ve}^1dx+\mathcal{O}\left(\int_{\Omega\backslash B_{d(\xi_\ve,\partial \Omega)}(\xi_\ve)}U_{\mu_\ve,\xi_\ve}^{p+1}dx\right)\notag\\
&=&-\alpha_N\mu_\ve^{\frac{N-2}{2}}\int_{B_{d(\xi_\ve,\partial \Omega)}(\xi_\ve)}f'\left(U_{\mu_\ve,\xi_\ve}\right)H(x,\xi_\ve)\Phi_{\mu_\ve,\xi_\ve}^1dx+\mathcal{O}\left(\frac{\kappa_\ve^{2N-3}}{\mu_\ve^{N-2}}\int_{B_{d(\xi_\ve,\partial \Omega)}(\xi_\ve)}U_{\mu_\ve,\xi_\ve}^{p-1}dx\right)\notag\\
&&+\mathcal{O}\left(\frac{\kappa_\ve^{N}}{\mu_\ve^{\frac{N-2}{2}}}\int_{B_{d(\xi_\ve,\partial \Omega)}(\xi_\ve)}U_{\mu_\ve,\xi_\ve}^{p-1}\left|\Phi_{\mu_\ve,\xi_\ve}^1\right|dx\right)+\mathcal{O}\left(\frac{\kappa_\ve^{2N-4}}{\mu_\ve^{N-2}}\int_{B_{d(\xi_\ve,\partial \Omega)}(\xi_\ve)}U_{\mu_\ve,\xi_\ve}^{p-2}\left|\Phi_{\mu_\ve,\xi_\ve}^1\right|dx\right){\bf 1}_{p\geq2}\notag\\
&&+\mathcal{O}\left(\frac{\kappa_\ve^{N+2}}{\mu_\ve^{\frac{N+2}{2}}}\int_{B_{d(\xi_\ve,\partial \Omega)}(\xi_\ve)}\left|\Phi_{\mu_\ve,\xi_\ve}^1\right|dx\right)+\mathcal{O}\left(\kappa_\ve^N\right)\notag\\
&=&-\alpha_N\mu_\ve^{\frac{N-2}{2}}\int_{B_{d(\xi_\ve,\partial \Omega)}(\xi_\ve)}f'\left(U_{\mu_\ve,\xi_\ve}\right)H(x,\xi_\ve)\Phi_{\mu_\ve,\xi_\ve}^1dx+\mathcal{O}\left(\kappa_\ve^N\right),
\end{eqnarray*}
which, together with
\begin{eqnarray*}\label{IV1-02}
&&\int_{B_{d(\xi_\ve,\partial \Omega)}(\xi_\ve)}f'\left(U_{\mu_\ve,\xi_\ve}\right)H(x,\xi_\ve)\Phi_{\mu_\ve,\xi_\ve}^1dx\\
&=&D_{N,4}\mu_\ve^{\frac{N}{2}}\frac{\partial \varphi(\xi_\ve)}{\partial \xi_{1,\ve}}+\mathcal{O}\left(\int_{B_{d(\xi_\ve,\partial \Omega)}(\xi_\ve)}|x-\xi_\ve|^2U_{\mu_\ve,\xi_\ve}^{p-1}\left|\Phi_{\mu_\ve,\xi_\ve}^1\right|dx\right)\notag\\
&&+\mathcal{O}\left(\frac{\kappa_\ve^{N+1}}{\mu_\ve^{\frac{N-2}{2}}}\right)\notag\\
&=&D_{N,4}\mu_\ve^{\frac{N}{2}}\frac{\partial \varphi(\xi_\ve)}{\partial \xi_{1,\ve}}+\mathcal{O}\left(\frac{\kappa_\ve^{N+1}}{\mu_\ve^{\frac{N-2}{2}}}\right)
\end{eqnarray*}
and \eqref{k-mu}, implies that
\begin{eqnarray}\label{IV1}
IV_1=-\alpha_ND_{N,4}\mu_\ve^{N-1}\frac{\partial \varphi(\xi_\ve)}{\partial \xi_{1,\ve}}
+\mathcal{O}\left(\kappa_\ve^{N}\right),
\end{eqnarray}
where $D_{N,4}=p\int_{\mathbb{R}^N}U_{1,0}^{p-1}\Phi_{1,0}^1xdx$.
Inserting \eqref{I1}, \eqref{II1}, \eqref{III1} and \eqref{IV1} into \eqref{test-1}, we have
\begin{eqnarray*}
0=-\alpha_ND_{N,4}\mu_\ve^{N-1}\frac{\partial \varphi(\xi_\ve)}{\partial \xi_{1,\ve}}+
o\left(\kappa_\ve^{N-1}\right)+\left\{\aligned
&o\left(\mu_\ve^{2}|\log\mu_\ve|\right),\quad N=4,\\
&o\left(\mu_\ve^{2}\right),
\endaligned\right.
\end{eqnarray*}
which, together with \eqref{exp-rob} and \eqref{k-mu}, implies a contradiction.  Thus, $\xi_0\in\partial\Omega$ is also impossible and the conclusion follows immediately.
 \end{proof}

With the basic estimate of $\|v_\ve\|_{H^1_0(\Omega)}$ in hands, we can locate the limit value $\overline{\lambda}$.
\begin{lemma}\label{locate-limit}
Let $N\geq4$ and $v_\ve$ is the solution of \eqref{v-equ}, then there exists $k\geq1$ such that $\overline{\lambda}=\lambda_k$.  Moreover, let $\bar{v}_{\ve}=\frac{v_{\ve}}{\|v_{\ve}\|_{H_0^1(\Omega)}}$ then $\bar{v}_{\ve}\to\bar{v}_{0}$ weakly in $H^1_0(\Omega)$ as $\ve\to0$ up to a subsequence such that $\bar{v}_{0}\in \Xi_k$.
\end{lemma}
\begin{proof}
By testing \eqref{v-equ} by $\varphi \in C_0^\infty (\Omega)$, we have
\begin{eqnarray}\label{test-2}
\int_{\Omega} \nabla v_\ve \nabla \varphi dx -\int_{\Omega} \left(f\left(\mathcal{W}_{\mu_\ve,\xi_\ve}+v_{\ve}\right)-f\left(\mathcal{W}_{\mu_\ve,\xi_\ve}\right)\right)\varphi dx+\lambda\int_{\Omega}(\mathcal{W}_{\mu_\ve,\xi_\ve}+v_{\ve})\varphi dx=0.
\end{eqnarray}
Similar to \eqref{taylor-1} and by the H\"older inequality,
\begin{eqnarray}\label{test-21}
&&\left|\int_{\Omega} \left(f\left(\mathcal{W}_{\mu_\ve,\xi_\ve}+v_{\ve}\right)-f\left(\mathcal{W}_{\mu_\ve,\xi_\ve}\right)\right)\varphi dx\right|\notag\\
&\lesssim&
\|v_\ve\|_{H^1_0(\Omega)}^2{\bf 1}_{p\geq2}+\|v_\ve\|_{H^1_0(\Omega)}^p+\left\{\aligned
&\|v_\ve\|_{H^1_0(\Omega)}\mu_\ve^{\frac{N-2}{2}},\quad N=4,5,\\
&\|v_\ve\|_{H^1_0(\Omega)}\mu_\ve^2\left|\log\mu_\ve\right|^{\frac{2}{3}},\quad N=6,\\
&\|v_\ve\|_{H^1_0(\Omega)}\mu_\ve^2,\quad N\geq7.
\endaligned\right.
\end{eqnarray}
By Lemma~\ref{App-esti} and \eqref{comp}, 
\begin{eqnarray}\label{test-22}
\int_{\Omega}\mathcal{W}_{\mu_\ve,\xi_\ve}\varphi dx=\mathcal{O}\left(\mu_\ve^{\frac{N-2}{2}}\right).
\end{eqnarray}
Let $\bar{v}_{\ve}=\frac{v_{\ve}}{\|v_{\ve}\|_{H_0^1(\Omega)}}$ then $\{\bar{v}_{\ve}\}$ is bounded in $H_0^1(\Omega)$ and
\begin{eqnarray*}
\left\{\aligned
\bar{v}_{\ve}\rightharpoonup\bar{v}_{0}\quad\text{weakly in }H^1_0(\Omega),\\
\bar{v}_{\ve}\to\bar{v}_{0}\quad\text{strongly in }L^2(\Omega),
\endaligned\right.
\end{eqnarray*}
as $\ve\to0$ up to a subsequence.  Now, by inserting \eqref{test-21} and \eqref{test-22} into \eqref{test-2}, it is easy to see that $\bar{v}_{\ve}\to\bar{v}_{0}$ weakly in $H^1_0(\Omega)$ as $\ve\to0$ up to a subsequence and 
\begin{eqnarray*}
\int_{\Omega} \nabla \bar{v}_{0}\nabla \varphi dx= \overline{\lambda} \int_{\Omega} \bar{v}_{0} \varphi dx.
\end{eqnarray*}
Since we have assumed $\overline{\lambda}>0$ and $\varphi$ is arbitrary, we must have that $\overline{\lambda}=\lambda_k$.
\end{proof}

\section{The first refinement and the classification in the case $\lambda\to\overline{\lambda}^+$}
We recall that we have used the notation $\{e_{k,l}\}_{1\leq l\leq m_k}$ to denote the orthogonal system of the eigenspace according to $\lambda_k$ and $m_k$ is the multiplicity of $\lambda_k$.  Then it is well known that $\{e_{k,l}\}_{1\leq l\leq h_k; k\geq1}$ is an orthogonal basis in $H_0^1(\Omega)$.  Without loss of generality, we assume that $\|e_{k,l}\|_{H_0^1(\Omega)}^2=1$ for every $k\geq1$ and $1\leq j\leq m_k$.  Thus, we can rewrite
\begin{eqnarray}\label{decomp-v}
v_\ve=\sum_{k=1}^{\infty}\sum_{l=1}^{m_k}\tau_{k,l,\ve}e_{k,l}\text{ in }H^1_0(\Omega)\quad\text{and}\quad\|v_\ve\|_{H^1_0(\Omega)}^2=\sum_{k=1}^{\infty}\sum_{l=1}^{m_k}\tau_{k,l,\ve}^2,
\end{eqnarray}
where $\tau_{k,l,\ve}=\langle v_\ve, e_{k,l}\rangle$.  
By Lemma~\ref{locate-limit}, we can assume that $\overline{\lambda}=\lambda_k$ for some $k\geq1$.  Then we can also rewrite
\begin{eqnarray}\label{orthogonal-decomp}
v_\ve=v_{\ve}^*+\sum_{l=1}^{m_k}\tau_{k,l,\ve}'e_{k,l}+\sum_{j=0}^{N}\alpha_{j,\ve}Z_{\mu_\ve,\xi_\ve}^j\quad\text{ in }H^1_0(\Omega)
\end{eqnarray}
such that
\begin{eqnarray}\label{rem-orth}
\langle v_{\ve}^*, e_{k,l}\rangle=\langle v_{\ve}^*, Z_{\mu_\ve,\xi_\ve}^j\rangle=0
\end{eqnarray}
for all $1\leq l\leq m_k$ and $0\leq j\leq N$, which is equivalent to solve the following system
\begin{eqnarray}\label{system}
\left\{\aligned
&\tau_{k,l,\ve}=\tau_{k,l,\ve}'+\sum_{i=0}^{N}\alpha_{i,\ve}\langle Z_{\mu_\ve,\xi_\ve}^i, e_{k,l}\rangle,\quad 1\leq l\leq m_k,\\
&0=\sum_{l=1}^{m_k}\tau_{k,l,\ve}'\langle e_{k,l}, Z_{\mu_\ve,\xi_\ve}^j\rangle+\sum_{i=0}^{N}\alpha_{i,\ve}\langle Z_{\mu_\ve,\xi_\ve}^i, Z_{\mu_\ve,\xi_\ve}^j\rangle,\quad 0\leq j\leq N,
\endaligned\right.
\end{eqnarray}
according to the orthogonal condition of $v_\ve$ given by \eqref{orthogonality}, the orthogonality of the eigenfunctions $\{e_{k,l}\}$ and \eqref{decomp-v}.
\begin{lemma}\label{compute-ez}
Let $N\geq4$, $\{e_{k,l}\}_{1\leq l\leq m_k}$ is the orthogonal system of the eigenspace according to $\lambda_k$ and $Z_{\mu_\ve,\xi_\ve}^j$ are given by \eqref{exp-ker}, where $m_k$ is the multiplicity of $\lambda_k$, then
\begin{eqnarray*}
\langle e_{k,l}, Z_{\mu_\ve,\xi_\ve}^j\rangle=\left\{\aligned
&D_{N,3}\mu_\ve^{\frac{N-2}{2}}e_{k,l}(\xi_\ve)+\mathcal{O}\left(\mu_\ve^{\frac{N-2}{2}}\kappa_\ve^2\left|\log\kappa_\ve\right|\right),\quad j=0,\\
&D_{N,4}\mu_\ve^{\frac{N}{2}}\frac{\partial e_{k,l}(\xi_\ve)}{\partial \xi_{j,\ve}}+\mathcal{O}\left(\mu_\ve^{\frac{N-2}{2}}\kappa_\ve^3+\mu_\ve^{\frac{N+2}{2}}\right),\quad 1\leq j\leq N,
\endaligned\right.
\end{eqnarray*}
where $D_{N,3}=p\int_{\mathbb{R}^N}U_{1,0}^{p-1}\Phi_{1,0}^0dx$ and $D_{N,4}=p\int_{\mathbb{R}^N}U_{1,0}^{p-1}\Phi_{1,0}^1xdx$.
\end{lemma}
\begin{proof}
By \eqref{projkernel},
\begin{eqnarray*}
\langle e_{k,l}, Z_{\mu_\ve,\xi_\ve}^j\rangle=p\int_{\Omega}U_{\mu_\ve,\xi_\ve}^{p-1}\Psi_{\mu_\ve,\xi_\ve}^j e_{k,l}dx
\end{eqnarray*}
for all $0\leq j\leq N$, $k\geq1$ and $1\leq l\leq m_k$.  Thus, by the regularity of $e_{k,l}$, the Taylor expansion, \eqref{comp} and Lemma~\ref{App-esti},
\begin{eqnarray}\label{esti-ez-0}
\langle e_{k,l}, Z_{\mu_\ve,\xi_\ve}^0\rangle&=&p\int_{\Omega}U_{\mu_\ve,\xi_\ve}^{p-1}\Psi_{\mu_\ve,\xi_\ve}^0\left(e_{k,l}(\xi_\ve)+\nabla e_{k,l}(\xi_\ve)(x-\xi_\ve)+\mathcal{O}(|x-\xi_\ve|^2)\right)dx\notag\\
&=&D_{N,3}\mu_\ve^{\frac{N-2}{2}}e_{k,l}(\xi_\ve)+\mathcal{O}\left(\int_{\Omega\backslash B_{d(\xi_\ve,\partial \Omega)}(\xi_\ve)}U_{\mu_\ve,\xi_\ve}^{p}dx+\int_{B_{d(\xi_\ve,\partial \Omega)}(\xi_\ve)}U_{\mu_\ve,\xi_\ve}^{p}|x-\xi_\ve|^2dx\right)\notag\\
&=&D_{N,3}\mu_\ve^{\frac{N-2}{2}}e_{k,l}(\xi_\ve)+\mathcal{O}\left(\mu_\ve^{\frac{N-2}{2}}\kappa_\ve^2\left|\log\kappa_\ve\right|\right)
\end{eqnarray}
and
\begin{eqnarray*}
\langle e_{k,l}, Z_{\mu_\ve,\xi_\ve}^j\rangle&=&p\int_{\Omega}U_{\mu_\ve,\xi_\ve}^{p-1}\Psi_{\mu_\ve,\xi_\ve}^j\left(e_{k,l}(\xi_\ve)+\nabla e_{k,l}(\xi_\ve)(x-\xi_\ve)+\mathcal{O}(|x-\xi_\ve|^2)\right)dx\\
&=&D_{N,4}\mu_\ve^{\frac{N}{2}}\frac{\partial e_{k,l}(\xi_\ve)}{\partial \xi_{j,\ve}}+\mathcal{O}\left(\int_{\Omega\backslash B_{d(\xi_\ve,\partial \Omega)}(\xi_\ve)}U_{\mu_\ve,\xi_\ve}^{p-1}\left|\Psi_{\mu_\ve,\xi_\ve}^j\right|dx\right)\\
&&+\mathcal{O}\left(\int_{B_{d(\xi_\ve,\partial \Omega)}(\xi_\ve)}U_{\mu_\ve,\xi_\ve}^{p-1}\left|\Psi_{\mu_\ve,\xi_\ve}^j\right||x-\xi_\ve|^2dx\right)\\
&=&D_{N,4}\mu_\ve^{\frac{N}{2}}\frac{\partial e_{k,l}(\xi_\ve)}{\partial \xi_{j,\ve}}+\mathcal{O}\left(\mu_\ve^{\frac{N-2}{2}}\kappa_\ve^3+\mu_\ve^{\frac{N+2}{2}}\right)
\end{eqnarray*}
for $1\leq j\leq N$.
\end{proof}

Let
\begin{eqnarray}\label{beta}
\beta_\ve=\max_{k\geq1; 1\leq l\leq m_k}|\tau_{k,l,\ve}|.
\end{eqnarray}
Since by \eqref{projkernel} and Lemma~\ref{App-esti},
\begin{eqnarray}\label{zi-zj}
\langle Z_{\mu_\ve,\xi_\ve}^i, Z_{\mu_\ve,\xi_\ve}^j\rangle=p\int_{\Omega}U_{\mu_\ve,\xi_\ve}^{p-1}\Psi_{\mu_\ve,\xi_\ve}^j Z_{\mu_\ve,\xi_\ve}^idx=\left\{\aligned
&D_{N,5}+\mathcal{O}\left(\kappa_\ve^{N}\right),\quad i=j=0,\\
&D_{N,6}+\mathcal{O}\left(\kappa_\ve^{N+2}\right),\quad i=j\not=0,\\
&\mathcal{O}\left(\kappa_\ve^{N+1}\right),\quad i\not=j,
\endaligned\right.
\end{eqnarray}
$0\leq i,j\leq N$, where $D_{N,5}=p\int_{\mathbb{R}^N}U_{1,0}^{p-1}\left|\Phi_{1,0}^0\right|^2dx$ and $D_{N,6}=p\int_{\mathbb{R}^N}U_{1,0}^{p-1}\left|\Phi_{1,0}^1\right|^2dx$.  we can solve \eqref{system} by
\begin{eqnarray*}
\tau_{k,l,\ve}'=\tau_{k,l,\ve}+\sum_{i=0}^{N}\mathcal{O}\left(\mu_\ve^{\frac{N-2}{2}}|\alpha_{i,\ve}|\right)\quad\text{and}\quad|\alpha_{i,\ve}|=\sum_{l=1}^{m_{k}}\mathcal{O}\left(\mu_\ve^{\frac{N-2}{2}}|\tau_{k,l,\ve}'|\right),
\end{eqnarray*}
which, together with \eqref{beta}, implies that
\begin{eqnarray}\label{solution-sys}
\tau_{k,l,\ve}'=\tau_{k,l,\ve}+\mathcal{O}\left(\mu_\ve^{N-2}\beta_\ve\right)\quad\text{and}\quad|\alpha_{i,\ve}|=\mathcal{O}\left(\mu_\ve^{\frac{N-2}{2}}\beta_\ve\right)
\end{eqnarray}
for all $k\geq1$, $1\leq l\leq m_k$ and $0\leq i\leq N$.

\vskip0.12in

By \eqref{orthogonal-decomp}, we can rewrite the equation of $v_\ve$ given by \eqref{v-equ} into the equation of $v_{\ve}^*$ as follows:
\begin{eqnarray}\label{v*-equ}
\left\{\aligned
&-\Delta v_{\ve}^*-\lambda v_{\ve}^*-f'(\mathcal{W}_{\mu_\ve,\xi_\ve}^*)v_{\ve}^*=\mathcal{R}_0(v_{\ve}^*)+\sum_{j=1}^2\mathcal{R}_j,\quad\text{in }\Omega,\\
&\langle v_\ve^*, \mathcal{Z}_{\mu_\ve, \xi_\ve}^i\rangle=\langle v_\ve^*, e_{k,l}\rangle=0,\quad 0\leq i\leq N,\quad1\leq l\leq m_k,\\
&v_\ve^*=0,\quad\text{on }\partial\Omega,
\endaligned\right.
\end{eqnarray}
where we denote
\begin{eqnarray}\label{new*}
\mathcal{E}_{\ve}^*=\sum_{l=1}^{m_k}\tau_{k,l,\ve}'e_{k,l},\quad\mathcal{Z}_\ve=\sum_{j=0}^{N}\alpha_{j,\ve}Z_{\mu_\ve,\xi_\ve}^j\quad\text{and}\quad\mathcal{W}_{\mu_\ve,\xi_\ve}^*=\mathcal{W}_{\mu_\ve,\xi_\ve}+\mathcal{Z}_\ve,
\end{eqnarray}
\begin{eqnarray}\label{Rv}
\mathcal{R}_0(v_{\ve}^*)&=&f(\mathcal{W}_{\mu_\ve,\xi_\ve}^*+\mathcal{E}_{\ve}^*+v_{\ve}^*)-f(\mathcal{W}_{\mu_\ve,\xi_\ve}^*+\mathcal{E}_{\ve}^*)-f'(\mathcal{W}_{\mu_\ve,\xi_\ve}^*+\mathcal{E}_{\ve}^*)v_{\ve}^*\notag\\
&&+\left(f'(\mathcal{W}_{\mu_\ve,\xi_\ve}^*+\mathcal{E}_{\ve}^*)-f'(\mathcal{W}_{\mu_\ve,\xi_\ve}^*)\right)v_{\ve}^*
\end{eqnarray}
and
\begin{eqnarray}\label{R}
\left\{\aligned
&\mathcal{R}_1=f(\mathcal{W}_{\mu_\ve,\xi_\ve})-f(U_{\mu_\ve,\xi_\ve})+\lambda \mathcal{W}_{\mu_\ve,\xi_\ve},\\
&\mathcal{R}_2=f(\mathcal{W}_{\mu_\ve,\xi_\ve}^*+\mathcal{E}_{\ve}^*)-f(\mathcal{W}_{\mu_\ve,\xi_\ve})+\ve\mathcal{E}_{\ve}^*+\lambda\mathcal{Z}_\ve-\sum_{j=0}^{N}\alpha_{j,\ve}f'\left(U_{\mu_\ve,\xi_\ve}\right)\Phi_{\mu_\ve,\xi_\ve}^j.
\endaligned\right.
\end{eqnarray}
By \eqref{kernel}, \eqref{point-z1}, \eqref{solution-sys} and the comparsion,
\begin{eqnarray}\label{esti-W*}
\mathcal{W}_{\mu_\ve,\xi_\ve}^*\sim\left(1+\mathcal{O}\left(\mu_\ve^{\frac{N-2}{2}}\beta_\ve\right)\right)U_{\mu_\ve,\xi_\ve}.
\end{eqnarray}
\begin{lemma}\label{esti-v*}
Let $N\geq4$ and $v_\ve^*$ is the solution of \eqref{v*-equ}.  Then
\begin{eqnarray*}
\|v_\ve^*\|_{H^1_0(\Omega)}\lesssim\ve\beta_\ve+\beta_\ve^{p}+\left\{\aligned
&\kappa_\ve^{N-2}+\mu_\ve^{\frac{N-2}{2}},\quad N=4,5,\\
&\kappa_\ve^4|\log\kappa_\ve|^{\frac{2}{3}}+\mu_\ve^2\left|\log\mu_\ve\right|^{\frac{2}{3}},\quad N=6,\\
&\kappa_\ve^{\frac{N+2}{2}}+\mu_\ve^2,\quad N\geq7
\endaligned\right.
\end{eqnarray*}
and
\begin{eqnarray*}
\|v_\ve^*\|_{L^\infty (\Omega)} \lesssim  \beta_\ve+\mu_\ve^{-\frac{N-2}{2}}\kappa_\ve^{N-2}+\mu_\ve^{\frac{6-N}{2}}.
\end{eqnarray*}
\end{lemma}
\begin{proof}
By Lemma~\ref{Lem2.1}, 
\begin{eqnarray}\label{eqa001}
\|v_\ve^*\|_{H^1_0(\Omega)}\lesssim\left\|\mathcal{R}_0(v_{\ve}^*)\right\|_{L^{\frac{p+1}{p}}(\Omega)}+\sum_{j=1}^2\left\|\mathcal{R}_j\right\|_{L^{\frac{p+1}{p}}(\Omega)}.
\end{eqnarray}
Similar to \eqref{taylor-1}, by \eqref{Rv} and \eqref{esti-W*},
\begin{eqnarray}\label{R0-1}
\left|\mathcal{R}_0(v_{\ve}^*)\right|
&\lesssim&\left(\left|v_{\ve}^*\right|+\beta_\ve \right)U_{\mu_\ve,\xi_\ve}^{p-2}\left|v_{\ve}^*\right|{\bf 1}_{(\Omega\backslash\mathcal{A}_2)\cap (\Omega\backslash\mathcal{A}_2')}+\beta_\ve^{p-1}\left|v_{\ve}^*\right|{\bf 1}_{(\Omega\backslash\mathcal{A}_2)\cap \mathcal{A}_2'}\notag\\
&&+\left(U_{\mu_\ve,\xi_\ve}^{p-2}\left|v_{\ve}^*\right|^2+\left|v_{\ve}^*\right|^{p}\right){\bf 1}_{\mathcal{A}_2\cap (\Omega\backslash\mathcal{A}_2')}+\beta_\ve^{p-1}\left|v_{\ve}^*\right|{\bf 1}_{\mathcal{A}_2\cap \mathcal{A}_2'}\notag\\
&\lesssim&\beta_\ve U_{\mu_\ve,\xi_\ve}^{p-2}\left|v_{\ve}^*\right|{\bf 1}_{(\Omega\backslash\mathcal{A}_2)\cap (\Omega\backslash\mathcal{A}_2')}+U_{\mu_\ve,\xi_\ve}^{p-2}\left|v_{\ve}^*\right|^2{\bf 1}_{\Omega\backslash\mathcal{A}_2'}+\beta_\ve^{p-1}\left|v_{\ve}^*\right|{\bf 1}_{\mathcal{A}_2'}+\left|v_{\ve}^*\right|^{p}{\bf 1}_{\mathcal{A}_2}
\end{eqnarray}
where
\begin{eqnarray}\label{A2}
\left\{\aligned
&\mathcal{A}_2=\left\{x \in \Omega\mid \left|\mathcal{W}_{\mu_\ve,\xi_\ve}^*(x)+\mathcal{E}_\ve^*(x)\right|\lesssim\left|v_{\ve}^*(x)\right|\right\},\\
&\mathcal{A}_2'=\left\{x \in \Omega\mid \left|\mathcal{W}_{\mu_\ve,\xi_\ve}^*(x)\right|\lesssim\left|\mathcal{E}_\ve^*(x)\right|\right\}.
\endaligned\right.
\end{eqnarray}
Thus,
\begin{eqnarray*}
\left\|\mathcal{R}_0(v_{\ve}^*)\right\|_{L^{\frac{p+1}{p}}(\Omega)}=o(\|v_\ve^*\|_{H^1_0(\Omega)}).
\end{eqnarray*}
Similarly, by \eqref{exp-bub} and \eqref{R},
\begin{eqnarray}\label{R1-exp}
&&\left|\mathcal{R}_1-f'(U_{\mu_\ve,\xi_\ve})\alpha_N\mu_\ve^{\frac{N-2}{2}}H(x,\xi_\ve)-\lambda\mathcal{W}_{\mu_\ve,\xi_\ve}\right|\notag\\
&\lesssim&\left\{\aligned
&\frac{\kappa_\ve^{2N-4}}{\mu_\ve^{N-2}}U_{\mu_\ve,\xi_\ve}^{p-2}{\bf 1}_{p\geq2}+\frac{\kappa_\ve^{N+2}}{\mu_\ve^{\frac{N+2}{2}}},\quad x\in B_{d(\xi_\ve,\partial \Omega)}(\xi_\ve),\\
&U_{\mu_\ve,\xi_\ve}^{p},\quad x\in \Omega\backslash B_{d(\xi_\ve,\partial \Omega)}(\xi_\ve)
\endaligned\right.
\end{eqnarray}
and by \eqref{kernel}, \eqref{point-z1}, \eqref{solution-sys}, \eqref{new*}, \eqref{R} and \eqref{esti-W*},
\begin{eqnarray}\label{R2-exp}
\left|\mathcal{R}_2-f(\mathcal{E}_{\ve}^*)\right|\lesssim \beta_\ve U_{\mu_\ve,\xi_\ve}^{p-1}+\beta_\ve^{p-1} U_{\mu_\ve,\xi_\ve}+\mu_\ve^{\frac{N-2}{2}}\beta_\ve U_{\mu_\ve,\xi_\ve}+\ve\beta_\ve.
\end{eqnarray}
It follows from Lemma~\ref{App-esti} that
\begin{eqnarray}\label{R1-esti-0}
\left\|\mathcal{R}_1\right\|_{L^{\frac{p+1}{p}}(\Omega)}\lesssim\left\{\aligned
&\kappa_\ve^{N-2}+\mu_\ve^{\frac{N-2}{2}},\quad N=4,5,\\
&\kappa_\ve^4|\log\kappa_\ve|^{\frac{2}{3}}+\mu_\ve^2\left|\log\mu_\ve\right|^{\frac{2}{3}},\quad N=6,\\
&\kappa_\ve^{\frac{N+2}{2}}+\mu_\ve^2,\quad N\geq7
\endaligned\right.
\end{eqnarray}
and
\begin{eqnarray*}
\left\|\mathcal{R}_2\right\|_{L^{\frac{p+1}{p}}(\Omega)}\lesssim\ve\beta_\ve+\beta_\ve^{p}+\left\{\aligned
&\beta_\ve\mu_\ve^{\frac{N-2}{2}}+\mu_\ve^{N-2}\beta_\ve,\quad N=4,5,\\
&\beta_\ve\mu_\ve^2|\log\mu_\ve|^{\frac{2}{3}},\quad N=6,\\
&\beta_\ve^{p-1}\mu_\ve^2,\quad N\geq7.
\endaligned\right.
\end{eqnarray*}
Inserting the above estimates into \eqref{eqa001}, we have the conclusion of $\|v_\ve^*\|_{H^1_0(\Omega)}$.  For the estimate of $\|v_\ve^*\|_{L^\infty(\Omega)}$, we notice that by \eqref{kernel}, \eqref{point-z1}, \eqref{beta} and \eqref{solution-sys},
\begin{eqnarray*}
\left|\sum_{l=1}^{m_k}\tau_{k,l,\ve}'e_{k,l}+\sum_{j=0}^{N}\alpha_{j,\ve}Z_{\mu_\ve,\xi_\ve}^j\right|\lesssim\beta_\ve
\end{eqnarray*}
in $\Omega$.  Thus, the estimate of $\|v_\ve^*\|_{L^\infty(\Omega)}$ follows from  \eqref{orthogonal-decomp} and Lemma~\ref{L-infinity}.
\end{proof}

By \eqref{decomp-v}, \eqref{orthogonal-decomp}, \eqref{rem-orth} and Lemma~\ref{compute-ez},
\begin{eqnarray*}
\sum_{s=1}^{\infty}\sum_{l=1}^{m_s}\tau_{s,l,\ve}^2=\|v_\ve^*\|_{H^1_0(\Omega)}^2+\sum_{l=1}^{m_k}\left|\tau_{k,l,\ve}'\right|^2+\mathcal{O}\left(\sum_{i=0}^{N}|\alpha_{i,\ve}|^2+\sum_{i=0}^{N}\sum_{l=1}^{m_k}\mu_\ve^{\frac{N-2}{2}}|\alpha_{i,\ve}\tau_{k,l,\ve}'|\right),
\end{eqnarray*}
which, together with \eqref{beta}, \eqref{solution-sys} and Lemmas~\ref{basic-esti}, \ref{locate-limit} and \ref{esti-v*}, implies that
\begin{eqnarray*}
\sum_{s=1; s\not=k}^{\infty}\sum_{l=1}^{m_s}\tau_{s,l,\ve}^2=o(\beta_\ve^2)+\left\{\aligned
&\mathcal{O}\left(\kappa_\ve^{N-2}\right),\quad N=4,5,\\
&\mathcal{O}\left(\kappa_\ve^4|\log\kappa_\ve|^{\frac{2}{3}}+\mu_\ve^2\left|\log\mu_\ve\right|^{\frac{2}{3}}\right),\quad N=6,\\
&\mathcal{O}\left(\kappa_\ve^{\frac{N+2}{2}}+\mu_\ve^2\right),\quad N\geq7.
\endaligned\right.
\end{eqnarray*}
It follows that 
\begin{eqnarray}\label{beta-new}
\beta_\ve^*:=\max_{1\leq l\leq m_k}|\tau_{k,l,\ve}|=\beta_\ve+\left\{\aligned
&\mathcal{O}\left(\kappa_\ve^{N-2}\right),\quad N=4,5,\\
&\mathcal{O}\left(\kappa_\ve^4|\log\kappa_\ve|^{\frac{2}{3}}+\mu_\ve^2\left|\log\mu_\ve\right|^{\frac{2}{3}}\right),\quad N=6,\\
&\mathcal{O}\left(\kappa_\ve^{\frac{N+2}{2}}+\mu_\ve^2\right),\quad N\geq7.
\endaligned\right.
\end{eqnarray}
Now, with Lemma~\ref{esti-v*} in hands, we can expand the orthogonal condition in \eqref{v*-equ} up to the leading order terms.
\begin{proposition}\label{exp-orth}
Let $N\geq4$, then
\begin{enumerate}
\item[$(1)$] we have
\begin{eqnarray*}
0&\geq&\left\{\aligned
&\lambda D_{N,1}\mu_\ve^2-\alpha_N\mu_\ve^{N-2}D_{N,3}\varphi(\xi_\ve)+\left(1+\frac{\ve}{\lambda_k}\right)D_{N,3}\mu_\ve^{\frac{N-2}{2}}\mathcal{E}_{\ve}^*(\xi_\ve)+o\left(\mu_\ve^2+\kappa_\ve^{N-2}\right),\quad N\geq5,\\
&\lambda D_{4,2}\mu_\ve^2\left|\log\kappa_\ve\right|-\alpha_4\mu_\ve^{2}D_{4,3}\varphi(\xi_\ve)+\left(1+\frac{\ve}{\lambda_k}\right)D_{4,3}\mu_\ve\mathcal{E}_{\ve}^*(\xi_\ve)+o\left(\kappa_\ve^{2}\right)+\mathcal{O}\left(\mu_\ve^2\right), \quad N=4,
\endaligned\right.\\
&&+\mathcal{O}\left((\ve\beta_\ve)^2+\beta_\ve^{2p}\right)
\end{eqnarray*}
and
\begin{eqnarray*}
0&\leq&\left\{\aligned
&\lambda D_{N,1}\mu_\ve^2-\alpha_N\mu_\ve^{N-2}D_{N,3}\varphi(\xi_\ve)+\left(1+\frac{\ve}{\lambda_k}\right)D_{N,3}\mu_\ve^{\frac{N-2}{2}}\mathcal{E}_{\ve}^*(\xi_\ve)+o\left(\mu_\ve^2+\kappa_\ve^{N-2}\right),\quad N\geq5,\\
&\lambda D_{4,2}\mu_\ve^2\left|\log\mu_\ve\right|-\alpha_4\mu_\ve^{2}D_{4,3}\varphi(\xi_\ve)+\left(1+\frac{\ve}{\lambda_k}\right)D_{4,3}\mu_\ve\mathcal{E}_{\ve}^*(\xi_\ve)+o\left(\kappa_\ve^{2}\right)+\mathcal{O}\left(\mu_\ve^2\right), \quad N=4,
\endaligned\right.\\
&&+\mathcal{O}\left((\ve\beta_\ve)^2+\beta_\ve^{2p}\right).
\end{eqnarray*}
\item[$(2)$] For all $1\leq j\leq N$, we have
\begin{eqnarray*}\label{zj}
0&=&\left(1+\frac{\ve}{\lambda_k}\right)D_{N,4}\mu_\ve^{\frac{N}{2}}\frac{\partial \mathcal{E}_{\ve}^*(\xi_\ve)}{\partial \xi_{i,\ve}}-\alpha_ND_{N,4}\mu_\ve^{N-1}\frac{\partial \varphi(\xi_\ve)}{\partial \xi_{i,\ve}}+\mathcal{O}\left(\beta_\ve\mu_\ve^{\frac{N-2}{2}}\kappa_\ve^3\left|\log\kappa_\ve\right|+\mu_\ve^{2}\kappa_\ve^{N-3}\right)+\mathcal{O}\left(\beta_\ve^{2p}\right)\\
&&+o\left(\kappa_\ve^{N-1}+\beta_\ve\mu_\ve^{\frac{N}{2}}\right)
+\left\{\aligned
&\mathcal{O}\left(\left(\ve\beta_\ve\right)^3+\kappa_\ve^5|\log d(\xi_\ve,\partial \Omega)|\right),\quad N=4,\\
&o\left(\left(\ve\beta_\ve\right)^2+\mu_\ve^3\right)+\mathcal{O}\left(\kappa_\ve^7\mu_\ve^{-1}\right),\quad N=5,\\
&\mathcal{O}\left(\left(\ve\beta_\ve\right)^2\right)+o\left(\mu_\ve^3\right),\quad N\geq6.
\endaligned\right.
\end{eqnarray*}
\item[$(3)$] For all $1\leq l\leq m_k$, we have
\begin{eqnarray*}\label{el}
0&=&\left(1+\frac{\ve}{\lambda_k}\right)D_{N,6}\mu_\ve^{\frac{N-2}{2}}e_{k,l}(\xi_\ve)+\int_{\Omega}f\left(\mathcal{E}_\ve^*\right)e_{k,l}dx+\ve\int_{\Omega}\mathcal{E}_\ve^*e_{k,l}dx+\mathcal{O}\left(\mu_\ve^{\frac{N-2}{2}}\kappa_\ve^2|\log\mu_\ve|\right)\notag\\
&&+o\left(\ve\beta_\ve+\beta_\ve^p+\kappa_\ve^{N-2}\right)
+\left\{\aligned
&o\left(\mu_\ve^{2}|\log\mu_\ve|\right),\quad N=4,\\
&o\left(\mu_\ve^{2}\right),\quad N=5,6,\\
&o\left(\mu_\ve^{2}\right)+\mathcal{O}\left(\kappa_\ve^{\frac{(N+2)p}{2}}\right),\quad N\geq7.
\endaligned\right.
\end{eqnarray*}
\end{enumerate}
\end{proposition}
\begin{proof}
$(1)$\quad We test \eqref{v*-equ} with $Z_{\mu_\ve,\xi_\ve}^0$, then by the orthogonality of $v_\ve^*$, we have
\begin{eqnarray}\label{equa0002}
0&=&\sum_{j=1}^2\int_{\Omega}\mathcal{R}_jZ_{\mu_\ve,\xi_\ve}^0dx+\int_{\Omega}\mathcal{R}_0(v_\ve^*)Z_{\mu_\ve,\xi_\ve}^0dx+\lambda\int_{\Omega}v_\ve^*Z_{\mu_\ve,\xi_\ve}^0dx+\int_{\Omega}f'(\mathcal{W}_{\mu_\ve,\xi_\ve}^*)v_{\ve}^*Z_{\mu_\ve,\xi_\ve}^0dx\notag\\
&:=&I+II+III+IV.
\end{eqnarray}
By \eqref{R}, the estimate of $\int_{\Omega}\mathcal{R}_1Z_{\mu_\ve,\xi_\ve}^0dx$ is the same as that of $III$ and $IV$ in the proof of Lemma~\ref{basic-esti}, which implies that
\begin{eqnarray}\label{R10+}
\int_{\Omega}\mathcal{R}_1Z_{\mu_\ve,\xi_\ve}^0dx=\lambda D_{N,1}\mu_\ve^2-\alpha_N\mu_\ve^{N-2}D_{N,3}\varphi(\xi_\ve)+\mathcal{O}\left(\kappa_\ve^N\left|\log\kappa_\ve\right|\right)+\mathcal{O}\left(\mu_\ve^2\kappa_\ve^{N-4}\right)
\end{eqnarray}
for $N\geq5$,
\begin{eqnarray}\label{R10-1}
\int_{\Omega}\mathcal{R}_1Z_{\mu_\ve,\xi_\ve}^0dx\geq\lambda D_{4,2}\mu_\ve^2\left|\log\kappa_\ve\right|-\alpha_4\mu_\ve^{2}D_{4,3}\varphi(\xi_\ve)+\mathcal{O}\left(\kappa_\ve^4\left|\log\kappa_\ve\right|\right)+\mathcal{O}\left(\mu_\ve^2\right)
\end{eqnarray}
and
\begin{eqnarray}\label{R10-2}
\int_{\Omega}\mathcal{R}_1Z_{\mu_\ve,\xi_\ve}^0dx\leq\lambda D_{4,2}\mu_\ve^2\left|\log\mu_\ve\right|-\alpha_4\mu_\ve^{2}D_{4,3}\varphi(\xi_\ve)+\mathcal{O}\left(\kappa_\ve^4\left|\log\kappa_\ve\right|\right)+\mathcal{O}\left(\mu_\ve^2\right)
\end{eqnarray}
for $N=4$.  Let
\begin{eqnarray*}\label{A3}
\mathcal{A}_3=\left\{x \in \Omega\mid \left|\mathcal{W}_{\mu_\ve,\xi_\ve}(x)\right|\lesssim\left|\mathcal{E}_\ve^*(x)+\mathcal{Z}_{\ve}(x)\right|\right\}.
\end{eqnarray*}
Then similar to \eqref{taylor-1}, by \eqref{solution-sys} and \eqref{esti-W*},
\begin{eqnarray}\label{fw}
f(\mathcal{W}_{\mu_\ve,\xi_\ve}^*+\mathcal{E}_{\ve}^*)-f(\mathcal{W}_{\mu_\ve,\xi_\ve})=\left\{\aligned
&f'\left(\mathcal{W}_{\mu_\ve,\xi_\ve}\right)\left(\mathcal{E}_{\ve}^*+\mathcal{Z}_{\ve}\right)+\mathcal{O}\left(\beta_{\ve}^2U_{\mu_\ve,\xi_\ve}^{p-2}\right),\quad x\in\Omega\backslash\mathcal{A}_3,\\
&\mathcal{O}\left(\beta_{\ve}^p\right),\quad x\in\mathcal{A}_3.
\endaligned\right.
\end{eqnarray}
It follows from \eqref{point-z1}, \eqref{R} and Lemma~\ref{App-esti} that
\begin{eqnarray*}
\int_{\Omega}\mathcal{R}_2Z_{\mu_\ve,\xi_\ve}^0dx&=&\int_{\Omega}f'\left(\mathcal{W}_{\mu_\ve,\xi_\ve}\right)\left(\mathcal{E}_{\ve}^*+\mathcal{Z}_{\ve}\right)Z_{\mu_\ve,\xi_\ve}^0dx-\sum_{j=0}^{N}\alpha_{j,\ve}\int_{\Omega}f'\left(U_{\mu_\ve,\xi_\ve}\right)\Phi_{\mu_\ve,\xi_\ve}^jZ_{\mu_\ve,\xi_\ve}^0dx\\
&&+\int_{\Omega}\left(\ve\mathcal{E}_{\ve}^*+\lambda\mathcal{Z}_\ve\right)Z_{\mu_\ve,\xi_\ve}^0dx+\mathcal{O}\left(\widetilde{\tau}_{k,\ve}^p\mu_\ve^{\frac{N-2}{2}}\right)+\left\{\aligned
&\mathcal{O}\left(\beta_{\ve}^2\mu_\ve^2|\log\mu_\ve|\right),\quad N=4,\\
&\mathcal{O}\left(\beta_{\ve}^2\mu_\ve^2\right),\quad N\geq5.
\endaligned\right.
\end{eqnarray*}
Since by \eqref{exp-bub}, \eqref{comp} and similar expanding as \eqref{taylor-1},
\begin{eqnarray}\label{fw1}
f'\left(\mathcal{W}_{\mu_\ve,\xi_\ve}\right)=\left\{\aligned
&f'\left(U_{\mu_\ve,\xi_\ve}\right)+\mathcal{O}\left(\frac{\kappa_\ve^{N-2}}{\mu_\ve^{\frac{N-2}{2}}}f''\left(U_{\mu_\ve,\xi_\ve}\right){\bf 1}_{p\geq2}+\frac{\kappa_\ve^{4}}{\mu_\ve^{2}}\right),\quad x\in B_{d(\xi_\ve,\partial \Omega)}(\xi_\ve),\\
&\mathcal{O}\left(U_{\mu_\ve,\xi_\ve}^{p-1}\right),\quad x\in \Omega\backslash B_{d(\xi_\ve,\partial \Omega)}(\xi_\ve),
\endaligned\right.
\end{eqnarray}
by \eqref{projkernel}, \eqref{exp-ker} , \eqref{system}, \eqref{solution-sys} and Lemma~\ref{App-esti},
\begin{eqnarray*}
\int_{\Omega}f'\left(\mathcal{W}_{\mu_\ve,\xi_\ve}\right)\left(\mathcal{E}_{\ve}^*+\mathcal{Z}_{\ve}\right)Z_{\mu_\ve,\xi_\ve}^0dx=\int_{\Omega}f'\left(U_{\mu_\ve,\xi_\ve}\right)\left(\mathcal{E}_{\ve}^*+\mathcal{Z}_{\ve}\right)Z_{\mu_\ve,\xi_\ve}^0dx+o\left(\kappa_\ve^{N-2}\right)+\mathcal{O}\left(\beta_{\ve}\mu_\ve^{\frac{N-2}{2}}\kappa_\ve^2\right),
\end{eqnarray*}
which, together with similar computations in the proof of Lemma~\ref{compute-ez}, \eqref{exp-ker}, \eqref{system}, \eqref{zi-zj} and Lemma~\ref{App-esti}, implies that
\begin{eqnarray*}
\int_{\Omega}pU_{\mu_\ve,\xi_\ve}^{p-1}\mathcal{E}_{\ve}^*Z_{\mu_\ve,\xi_\ve}^0dx&=&\int_{\Omega}pU_{\mu_\ve,\xi_\ve}^{p-1}\Phi_{\mu_\ve,\xi_\ve}^0\mathcal{E}_{\ve}^*dx+\mathcal{O}\left(\frac{\beta_{\ve}\kappa_{\ve}^{N-2}}{\mu_\ve^{\frac{N-2}{2}}}\int_{B_{d(\xi_\ve,\partial \Omega)}(\xi_\ve)}U_{\mu_\ve,\xi_\ve}^{p-1}dx\right)\\
&&+\mathcal{O}\left(\beta_{\ve}\int_{\Omega\backslash B_{d(\xi_\ve,\partial \Omega)}(\xi_\ve)}U_{\mu_\ve,\xi_\ve}^{p}dx\right)\\
&=&D_{N,3}\mu_\ve^{\frac{N-2}{2}}\mathcal{E}_{\ve}^*(\xi_\ve)+\mathcal{O}\left(\beta_{\ve}\mu_\ve^{\frac{N-2}{2}}\kappa_\ve^2\left|\log\kappa_\ve\right|\right)
\end{eqnarray*}
and
\begin{eqnarray*}
\int_{\Omega}pU_{\mu_\ve,\xi_\ve}^{p-1}\mathcal{Z}_{\ve}Z_{\mu_\ve,\xi_\ve}^0dx&=&-D_{N,3}\mu_\ve^{\frac{N-2}{2}}\mathcal{E}_{\ve}^*(\xi_\ve)+\mathcal{O}\left(\beta_\ve\mu_\ve^{\frac{N-2}{2}}\kappa_\ve^{N-2}\right).
\end{eqnarray*}
Thus,
\begin{eqnarray*}
\int_{\Omega}f'\left(\mathcal{W}_{\mu_\ve,\xi_\ve}\right)\left(\mathcal{E}_{\ve}^*+\mathcal{Z}_{\ve}\right)Z_{\mu_\ve,\xi_\ve}^0dx=o\left(\kappa_\ve^{N-2}\right)+\mathcal{O}\left(\beta_{\ve}\mu_\ve^{\frac{N-2}{2}}\kappa_\ve^2\left|\log\kappa_\ve\right|\right).
\end{eqnarray*}
Similarly, by \eqref{projkernel}, \eqref{system}, \eqref{solution-sys} and Lemma~\ref{compute-ez},
\begin{eqnarray*}
\sum_{j=0}^{N}\alpha_{j,\ve}\int_{\Omega}f'\left(U_{\mu_\ve,\xi_\ve}\right)\Phi_{\mu_\ve,\xi_\ve}^jZ_{\mu_\ve,\xi_\ve}^0dx&=&-\sum_{l=1}^{m_k}\tau_{k,l,\ve}'\langle e_{k,l}, Z_{\mu_\ve,\xi_\ve}^0\rangle\\
&=&-D_{N,3}\mu_\ve^{\frac{N-2}{2}}\mathcal{E}_{\ve}^*(\xi_\ve)+\mathcal{O}\left(\beta_{\ve}\mu_\ve^{\frac{N-2}{2}}\kappa_\ve^2\left|\log\kappa_\ve\right|\right)
\end{eqnarray*}
and
\begin{eqnarray*}
\ve\int_{\Omega}\mathcal{E}_{\ve}^*Z_{\mu_\ve,\xi_\ve}^0dx=\sum_{l=1}^{m_k}\frac{\ve\tau_{k,l,\ve}'}{\lambda_k}\langle e_{k,l}, Z_{\mu_\ve,\xi_\ve}^0\rangle=\frac{\ve D_{N,3}}{\lambda_k}\mu_\ve^{\frac{N-2}{2}}\mathcal{E}_{\ve}^*(\xi_\ve)+\mathcal{O}\left(\beta_{\ve}\mu_\ve^{\frac{N-2}{2}}\kappa_\ve^2\left|\log\kappa_\ve\right|\right).
\end{eqnarray*}
By \eqref{comp}, \eqref{point-z1}, \eqref{new*} and Lemma~\ref{App-esti},
\begin{eqnarray}\label{equa0003}
\int_{\Omega}\mathcal{Z}_\ve Z_{\mu_\ve,\xi_\ve}^0dx&=&\mathcal{O}\left(\beta_\ve\mu_\ve^{\frac{N-2}{2}}\int_{\Omega}\left|\Phi_{\mu_\ve,\xi_\ve}^{1}\right| U_{\mu_\ve,\xi_\ve}+U_{\mu_\ve,\xi_\ve}^2dx\right)\notag\\
&=&\left\{\aligned
&\mathcal{O}\left(\beta_\ve\mu_\ve^{3}|\log\mu_\ve|\right),\quad N=4,\\
&\mathcal{O}\left(\beta_\ve\mu_\ve^{\frac{N+2}{2}}\right),\quad N\geq5.
\endaligned\right.
\end{eqnarray}
Summarizing the above estimates, we have
\begin{eqnarray}\label{R2-0}
\int_{\Omega}\mathcal{R}_2Z_{\mu_\ve,\xi_\ve}^0dx&=&\left(1+\frac{\ve}{\lambda_k}\right)D_{N,3}\mu_\ve^{\frac{N-2}{2}}\mathcal{E}_{\ve}^*(\xi_\ve)+\mathcal{O}\left(\widetilde{\tau}_{k,\ve}\mu_\ve^{\frac{N-2}{2}}\kappa_\ve^2\left|\log\kappa_\ve\right|\right)+o\left(\kappa_\ve^{N-2}\right)\notag\\
&&+\mathcal{O}\left(\beta_\ve^p\mu_\ve^{\frac{N-2}{2}}\right)+\left\{\aligned
&\mathcal{O}\left(\beta_{\ve}^2\mu_\ve^2|\log\mu_\ve|\right),\quad N=4,\\
&o\left(\mu_\ve^2\right),\quad N\geq5.
\endaligned\right.
\end{eqnarray}
By \eqref{comp}, \eqref{new*}, \eqref{esti-W*}, \eqref{R0-1}, \eqref{A2} and Lemma~\ref{esti-v*},
\begin{eqnarray}\label{R0-0}
\int_{\Omega}\mathcal{R}_0(v_\ve^*)Z_{\mu_\ve,\xi_\ve}^0dx&=&\mathcal{O}\left(\|v_\ve^*\|_{H^1_0(\Omega)}^2+\beta_\ve^{p}\|v_\ve^*\|_{H^1_0(\Omega)}+\beta_\ve\int_{\Omega}U_{\mu_\ve,\xi_\ve}^{p-1}\left|v_\ve^*\right|dx\right)\notag\\
&=&\mathcal{O}\left((\ve\beta_\ve)^2+\beta_\ve^{2p}\right)+o\left(\kappa_\ve^{N-2}\right)+\left\{\aligned
&\mathcal{O}\left(\mu_\ve^{2}\right),\quad N=4,\\
&o\left(\mu_\ve^{2}\right),\quad N\geq5.
\endaligned\right.
\end{eqnarray}
By \eqref{comp}, Lemmas~\ref{esti-v*} and \ref{App-esti},
\begin{eqnarray}\label{III-10}
\lambda\int_{\Omega}v_\ve^*Z_{\mu_\ve,\xi_\ve}^0dx&=&
\mathcal{O}\left(\|v_\ve^*\|_{H_0^1 (\Omega)}\left(\int_{\Omega}U_{\mu_\ve,\xi_\ve}^{\frac{p+1}{p}}dx\right)^{\frac{p}{p+1}}\right)\notag\\
&=&\left\{\aligned
&\mathcal{O}\left((\ve\beta_\ve)^2+\beta_\ve^{2p}+\mu_\ve^2\right)+o\left(\kappa_\ve^{2}\right),\quad N=4,\\
&\mathcal{O}\left((\ve\beta_\ve)^2+\beta_\ve^{2p}\right)+o\left(\mu_\ve^2+\kappa_\ve^{3}\right),\quad N=5,\\
&o\left(\mu_\ve^2+\kappa_\ve^{4}\right),\quad N=6,\\
&o\left(\mu_\ve^2\right),\quad N\geq7.
\endaligned\right.
\end{eqnarray}
By the orthogonal condition of $v_{\ve}^*$ given in \eqref{v*-equ}, \eqref{decay-H},  similar estimates in \eqref{I2-0} and \eqref{fw1} and Lemma~\ref{esti-v*},
\begin{eqnarray}\label{IV-10}
\int_{\Omega}f'(\mathcal{W}_{\mu_\ve,\xi_\ve}^*)v_{\ve}^*Z_{\mu_\ve,\xi_\ve}^0dx&=&\int_{\Omega}\left(f'(\mathcal{W}_{\mu_\ve,\xi_\ve}^*)-f'(U_{\mu_\ve,\xi_\ve})\right)v_{\ve}^*\Phi_{\mu_\ve,\xi_\ve}^0dx\notag\\
&&+\mathcal{O}\left(\|v_\ve^*\|_{H^1_0(\Omega)}\frac{\kappa_\ve^{N-2}}{\mu_\ve^{\frac{N-2}{2}}}\left(\int_{B_{d(\xi_\ve,\partial \Omega)}(\xi_\ve)}U_{\mu_\ve,\xi_\ve}^{\frac{p^2-1}{p}}dx\right)^{\frac{p}{p+1}}\right)\notag\\
&&+\mathcal{O}\left(\|v_\ve^*\|_{H^1_0(\Omega)}\left(\int_{\Omega\backslash B_{d(\xi_\ve,\partial \Omega)}(\xi_\ve)}U_{\mu_\ve,\xi_\ve}^{p+1}dx\right)^{\frac{p}{p+1}}\right)\notag\\
&=&\mathcal{O}\left(\int_{\Omega}\left(\mu_\ve^{\frac{N-2}{2}}\beta_\ve +\left(\mu_\ve^{\frac{N-2}{2}}\beta_\ve\right)^{p-1}\right) U_{\mu_\ve,\xi_\ve}^{p}|v_{\ve}^*|dx\right)\notag\\
&&+\mathcal{O}\left(\|v_\ve^*\|_{H^1_0(\Omega)}\frac{\kappa_\ve^{N-2}}{\mu_\ve^{\frac{N-2}{2}}}\left(\int_{B_{d(\xi_\ve,\partial \Omega)}(\xi_\ve)}U_{\mu_\ve,\xi_\ve}^{\frac{p^2-1}{p}}dx\right)^{\frac{p}{p+1}}\right)\notag\\
&&+\mathcal{O}\left(\|v_\ve^*\|_{H^1_0(\Omega)}\left(\int_{\Omega\backslash B_{d(\xi_\ve,\partial \Omega)}(\xi_\ve)}U_{\mu_\ve,\xi_\ve}^{p+1}dx\right)^{\frac{p}{p+1}}\right)\notag\\
&=&\left\{\aligned
&\mathcal{O}\left(\left(\beta_\ve\mu_\ve^{\frac{N-2}{2}}+\kappa_\ve^{N-2}\right)\|v_\ve^*\|_{H^1_0(\Omega)}\right),\quad N=4,5,\\
&\mathcal{O}\left(\left(\beta_\ve\mu_\ve^{2}+\kappa_\ve^4|\log\kappa_\ve|^{\frac{2}{3}}\right)\|v_\ve^*\|_{H^1_0(\Omega)}\right),\quad N=6,\\
&\mathcal{O}\left(\left(\beta_\ve^{p-1}\mu_\ve^{2}+\kappa_\ve^{\frac{N+2}{2}}\right)\|v_\ve^*\|_{H^1_0(\Omega)}\right),\quad N\geq7,
\endaligned\right.\notag\\
&=&\mathcal{O}\left((\ve\beta_\ve)^2+\beta_\ve^{2p}\right)+o\left(\mu_\ve^2+\kappa_\ve^{N-2}\right).
\end{eqnarray}
Inserting \eqref{R10+}, \eqref{R10-1}, \eqref{R10-2}, \eqref{R2-0}, \eqref{R0-0}, \eqref{III-10} and \eqref{IV-10} into \eqref{equa0002}, we have the conclusions.

\vskip0.12in

$(2)$\quad We test \eqref{v*-equ} with $Z_{\mu_\ve,\xi_\ve}^i$, then by the orthogonality of $v_\ve^*$, we have
\begin{eqnarray}\label{equa0004}
0&=&\sum_{j=1}^2\int_{\Omega}\mathcal{R}_jZ_{\mu_\ve,\xi_\ve}^idx+\int_{\Omega}\mathcal{R}_0(v_\ve^*)Z_{\mu_\ve,\xi_\ve}^idx+\lambda\int_{\Omega}v_\ve^*Z_{\mu_\ve,\xi_\ve}^idx+\int_{\Omega}f'(\mathcal{W}_{\mu_\ve,\xi_\ve}^*)v_{\ve}^*Z_{\mu_\ve,\xi_\ve}^idx\notag\\
&:=&I_1+II_1+III_1+IV_1.
\end{eqnarray}
Again, by \eqref{R}, the estimate of $\int_{\Omega}\mathcal{R}_1Z_{\mu_\ve,\xi_\ve}^idx$ is the same as that of $III_1$ and $IV_1$ in the proof of Lemma~\ref{basic-esti}, which implies that
\begin{eqnarray}\label{R1-i}
\int_{\Omega}\mathcal{R}_1Z_{\mu_\ve,\xi_\ve}^idx&=&-\alpha_ND_{N,4}\mu_\ve^{N-1}\frac{\partial \varphi(\xi_\ve)}{\partial \xi_{i,\ve}}+\mathcal{O}\left(\mu_\ve^{2}\kappa_\ve^{N-3}\right)\notag\\
&&+o\left(\kappa_\ve^{N-1}\right)+\left\{\aligned
&\mathcal{O}\left(\kappa_\ve\mu_\ve^{2}|\log d(\xi_\ve,\partial \Omega)|\right),\quad N=4,\\
&\mathcal{O}\left(\mu_\ve^{N-2}\kappa_\ve\right),\quad N\geq5.
\endaligned\right.
\end{eqnarray}
By \eqref{point-z1}, \eqref{R}, \eqref{fw} and Lemma~\ref{App-esti},
\begin{eqnarray*}
\int_{\Omega}\mathcal{R}_2Z_{\mu_\ve,\xi_\ve}^idx&=&\int_{\Omega}f'\left(\mathcal{W}_{\mu_\ve,\xi_\ve}\right)\left(\mathcal{E}_{\ve}^*+\mathcal{Z}_{\ve}\right)Z_{\mu_\ve,\xi_\ve}^idx-\sum_{j=0}^{N}\alpha_{j,\ve}\int_{\Omega}f'\left(U_{\mu_\ve,\xi_\ve}\right)\Phi_{\mu_\ve,\xi_\ve}^jZ_{\mu_\ve,\xi_\ve}^idx\\
&&+\int_{\Omega}\left(\ve\mathcal{E}_{\ve}^*+\lambda\mathcal{Z}_\ve\right)Z_{\mu_\ve,\xi_\ve}^idx+\mathcal{O}\left(\beta_{\ve}^p\mu_\ve^{\frac{N+2}{2}}\kappa_\ve^{-1}\right)\\
&&+\left\{\aligned
&\mathcal{O}\left(\beta_{\ve}^2\mu_\ve^2(1+\kappa_\ve|\log d(\xi_\ve,\partial \Omega)|)\right),\quad N=4,\\
&\mathcal{O}\left(\beta_{\ve}^2\mu_\ve^2(\mu_\ve|\log\kappa_\ve|+\kappa_\ve)\right),\quad N=5,\\
&\mathcal{O}\left(\beta_{\ve}^2\mu_\ve^2\kappa_\ve\right),\quad N=6.
\endaligned\right.
\end{eqnarray*}
By \eqref{point-z1}, \eqref{fw1} and Lemma~\ref{App-esti},
\begin{eqnarray*}
\int_{\Omega}f'\left(\mathcal{W}_{\mu_\ve,\xi_\ve}\right)\left(\mathcal{E}_{\ve}^*+\mathcal{Z}_{\ve}\right)Z_{\mu_\ve,\xi_\ve}^idx=\int_{\Omega}f'\left(U_{\mu_\ve,\xi_\ve}\right)\left(\mathcal{E}_{\ve}^*+\mathcal{Z}_{\ve}\right)Z_{\mu_\ve,\xi_\ve}^idx+o\left(\kappa_\ve^{N-1}\right)+\mathcal{O}\left(\beta_{\ve}\mu_\ve^{\frac{N-2}{2}}\kappa_\ve^3\right).
\end{eqnarray*}
Again, by \eqref{exp-ker} and similar computations in the proof of Lemma~\ref{compute-ez},
\begin{eqnarray*}
\int_{\Omega}pU_{\mu_\ve,\xi_\ve}^{p-1}\mathcal{E}_{\ve}^*Z_{\mu_\ve,\xi_\ve}^idx&=&\int_{\Omega}pU_{\mu_\ve,\xi_\ve}^{p-1}\Phi_{\mu_\ve,\xi_\ve}^i\mathcal{E}_{\ve}^*dx+\mathcal{O}\left(\frac{\beta_{\ve}\kappa_{\ve}^{N-1}}{\mu_\ve^{\frac{N-2}{2}}}\int_{B_{d(\xi_\ve,\partial \Omega)}(\xi_\ve)}U_{\mu_\ve,\xi_\ve}^{p-1}dx\right)\\
&&+\mathcal{O}\left(\beta_{\ve}\kappa_\ve\int_{\Omega\backslash B_{d(\xi_\ve,\partial \Omega)}(\xi_\ve)}U_{\mu_\ve,\xi_\ve}^{p}dx\right)\\
&=&D_{N,4}\mu_\ve^{\frac{N}{2}}\frac{\partial \mathcal{E}_{\ve}^*(\xi_\ve)}{\partial \xi_{i,\ve}}+\mathcal{O}\left(\beta_{\ve}\mu_\ve^{\frac{N-2}{2}}\kappa_\ve^3\left|\log\kappa_\ve\right|+\beta_{\ve}\mu_\ve^{\frac{N+2}{2}}\right)
\end{eqnarray*}
and
\begin{eqnarray*}
\int_{\Omega}pU_{\mu_\ve,\xi_\ve}^{p-1}\left(\sum_{j=0}^{N}\alpha_{j,\ve}Z_{\mu_\ve,\xi_\ve}^j\right)Z_{\mu_\ve,\xi_\ve}^idx&=&-D_{N,4}\mu_\ve^{\frac{N}{2}}\frac{\partial \mathcal{E}_{\ve}^*(\xi_\ve)}{\partial \xi_{i,\ve}}+\mathcal{O}\left(\beta_{\ve}\mu_\ve^{\frac{N-2}{2}}\kappa_\ve^3\left|\log\kappa_\ve\right|+\beta_{\ve}\mu_\ve^{\frac{N+2}{2}}\right).
\end{eqnarray*}
Thus,
\begin{eqnarray*}
\int_{\Omega}f'\left(U_{\mu_\ve,\xi_\ve}\right)\left(\mathcal{E}_{\ve}^*+\mathcal{Z}_{\ve}\right)Z_{\mu_\ve,\xi_\ve}^idx=\mathcal{O}\left(\beta_{\ve}\mu_\ve^{\frac{N-2}{2}}\kappa_\ve^3\left|\log\kappa_\ve\right|+\beta_{\ve}\mu_\ve^{\frac{N+2}{2}}\right).
\end{eqnarray*}
Similarly, by \eqref{system} and Lemma~\ref{compute-ez},
\begin{eqnarray*}
\sum_{j=0}^{N}\alpha_{j,\ve}\int_{\Omega}f'\left(U_{\mu_\ve,\xi_\ve}\right)\Phi_{\mu_\ve,\xi_\ve}^jZ_{\mu_\ve,\xi_\ve}^idx&=&-\sum_{l=1}^{m_k}\tau_{k,l,\ve}'\langle e_{k,l}, Z_{\mu_\ve,\xi_\ve}^i\rangle\\
&=&-D_{N,4}\mu_\ve^{\frac{N}{2}}\frac{\partial \mathcal{E}_{\ve}^*(\xi_\ve)}{\partial \xi_{i,\ve}}+\mathcal{O}\left(\beta_{\ve}\mu_\ve^{\frac{N-2}{2}}\kappa_\ve^3\left|\log\kappa_\ve\right|+\beta_{\ve}\mu_\ve^{\frac{N+2}{2}}\right)
\end{eqnarray*}
and
\begin{eqnarray*}
\ve\int_{\Omega}\mathcal{E}_{\ve}^*Z_{\mu_\ve,\xi_\ve}^idx=\sum_{l=1}^{m_k}\frac{\ve\tau_{k,l,\ve}'}{\lambda_k}\langle e_{k,l}, Z_{\mu_\ve,\xi_\ve}^i\rangle=\frac{\ve D_{N,4}}{\lambda_k}\mu_\ve^{\frac{N}{2}}\frac{\partial \mathcal{E}_{\ve}^*(\xi_\ve)}{\partial \xi_{i,\ve}}+\mathcal{O}\left(\beta_{\ve}\mu_\ve^{\frac{N-2}{2}}\kappa_\ve^3\left|\log\kappa_\ve\right|+\beta_{\ve}\mu_\ve^{\frac{N+2}{2}}\right).
\end{eqnarray*}
Similar to the computations in \eqref{equa0003},
\begin{eqnarray*}
\int_{\Omega}\mathcal{Z}_\ve Z_{\mu_\ve,\xi_\ve}^idx&=&\mathcal{O}\left(\beta_\ve\mu_\ve^{\frac{N-2}{2}}\left(\int_{B_{d(\xi_\ve,\partial \Omega)}(\xi_\ve)}\left|\Phi_{\mu_\ve,\xi_\ve}^{1}\right| U_{\mu_\ve,\xi_\ve}dx+\kappa_\ve\int_{\Omega\backslash B_{d(\xi_\ve,\partial \Omega)}(\xi_\ve)}U_{\mu_\ve,\xi_\ve}^2dx\right)\right)\\
&=&\left\{\aligned
&\mathcal{O}\left(\beta_\ve\mu_\ve^{3}\left(1+\kappa_\ve|\log d(\xi_\ve,\partial \Omega)|\right)\right),\quad N=4,\\
&\mathcal{O}\left(\beta_\ve^2\mu_\ve^{\frac{N+2}{2}}(1+\kappa_\ve^{N-3})\right),\quad N\geq5.
\endaligned\right.
\end{eqnarray*}
Summarizing the above estimates, we have
\begin{eqnarray}\label{R2-i}
\int_{\Omega}\mathcal{R}_2Z_{\mu_\ve,\xi_\ve}^idx&=&\left(1+\frac{\ve}{\lambda_k}\right)D_{N,4}\mu_\ve^{\frac{N}{2}}\frac{\partial \mathcal{E}_{\ve}^*(\xi_\ve)}{\partial \xi_{i,\ve}}+\mathcal{O}\left(\beta_{\ve}\mu_\ve^{\frac{N-2}{2}}\kappa_\ve^3\left|\log\kappa_\ve\right|+\beta_{\ve}\mu_\ve^{\frac{N+2}{2}}\right)\notag\\
&&+\mathcal{O}\left(\beta_{\ve}^p\mu_\ve^{\frac{N+2}{2}}\kappa_\ve^{-1}\right)+o\left(\kappa_\ve^{N-1}+\beta_\ve\mu_\ve^{\frac N2}\right)\notag\\
&&+\left\{\aligned
&\mathcal{O}\left(\beta_{\ve}^2\mu_\ve^3|\log\kappa_\ve|\right),\quad N=5,\\
&\mathcal{O}\left(\beta_{\ve}^2\mu_\ve^2\kappa_\ve\right),\quad N=6.
\endaligned\right.
\end{eqnarray}
By \eqref{comp}, \eqref{point-z1}, Lemmas~\ref{esti-v*} and \ref{App-esti} and similar to \eqref{R0-0},
\begin{eqnarray}\label{R0-i-5}
\int_{\Omega}\mathcal{R}_0(v_\ve^*)Z_{\mu_\ve,\xi_\ve}^idx&=&\mathcal{O}\left(\|v_\ve^*\|_{H^1_0(\Omega)}^2+\beta_\ve^{p}\|v_\ve^*\|_{H^1_0(\Omega)}+\beta_\ve\int_{\Omega}U_{\mu_\ve,\xi_\ve}^{p-2}\left|Z_{\mu_\ve,\xi_\ve}^iv_\ve^*\right|dx\right)\notag\\
&=&\mathcal{O}\left((\ve\beta_\ve)^2+\beta_\ve^{2p}\right)+o\left(\kappa_\ve^{N-1}+\beta_\ve\mu_\ve^{\frac N2}\right)\notag\\
&&+\left\{\aligned
&\mathcal{O}\left(\mu_\ve^{4}|\log\mu_\ve|^{\frac{4}{3}}\right),\quad N=6,\\
&\mathcal{O}\left(\mu_\ve^{4}\right),\quad N\geq7
\endaligned\right.
\end{eqnarray}
for $N\geq6$.  For $N=4,5$, by \eqref{comp}, \eqref{point-z1}, \eqref{new*}, \eqref{esti-W*}, \eqref{R0-1}, \eqref{A2} and Lemma~\ref{esti-v*},
\begin{eqnarray}\label{R0-i-4}
\int_{\Omega}\mathcal{R}_0(v_\ve^*)Z_{\mu_\ve,\xi_\ve}^idx&=&\mathcal{O}\left(\|v_\ve^*\|_{L^\infty(\Omega)}^2\left(\int_{B_{d(\xi_\ve,\partial \Omega)}(\xi_\ve)}U_{\mu_\ve,\xi_\ve}^{p-2}\left|\Phi_{\mu_\ve,\xi_\ve}^i\right|dx+\kappa_\ve\int_{\Omega\backslash B_{d(\xi_\ve,\partial \Omega)}(\xi_\ve)}U_{\mu_\ve,\xi_\ve}^{p-1}dx\right)\right)\notag\\
&&+\mathcal{O}\left(\int_{B_{d(\xi_\ve,\partial \Omega)}(\xi_\ve)}\left(\beta_\ve^{p-1}+|v_\ve^*|^{p-1}\right)|v_\ve^*|\left|\Phi_{\mu_\ve,\xi_\ve}^i\right|dx\right)\notag\\
&&+\mathcal{O}\left(\kappa_\ve\int_{\Omega\backslash B_{d(\xi_\ve,\partial \Omega)}(\xi_\ve)}\left(\beta_\ve^{p-1}+|v_\ve^*|^{p-1}\right)|v_\ve^*|U_{\mu_\ve,\xi_\ve}dx\right)\notag\\
&=&o\left(\mu_\ve^{\frac{N}{2}}\beta_\ve+\kappa_\ve^{N-1}\right)+\left\{\aligned
&\mathcal{O}\left(\kappa_\ve^5|\log d(\xi_\ve,\partial \Omega)|\right),\quad N=4,\\
&\mathcal{O}\left(\kappa_\ve^7\mu_\ve^{-1}\right), \quad N=5.
\endaligned\right.
\end{eqnarray}
By \eqref{point-z1} and similar to \eqref{III-10}, 
\begin{eqnarray}\label{III-1i}
&&\lambda\int_{\Omega}v_\ve^*Z_{\mu_\ve,\xi_\ve}^idx\notag\\
&=&
\mathcal{O}\left(\|v_\ve^*\|_{H_0^1 (\Omega)}\left(\left(\int_{B_{d(\xi_\ve,\partial \Omega)}(\xi_\ve)}\left|\Phi_{\mu_\ve,\xi_\ve}^i\right|^{\frac{p+1}{p}}dx\right)^{\frac{p}{p+1}}+\kappa_\ve\left(\int_{\Omega\backslash B_{d(\xi_\ve,\partial \Omega)}(\xi_\ve)}U_{\mu_\ve,\xi_\ve}^{\frac{p+1}{p}}dx\right)^{\frac{p}{p+1}}\right)\right)\notag\\
&=&\mathcal{O}\left((\ve\beta_\ve)^2{\bf 1}_{N\geq6}+\beta_\ve^{2p}\right)+o\left(\kappa_\ve^{N-1}\right)+\left\{\aligned
&o\left(\mu_\ve^{2}\beta_\ve\right),\quad N=4,\\
&o\left((\ve\beta_\ve)^2+\mu_\ve^3\right),\quad N=5,\\
&o\left(\mu_\ve^3\right),\quad N\geq7.
\endaligned\right.
\end{eqnarray}
By \eqref{point-z1}, Lemma~\ref{esti-v*} and similar to \eqref{IV-10}, 
\begin{eqnarray}\label{IV-1i}
&&\int_{\Omega}f'(\mathcal{W}_{\mu_\ve,\xi_\ve}^*)v_{\ve}^*Z_{\mu_\ve,\xi_\ve}^idx\notag\\
&=&\int_{\Omega}\left(f'(\mathcal{W}_{\mu_\ve,\xi_\ve}^*)-f'(U_{\mu_\ve,\xi_\ve})\right)v_{\ve}^*\Phi_{\mu_\ve,\xi_\ve}^idx+\mathcal{O}\left(\kappa_\ve^{\frac{N+4}{2}}\|v_\ve^*\|_{H^1_0(\Omega)}\right)\notag\\
&&+\left\{\aligned
&\mathcal{O}\left(\|v_\ve^*\|_{L^\infty(\Omega)}\frac{\kappa_\ve^{N-1}}{\mu_\ve^{\frac{N-2}{2}}}\int_{B_{d(\xi_\ve,\partial \Omega)}(\xi_\ve)}U_{\mu_\ve,\xi_\ve}^{p-1}dx\right),\quad N=4,5,\\
&\mathcal{O}\left(\|v_\ve^*\|_{H^1_0(\Omega)}\frac{\kappa_\ve^{N-1}}{\mu_\ve^{\frac{N-2}{2}}}\left(\int_{B_{d(\xi_\ve,\partial \Omega)}(\xi_\ve)}U_{\mu_\ve,\xi_\ve}^{\frac{p^2-1}{p}}dx\right)^{\frac{p}{p+1}}\right),\quad N\geq6,
\endaligned\right.\notag\\
&=&\mathcal{O}\left(\int_{\Omega}\left(\beta_\ve \mu_\ve^{\frac{N-2}{2}}+\left(\beta_\ve\mu_\ve^{\frac{N-2}{2}}\right)^{p-1}\right)U_{\mu_\ve,\xi_\ve}^{p-1}|\Phi_{\mu_\ve,\xi_\ve}^iv_{\ve}^*|dx\right)\notag\\
&&+\mathcal{O}\left(\kappa_\ve^{\frac{N+4}{2}}|\log\kappa_\ve|^{\frac{2}{3}}\|v_\ve^*\|_{H^1_0(\Omega)}+\mu_\ve^{\frac{N-2}{2}}\kappa_\ve^3\beta_\ve+\mu_\ve^2\kappa_\ve^3\right)+o\left(\kappa_\ve^{N-1}+\beta_\ve\mu_\ve^{\frac{N}{2}}\right)\notag\\
&&+\left\{\aligned
&\mathcal{O}\left(\|v_\ve^*\|_{L^\infty(\Omega)}\int_{B_{d(\xi_\ve,\partial \Omega)}(\xi_\ve)}\frac{\kappa_\ve^{N-2}}{\mu_\ve^{\frac{N-2}{2}}}\left|\Phi_{\mu_\ve,\xi_\ve}^i\right|U_{\mu_\ve,\xi_\ve}^{p-2}dx\right),\quad N=4,5,\\
&\mathcal{O}\left(\|v_\ve^*\|_{L^\infty(\Omega)}\int_{B_{d(\xi_\ve,\partial \Omega)}(\xi_\ve)}\frac{\kappa_\ve^{4}}{\mu_\ve^{2}}|\Phi_{\mu_\ve,\xi_\ve}^i|dx\right),\quad N\geq6,
\endaligned\right.\notag\\
&=&\mathcal{O}\left((\ve\beta_\ve)^2{\bf 1}_{N\geq6}+\beta_\ve^{2p}\right)+\left\{\aligned
&\mathcal{O}\left(\mu_\ve^2\kappa_\ve^3\right),\quad N=5,6,\\
&\mathcal{O}\left(\beta_\ve^{p-1}\mu_\ve^{4}+\beta_\ve\mu_\ve^{\frac{N-2}{2}}\kappa_\ve^3\right),\quad N\geq7.
\endaligned\right.
\end{eqnarray}
Inserting \eqref{R1-i}, \eqref{R2-i}, \eqref{R0-i-5}, \eqref{R0-i-4}, \eqref{III-1i} and \eqref{IV-1i} into \eqref{equa0004}, we obtain the conclusion.
\vskip0.12in

$(3)$\quad We test \eqref{v*-equ} with $e_{k,l}$, then by the orthogonality of $v_\ve^*$, we have
\begin{eqnarray}\label{equa0005}
0&=&\sum_{j=1}^2\int_{\Omega}\mathcal{R}_je_{k,l}dx+\int_{\Omega}\mathcal{R}_0(v_\ve^*)e_{k,l}dx+\int_{\Omega}f'(\mathcal{W}_{\mu_\ve,\xi_\ve}^*)v_{\ve}^*e_{k,l}dx\notag\\
&:=&I_2+II_2+III_2.
\end{eqnarray}
By \eqref{H-esti}, \eqref{R1-exp} and Lemma~\ref{App-esti},
\begin{eqnarray*}
\int_{\Omega}\mathcal{R}_1e_{k,l}dx&=&\lambda\int_{\Omega}\mathcal{W}_{\mu_\ve,\xi_\ve}e_{k,l}dx-\alpha_N\mu_\ve^{\frac{N-2}{2}}\int_{\Omega}f'\left(U_{\mu_\ve,\xi_\ve}\right)H(x,\xi_\ve)e_{k,l}dx\\
&&+\mathcal{O}\left(\int_{B_{d(\xi_\ve,\partial \Omega)}(\xi_\ve)}\mu_\ve^{N-2}U_{\mu_\ve,\xi_\ve}^{p-2}\left|H(x,\xi_\ve)\right|^2{\bf 1}_{p\geq2}+\mu_\ve^{\frac{N+2}{2}}\left|H(x,\xi_\ve)\right|^{p}dx\right)\\
&&+\mathcal{O}\left(\int_{\Omega\backslash B_{d(\xi_\ve,\partial \Omega)}(\xi_\ve)}U_{\mu_\ve,\xi_\ve}^{p}dx\right)\\
&=&\lambda\int_{\Omega}\mathcal{W}_{\mu_\ve,\xi_\ve}e_{k,l}dx-\alpha_N\mu_\ve^{\frac{N-2}{2}}\int_{\Omega}f'\left(U_{\mu_\ve,\xi_\ve}\right)H(x,\xi_\ve)e_{k,l}dx+\mathcal{O}\left(\mu_\ve^{\frac{N-2}{2}}\kappa_\ve^2\right).
\end{eqnarray*}
By \eqref{Projection} and similar estimate in \eqref{esti-ez-0},
\begin{eqnarray}
\lambda\int_{\Omega}\mathcal{W}_{\mu_\ve,\xi_\ve}e_{k,l}dx&=&\frac{\lambda}{\lambda_k}\int_{\Omega}U_{\mu_\ve,\xi_\ve}^{p}e_{k,l}dx\notag\\
&=&\frac{\lambda}{\lambda_k}\int_{\Omega}U_{\mu_\ve,\xi_\ve}^{p}\left(e_{k,l}(\xi_\ve)+\nabla e_{k,l}(\xi_\ve)\cdot(x-\xi_\ve)+\mathcal{O}\left(|x-\xi_\ve|^2\right) \right)dx\notag\\
&=&\left(1+\frac{\ve}{\lambda_k}\right)D_{N,6}\mu_\ve^{\frac{N-2}{2}}e_{k,l}(\xi_\ve)+\mathcal{O}\left(\int_{\Omega\backslash B_{d(\xi_\ve,\partial \Omega)}(\xi_\ve)}U_{\mu_\ve,\xi_\ve}^{p}dx\right)\notag\\
&&+\mathcal{O}\left(\int_{B_{d(\xi_\ve,\partial \Omega)}(\xi_\ve)}U_{\mu_\ve,\xi_\ve}^{p}|x-\xi_\ve|^2dx\right)\notag\\
&=&\left(1+\frac{\ve}{\lambda_k}\right)D_{N,6}\mu_\ve^{\frac{N-2}{2}}e_{k,l}(\xi_\ve)+\mathcal{O}\left(\mu_\ve^{\frac{N-2}{2}}\kappa_\ve^2\left|\log \kappa_\ve\right|\right),\label{compute-0}
\end{eqnarray}
where $D_{N,6}=\int_{\mathbb{R}^N}U_{1,0}^{p}dx$.
By \eqref{H-esti} and Lemma~\ref{App-esti},
\begin{eqnarray*}
\int_{\Omega}f'\left(U_{\mu_\ve,\xi_\ve}\right)H(x,\xi_\ve)e_{k,l}dx=\mathcal{O}\left(\kappa_\ve^2\mu_\ve^{\frac{N-2}{2}}|\log\mu_\ve|\right).
\end{eqnarray*}
Thus,
\begin{eqnarray}\label{R1-e}
\int_{\Omega}\mathcal{R}_1e_{k,l}dx
=\left(1+\frac{\ve}{\lambda_k}\right)D_{N,6}\mu_\ve^{\frac{N-2}{2}}e_{k,l}(\xi_\ve)+\mathcal{O}\left(\mu_\ve^{\frac{N-2}{2}}\kappa_\ve^2|\log\mu_\ve|\right).
\end{eqnarray}
By \eqref{R}, \eqref{R2-exp} and Lemmas~\ref{basic-esti} and \ref{App-esti},
\begin{eqnarray}\label{R2-e}
\int_{\Omega}\mathcal{R}_2e_{k,l}dx
&=&\int_{\Omega}f\left(\mathcal{E}_\ve^*\right)e_{k,l}dx+\ve\int_{\Omega}\mathcal{E}_\ve^*e_{k,l}dx+\mathcal{O}\left(\int_{\Omega}\left(\beta_\ve U_{\mu_\ve,\xi_\ve}^{p-1}+\left(\beta_\ve^{p-1}+\mu_\ve^{\frac{N-2}{2}}\beta_\ve\right)U_{\mu_\ve,\xi_\ve}dx\right)\right)\notag\\
&=&\int_{\Omega}f\left(\mathcal{E}_\ve^*\right)e_{k,l}dx+\ve\int_{\Omega}\mathcal{E}_\ve^*e_{k,l}dx+o\left(\beta_\ve^p\right)+\left\{\aligned
&o\left(\mu_\ve^2|\log\mu_\ve|\right),\quad N=4,\\
&o\left(\mu_\ve^{2}\right),\quad N\geq 5.
\endaligned\right.
\end{eqnarray}
By \eqref{R0-1} and Lemmas~\ref{basic-esti} and \ref{App-esti},
\begin{eqnarray}\label{R0-e}
\int_{\Omega}\mathcal{R}_0(v_\ve^*)e_{k,l}dx&=&\mathcal{O}\left(\|v_\ve^*\|_{H^1_0(\Omega)}^{2\wedge p}+\beta_\ve^{p-1}\|v_\ve^*\|_{H^1_0(\Omega)}+\beta_\ve\|v_\ve^*\|_{H^1_0(\Omega)}^{p-1}\right)\notag\\
&=&o(\ve\beta_\ve+\beta_\ve^{p}+\kappa_\ve^{N-2})+\left\{\aligned
&\mathcal{O}\left(\mu_\ve^{2}\right),\quad N=4,\\
&o\left(\mu_\ve^{2}\right),\quad N=5,6,\\
&o\left(\mu_\ve^{2}\right)+\mathcal{O}\left(\kappa_\ve^{\frac{(N+2)p}{2}}\right),\quad N\geq7.
\endaligned\right.
\end{eqnarray}
Similar to \eqref{IV-10}, by \eqref{esti-W*} and Lemmas~\ref{basic-esti} and \ref{App-esti},
\begin{eqnarray}\label{III2-e}
\int_{\Omega}f'(\mathcal{W}_{\mu_\ve,\xi_\ve}^*)v_{\ve}^*e_{k,l}dx&=&\mathcal{O}\left(\int_{\Omega}U_{\mu_\ve,\xi_\ve}^{p-1}\left|v_{\ve}^*\right|dx\right)\notag\\
&=&\left\{\aligned
&\mathcal{O}\left(\mu_\ve^2|\log\mu_\ve|\|v_\ve^*\|_{L^\infty(\Omega)}\right),\quad N=4,\\
&\mathcal{O}\left(\mu_\ve^3\|v_\ve^*\|_{L^\infty(\Omega)}\right),\quad N=5,\\
&\mathcal{O}\left(\mu_\ve^2|\log\mu_\ve|^{\frac{2}{3}}\|v_\ve^*\|_{H^1_0(\Omega)}\right),\quad N=6,\\
&\mathcal{O}\left(\mu_\ve^2\|v_\ve^*\|_{H^1_0(\Omega)}\right),\quad N\geq7,
\endaligned\right.\notag\\
&=&o(\ve\beta_\ve+\beta_\ve^{p}+\kappa_\ve^{N-2})+\left\{\aligned
&\mathcal{O}\left(\beta_\ve\mu_\ve^{2}|\log \mu_\ve|\right),\quad N=4,\\
&o\left(\mu_\ve^{2}\right),\quad N\geq5.
\endaligned\right.
\end{eqnarray}
Inserting \eqref{R1-e}, \eqref{R2-e}, \eqref{R0-e} and \eqref{III2-e} into \eqref{equa0005}, we have the conclusion.
\end{proof}

\vskip0.12in

We are in the position to prove the first classification result.

\vskip0.06in

\noindent\textbf{Proof of Theorem~\ref{zero weak limit+}:} \quad
By multiplying the equation in $(3)$ of Proposition~\ref{exp-orth} with $\tau_{k,l,\ve}'$ and summarizing them from $l=1$ to $l=m_k$, we have
\begin{eqnarray}\label{el}
0&=&\left(1+\frac{\ve}{\lambda_k}\right)D_{N,6}\mu_\ve^{\frac{N-2}{2}}\mathcal{E}_\ve^*(\xi_\ve)+\|\mathcal{E}_\ve^*\|_{L^{p+1}(\Omega)}^{p+1}+\ve\|\mathcal{E}_\ve^*\|_{L^2(\Omega)}^2+\mathcal{O}\left(\beta_\ve\mu_\ve^{\frac{N-2}{2}}\kappa_\ve^2|\log\mu_\ve|\right)\notag\\
&&+o\left(\ve\beta_\ve^2+\beta_\ve^{p+1}+\beta_\ve\kappa_\ve^{N-2}\right)
+\left\{\aligned
&o\left(\beta_\ve\mu_\ve^{2}|\log\mu_\ve|\right),\quad N=4,\\
&o\left(\beta_\ve\mu_\ve^{2}\right),\quad N=5,6,\\
&o\left(\beta_\ve\mu_\ve^{2}\right)+\mathcal{O}\left(\beta_\ve\kappa_\ve^{\frac{(N+2)p}{2}}\right),\quad N\geq7,
\endaligned\right.
\end{eqnarray}
which, together with $(1)$ of Proposition~\ref{exp-orth}, implies that
\begin{eqnarray}\label{basic-relat+}
\alpha_N\mu_\ve^{N-2}D_{N,3}\varphi(\xi_\ve)+\|\mathcal{E}_\ve^*\|_{L^{p+1}(\Omega)}^{p+1}+\ve\|\mathcal{E}_\ve^*\|_{L^2(\Omega)}^2
\leq\left\{\aligned
&(\lambda D_{N,1}+o(1))\mu_\ve^2,\quad N\geq5,\\
&(\lambda D_{4,2}+o(1))\mu_\ve^2\left|\log\mu_\ve\right|,\quad N=4
\endaligned\right.
\end{eqnarray}
and
\begin{eqnarray}\label{basic-relat-}
\alpha_N\mu_\ve^{N-2}D_{N,3}\varphi(\xi_\ve)+\|\mathcal{E}_\ve^*\|_{L^{p+1}(\Omega)}^{p+1}+\ve\|\mathcal{E}_\ve^*\|_{L^2(\Omega)}^2
\geq\left\{\aligned
&(\lambda D_{N,1}+o(1))\mu_\ve^2,\quad N\geq5,\\
&(\lambda D_{4,2}+o(1))\mu_\ve^2\left|\log\kappa_\ve\right|,\quad N=4.
\endaligned\right.
\end{eqnarray}
Let $t_{k,l,\ve}=\frac{\tau_{k,l,\ve}'}{\max_{1\leq l\leq m_k}\left|\tau_{k,l,\ve}'\right|}$, then we may assume that $t_{k,l,\ve}\to t_{k,l,0}$ as $\ve\to0^+$ such that
\begin{eqnarray}\label{limit-e}
\|\mathcal{E}_0^*\|_{L^{p+1}(\Omega)}^{p+1}\not=0\quad\text{and}\quad \|\mathcal{E}_0^*\|_{L^{2}(\Omega)}^{2}\not=0,
\end{eqnarray}
where $\mathcal{E}_0^*=\sum_{l=1}^{m_k}t_{k,l,0}e_{k,l}$.
It follows from \eqref{beta-new}, \eqref{basic-relat+} and \eqref{limit-e} that
\begin{eqnarray}\label{beta-mu}
\mu_\ve^{N-2}\varphi(\xi_\ve)+\beta_\ve^{p+1}+\ve\beta_\ve^2\lesssim\left\{\aligned
&\mu_\ve^2,\quad N\geq5,\\
&\mu_\ve^2\left|\log\mu_\ve\right|,\quad N=4
\endaligned\right.
\end{eqnarray}
and for $N=4$,
\begin{eqnarray}\label{logd}
|\log d(\xi_\ve,\partial \Omega)|\lesssim|\log(\log\mu_\ve)|.
\end{eqnarray}

\vskip0.12in

{\bf The case $N\geq6$}.

\vskip0.06in

By $(1)$ of Proposition~\ref{exp-orth}, we have that $\xi_\ve\to\xi_0\in\partial\Omega$ as $\ve\to0^+$, which together with \eqref{exp-rob}, implies that
\begin{eqnarray}\label{N6}
\kappa_\ve^{N-2}\sim\mu_\ve^2\quad\text{for }N\geq6.
\end{eqnarray}
By \eqref{beta-new}, \eqref{el}, \eqref{beta-mu} and $(3)$ of Proposition~\ref{exp-orth}, we also have 
\begin{eqnarray}\label{beta-mu-2}
\ve\beta_\ve+\beta_\ve^{p}\lesssim\mu_\ve^{\frac{N-2}{2}}.
\end{eqnarray}
Thus, by \eqref{exp-rob}, \eqref{beta-mu-2}, \eqref{N6} and $(2)$ of Proposition~\ref{exp-orth}, 
\begin{eqnarray*}
\kappa_\ve^{N-1}\lesssim\mu_\ve^{\frac{N}{2}}\beta_\ve=o(\kappa_\ve^{N-1})
\end{eqnarray*}
for $N\geq6$ as $\ve\to0^+$, which is impossible.  

\vskip0.12in

{\bf The case $N=5$}.

\vskip0.06in

If $\xi_\ve\to\xi_0\in\partial\Omega$ as $\ve\to0^+$, then arguments used for $N\geq6$ also yields a contradiction in this case, since by \eqref{beta-mu-2}, we have $\mu_\ve^{\frac{3}{2}}\beta_\ve=o\left(\mu_\ve^2\right)$ which still implies that $\kappa_\ve^3\sim\mu_\ve^2$ in this case by $(1)$ of Proposition~\ref{exp-orth}.  Thus, for $N=5$, we must have $\xi_\ve\to\xi_0\in\Omega$ as $\ve\to0^+$, which together with $(1)$ of Proposition~\ref{exp-orth} and \eqref{basic-relat+}, implies that
\begin{eqnarray}\label{N5}
\mu_\ve^{\frac{3}{2}}\mathcal{E}_{\ve}^*(\xi_\ve)\sim-\mu_\ve^2.
\end{eqnarray}
If $\mathcal{E}_0^*(\xi_0)\not=0$, then by \eqref{N5}, we have $\beta_\ve\sim\mu_\ve^{\frac{1}{2}}$.  It follows that $\beta_\ve^{\frac{10}{3}}\sim\mu_\ve^{\frac{5}{3}}$ as $\ve\to0^+$, which contradicts \eqref{basic-relat+}.  Thus, for $N=5$, we must have that $\mathcal{E}_0^*(\xi_0)=0$ and $\beta_\ve\mu_\ve^{-\frac{1}{2}}\to+\infty$ as $\ve\to0^+$, which, together with $(2)$ of Proposition~\ref{exp-orth} and \eqref{basic-relat+} again, implies that $\nabla \mathcal{E}_0^*(\xi_0)=0$.  Thus, for $N=5$, we must have that $\xi_0$ is a singular point of $\mathcal{E}_0^*$.

\vskip0.12in

{\bf The case $N=4$}.

\vskip0.06in

We first consider the case that $\xi_\ve\to\xi_0\in\Omega$ as $\ve\to0^+$.  Then by \eqref{beta-new}, \eqref{basic-relat+}, \eqref{basic-relat-} and \eqref{limit-e}, we have
\begin{eqnarray}\label{beta-mu-4-0}
\beta_\ve^{p+1}+\ve\beta_\ve^2\sim
\mu_\ve^2\left|\log\mu_\ve\right|,
\end{eqnarray}
which, together with $(1)$ of Proposition~\ref{exp-orth}, implies that 
\begin{eqnarray}\label{equa0006}
0=(\lambda D_{4,2}+o(1))\mu_\ve^2\left|\log\mu_\ve\right|+\left(1+\frac{\ve}{\lambda_k}\right)D_{4,3}\mu_\ve\mathcal{E}_{\ve}^*(\xi_\ve).
\end{eqnarray}
Thus, $\beta_\ve\gtrsim\mu_\ve|\log\mu_\ve|$.  It follows from $(2)$ of Proposition~\ref{exp-orth} and \eqref{beta-mu-4-0} that $\nabla\mathcal{E}_0^*(\xi_0)=0$.  Moreover, if $\mathcal{E}_0^*(\xi_0)\not=0$, that is, $\xi_0$ is a regular point of $\mathcal{E}_0^*$, then by \eqref{beta-mu-4-0} and \eqref{equa0006}, 
\begin{eqnarray*}
\beta_\ve\sim\frac{1}{\ve}e^{-\frac{1}{\ve}}\quad\text{and}\quad\mu_\ve\sim e^{-\frac{1}{\ve}}.
\end{eqnarray*}
It follows from \eqref{beta-new} that 
\begin{eqnarray*}
t_{k,l,0}=\lim_{\ve\to0^+}t_{k,l,\ve}=\lim_{\ve\to0^+}\frac{\tau_{k,l,\ve}'}{\max_{1\leq l\leq m_k}\left|\tau_{k,l,\ve}'\right|}=\lim_{\ve\to0^+}\frac{\tau_{k,l,\ve}}{\beta_\ve}.
\end{eqnarray*}
We next consider the case that $\xi_\ve\to\xi_0\in\partial\Omega$ as $\ve\to0^+$.  Without loss of generality, we assume that $\xi_0$ is a regular point of $\mathcal{E}_0^*$, that is, $\nabla\mathcal{E}_0^*(\xi_0)\not=0$.  Thus, by \eqref{beta-new}, \eqref{basic-relat+}, \eqref{logd}, \eqref{beta-mu-2} and $(2)$ of Proposition~\ref{exp-orth}, 
\begin{eqnarray}\label{equa0007}
\beta_\ve^6\lesssim\left(\mu_\ve^2|\log\mu_\ve|\right)^{\frac{5}{4}}\beta_\ve=o\left(\mu_\ve^2\beta_\ve\right)\quad\text{and}\quad\kappa_\ve^3\sim\mu_\ve^2\beta_\ve,
\end{eqnarray}
which implies that $\log\kappa_\ve\sim\log\mu_\ve$.  Suppose the contrary that $\kappa_\ve^2<<\mu_\ve^2|\log\mu_\ve|$.  Then by $(1)$ and $(2)$ of Proposition~\ref{exp-orth}, \eqref{equa0008} and \eqref{equa0009},
\begin{eqnarray}\label{equa0010}
\mu_\ve\beta_\ve d(\xi_\ve, \partial\Omega)\sim\mu_\ve^2|\log\mu_\ve|\quad\text{and}\quad\ve\beta_\ve^2\sim\mu_\ve^2|\log\mu_\ve|,
\end{eqnarray}
which, together with \eqref{equa0007}, implies that $d(\xi_\ve, \partial\Omega)\sim\frac{\ve\beta_\ve}{\mu_\ve}$ and $\ve\sim|\log\mu_\ve|^{-2}$.  However, by $d(\xi_\ve, \partial\Omega)\sim\frac{\ve\beta_\ve}{\mu_\ve}$, $\kappa_\ve^2<<\mu_\ve^2|\log\mu_\ve|$ and \eqref{equa0010}, we also have $\ve>>|\log\mu_\ve|^{-2}$.  It is a contradiction.  Thus, we must have $\kappa_\ve^2\sim\mu_\ve^2|\log\mu_\ve|$, which, together with \eqref{equa0007} and $(3)$ of Proposition~\ref{exp-orth}, implies that $d(\xi_\ve, \partial\Omega)\sim|\log\mu_\ve|^{-\frac12}$, $\beta_\ve\sim\mu_\ve|\log\mu_\ve|^{\frac{3}{2}}$ and
\begin{eqnarray*}
\mu_\ve\beta_\ve d(\xi_\ve, \partial\Omega)\sim\mu_\ve^2|\log\mu_\ve|\sim\ve\beta_\ve^2.
\end{eqnarray*}
It follows that
\begin{eqnarray*}
\mu_\ve\sim e^{-\frac{1}{\sqrt{\ve}}},\quad \beta_\ve\sim \frac{1}{\ve^{\frac{3}{4}}}e^{-\frac{1}{\sqrt{\ve}}}\quad\text{and}\quad d(\xi_\ve, \partial\Omega)\sim\ve^{\frac{1}{4}}.
\end{eqnarray*}
It completes the proof and we remark that all the computations for regular $\xi_0$ can be precise up to the leading order terms.
\hfill$\Box$

\section{The second refinement and the classification in the case $\lambda\to\overline{\lambda}^-$}
We recall that by Lemma~\ref{locate-limit}, we have determined that $\overline{\lambda}=\lambda_k$ for some $k\geq1$.  In the case $\lambda\to\lambda_k^{-}$, by $(3)$ of Proposition~\ref{exp-orth}, cancelations may occur in the leading order terms if and only if
\begin{eqnarray}\label{cancelation}
\|\mathcal{E}_0^*\|_{L^{p+1}(\Omega)}^{p+1}+\|\mathcal{E}_0^*\|_{L^{2}(\Omega)}^{2}\lim_{\ve\to0^-}\frac{\ve}{\left(\beta_\ve^*\right)^{p-1}}=0,
\end{eqnarray}
where $\mathcal{E}_0^*(x)=\sum_{i=1}^{m_{k}}t_{k,i,0}e_{k,i}(x)$
with $t_{k,i,0}=\lim_{\ve\to0}\frac{\tau_{k,i,\ve}}{\beta_{\ve}^*}$ and $\beta_\ve^*=\max_{1\leq i\leq m_{k}}|\tau_{k,i,\ve}|\to0$.
Thus, under the refined expansion in Proposition~\ref{exp-orth}, we can only establish a partial classification result.

\vskip0.12in

\noindent\textbf{Proof of Theorem~\ref{zero weak limit-1}:} \quad
Since 
\begin{eqnarray*}
\|\mathcal{E}_0^*\|_{L^{p+1}(\Omega)}^{p+1}+\|\mathcal{E}_0^*\|_{L^{2}(\Omega)}^{2}\lim_{\ve\to0^-}\frac{\ve}{\left(\beta_\ve^*\right)^{p-1}}\not=0,
\end{eqnarray*}
by $(1)$ and $(3)$ of Proposition~\ref{exp-orth}, we must have 
\begin{eqnarray}\label{beta-mu-}
\max\{\ve\beta_\ve,\beta_\ve^p\}\lesssim \mu_\ve^{\frac{N-2}{2}}.
\end{eqnarray}
Thus, we can go though the same argument for Theorem~\ref{zero weak limit+} to conclude that the blow-up is impossible for $N\geq6$ and for $N=5$, we must have that 
\begin{eqnarray*}
\|\mathcal{E}_0^*\|_{L^{p+1}(\Omega)}^{p+1}+\|\mathcal{E}_0^*\|_{L^{2}(\Omega)}^{2}\lim_{\ve\to0^-}\frac{\ve}{\left(\beta_\ve^*\right)^{p-1}}>0
\end{eqnarray*}
and $\xi_\ve\to\xi_0\in\Omega$ is a singular point of $\mathcal{E}_0^*$ given in \eqref{limit-e}.  For $N=4$, if $\xi_\ve\to\xi_0\in\Omega$ then by $(1)$  of Proposition~\ref{exp-orth} and \eqref{beta-mu-}, we must have $\beta_\ve\gtrsim\mu_\ve|\log\mu_\ve|$, which, together with $(2)$ of Proposition~\ref{exp-orth} and \eqref{beta-mu-} once more, implies that $\nabla\mathcal{E}_0^*(\xi_0)=0$.  Thus, either $\xi_0$ is a singular point of $\mathcal{E}_0^*$ or $\mathcal{E}_0^*(\xi_0)\not=0$.  If $\mathcal{E}_0^*(\xi_0)\not=0$ then by $(1)$ of Proposition~\ref{exp-orth} and \eqref{beta-mu-} once more,
\begin{eqnarray*}
\mu_\ve|\log\mu_\ve|\sim\beta_\ve,\quad \mathcal{E}_\ve^*(\xi_\ve)<0
\end{eqnarray*}
and
\begin{eqnarray*}
\|\mathcal{E}_0^*\|_{L^{p+1}(\Omega)}^{p+1}+\|\mathcal{E}_0^*\|_{L^{2}(\Omega)}^{2}\lim_{\ve\to0^-}\frac{\ve}{\left(\beta_\ve^*\right)^{p-1}}>0,
\end{eqnarray*}
which, together with $(3)$ of Proposition~\ref{exp-orth} and $\ve<0$, implies that 
\begin{eqnarray*}
0=\left(1+o(1)\right)D_{N,6}\mu_\ve^{\frac{N-2}{2}}\mathcal{E}_\ve^*(\xi_\ve)+(1+o(1))\ve\|\mathcal{E}_\ve^*\|_{L^2(\Omega)}^2<0
\end{eqnarray*} 
for $|\ve|$ sufficiently small.  It is a contradiction.  If $\xi_\ve\to\xi_0\in\partial\Omega$ then without loss of generality, we may assume that $\nabla\mathcal{E}_0^*(\xi_0)\not=0$.  If 
\begin{eqnarray*}
\|\mathcal{E}_0^*\|_{L^{p+1}(\Omega)}^{p+1}+\|\mathcal{E}_0^*\|_{L^{2}(\Omega)}^{2}\lim_{\ve\to0^-}\frac{\ve}{\left(\beta_\ve^*\right)^{p-1}}>0,
\end{eqnarray*}
then we still have $\mathcal{E}_\ve^*(\xi_\ve)<0$ by $(1)$ and $(3)$ of Proposition~\ref{exp-orth}.  Thus, by \eqref{beta-mu-}, we can go though the same argument for Theorem~\ref{zero weak limit+} to conclude that $\kappa_\ve^3\sim\mu_\ve^2\beta_\ve$.  If 
\begin{eqnarray*}
\|\mathcal{E}_0^*\|_{L^{p+1}(\Omega)}^{p+1}+\|\mathcal{E}_0^*\|_{L^{2}(\Omega)}^{2}\lim_{\ve\to0^-}\frac{\ve}{\left(\beta_\ve^*\right)^{p-1}}<0
\end{eqnarray*}
and $\left(\beta_\ve^*\right)^4<<\ve\left(\beta_\ve^*\right)^2$, then by $(1)$ and $(3)$ of Proposition~\ref{exp-orth}, $\mathcal{E}_\ve^*(\xi_\ve)>0$ and $\kappa_\ve^2\gtrsim\mu_\ve^2|\log \kappa_\ve|+\ve\beta_\ve^2>>\beta_\ve^4$, which implies that $|\log d(\xi_\ve,\partial \Omega)|\lesssim |\log \mu_\ve|\sim|\log\kappa_\ve|$.  It follows from $(2)$ of Proposition~\ref{exp-orth} and \eqref{beta-mu-} that $\kappa_\ve^3\sim\mu_\ve^2\beta_\ve$.  If 
\begin{eqnarray*}
\|\mathcal{E}_0^*\|_{L^{p+1}(\Omega)}^{p+1}+\|\mathcal{E}_0^*\|_{L^{2}(\Omega)}^{2}\lim_{\ve\to0^-}\frac{\ve}{\left(\beta_\ve^*\right)^{p-1}}<0,
\end{eqnarray*}
$\left(\beta_\ve^*\right)^4\sim\ve\left(\beta_\ve^*\right)^2$ and $\ve\left(\beta_\ve^*\right)^2\lesssim\mu_\ve^2|\log\mu_\ve^2|$, then by \eqref{equa0007}, \eqref{beta-mu-} and $(2)$ of Proposition~\ref{exp-orth} again, we can also conclude that $\kappa_\ve^3\sim\mu_\ve^2\beta_\ve$.  If 
\begin{eqnarray*}
\|\mathcal{E}_0^*\|_{L^{p+1}(\Omega)}^{p+1}+\|\mathcal{E}_0^*\|_{L^{2}(\Omega)}^{2}\lim_{\ve\to0^-}\frac{\ve}{\left(\beta_\ve^*\right)^{p-1}}<0,
\end{eqnarray*}
$\left(\beta_\ve^*\right)^4\sim\ve\left(\beta_\ve^*\right)^2$ and $\ve\left(\beta_\ve^*\right)^2>>\mu_\ve^2|\log\mu_\ve^2|$, then by $(1)$ and $(3)$ of Proposition~\ref{exp-orth}, $\kappa_\ve^2\sim|\ve|\beta_\ve^2$, which, together with $(3)$ of Proposition~\ref{exp-orth} once more, implies that
\begin{eqnarray*}
\kappa_\ve^2\sim|\ve|\beta_\ve^2\sim\beta_\ve^4\sim\mu_\ve\beta_\ve d(\xi_\ve, \partial\Omega).
\end{eqnarray*}
Thus, $\mu_\ve^2\beta_\ve\sim|\ve|^{\frac{11}{4}}$ and $\beta_\ve^6\sim|\ve|^3$.  It 
follows from $(2)$ of Proposition~\ref{exp-orth} and \eqref{beta-mu-} that we still have $\kappa_\ve^3\sim\mu_\ve^2\beta_\ve$.  In a word, we always have 
\begin{eqnarray}\label{equa0011}
\left(1+\frac{\ve}{\lambda_k}\right)\mu_\ve^{2}\frac{\partial \mathcal{E}_{\ve}^*(\xi_\ve)}{\partial \xi_{i,\ve}}=(\alpha_4+o(1))\mu_\ve^{3}\frac{\partial \varphi(\xi_\ve)}{\partial \xi_{i,\ve}}
\end{eqnarray}
if $\xi_\ve\to\xi_0\in\partial\Omega$ such that $\nabla\mathcal{E}_0^*(\xi_0)\not=0$.  Thus, as in the proof of Theorem~\ref{zero weak limit+}, we have $\log\kappa_\ve\sim\log\mu_\ve$ as $\ve\to0^-$.  If
\begin{eqnarray}\label{equa0041}
\|\mathcal{E}_0^*\|_{L^{p+1}(\Omega)}^{p+1}+\|\mathcal{E}_0^*\|_{L^{2}(\Omega)}^{2}\lim_{\ve\to0^-}\frac{\ve}{\left(\beta_\ve^*\right)^{p-1}}>0
\end{eqnarray}
and $\mathcal{E}_\ve^*(\xi_\ve)<0$, then by \eqref{basic-relat+}, \eqref{basic-relat-} and $\kappa_\ve^3\sim\mu_\ve^2\beta_\ve$, we have 
\begin{eqnarray}\label{equa0008}
\kappa_\ve^2+\beta_\ve^4+\ve\beta_\ve^2\sim\mu_\ve^2|\log\mu_\ve|,
\end{eqnarray}
$\kappa_\ve\lesssim\mu_\ve|\log\mu_\ve|^{\frac{1}{2}}$ and $\beta_\ve\lesssim\mu_\ve|\log\mu_\ve|^{\frac{3}{2}}$, which implies that \begin{eqnarray}\label{equa0009}
\kappa_\ve^2+\ve\beta_\ve^2\sim\mu_\ve^2|\log\mu_\ve|.
\end{eqnarray}
We claim that $\kappa_\ve^2\sim\mu_\ve^2|\log\mu_\ve|$.  Suppose the contrary that $\kappa_\ve^2<<\mu_\ve^2|\log\mu_\ve|$.  Then by $(1)$ and $(2)$ of Proposition~\ref{exp-orth}, \eqref{equa0008} and \eqref{equa0009},
\begin{eqnarray}\label{equa0010}
\mu_\ve\beta_\ve d(\xi_\ve, \partial\Omega)\sim\mu_\ve^2|\log\mu_\ve|\quad\text{and}\quad\ve\beta_\ve^2\sim\mu_\ve^2|\log\mu_\ve|,
\end{eqnarray}
which, together with $\kappa_\ve^3\sim\mu_\ve^2\beta_\ve$, implies that $d(\xi_\ve, \partial\Omega)\sim\frac{\ve\beta_\ve}{\mu_\ve}$ and $\ve\sim|\log\mu_\ve|^{-2}$.  It follows from \eqref{equa0010} that $\kappa_\ve^2\sim\mu_\ve^2|\log\mu_\ve|$.  
It is a contradiction.  Thus, we must have $\kappa_\ve^2\sim\mu_\ve^2|\log\mu_\ve|$, which, together with $\kappa_\ve^3\sim\mu_\ve^2\beta_\ve$ and $(3)$ of Proposition~\ref{exp-orth}, implies that $d(\xi_\ve, \partial\Omega)\sim|\log\mu_\ve|^{-\frac12}$, $\beta_\ve\sim\mu_\ve|\log\mu_\ve|^{\frac{3}{2}}$ and
\begin{eqnarray*}
\mu_\ve\beta_\ve d(\xi_\ve, \partial\Omega)\sim\mu_\ve^2|\log\mu_\ve|\sim\ve\beta_\ve^2.
\end{eqnarray*}
It follows that
\begin{eqnarray}\label{equa0040}
\mu_\ve\sim e^{-\frac{1}{\sqrt{\ve}}},\quad \beta_\ve\sim \frac{1}{\ve^{\frac{3}{4}}}e^{-\frac{1}{\sqrt{\ve}}}\quad\text{and}\quad d(\xi_\ve, \partial\Omega)\sim\ve^{\frac{1}{4}}.
\end{eqnarray}
However, by \eqref{beta-new} and \eqref{equa0040}, we have $\left(\beta_\ve^*\right)^4=o\left(|\ve|\left(\beta_\ve^*\right)^2\right)$, which contradicts \eqref{equa0041}.  Thus, we must have 
\begin{eqnarray*}\label{equa0030}
\|\mathcal{E}_0^*\|_{L^{p+1}(\Omega)}^{p+1}+\|\mathcal{E}_0^*\|_{L^{2}(\Omega)}^{2}\lim_{\ve\to0^-}\frac{\ve}{\left(\beta_\ve^*\right)^{p-1}}<0\quad\text{and}\quad \mathcal{E}_\ve^*(\xi_\ve)>0.
\end{eqnarray*}
Again, by \eqref{exp-rob} and \eqref{equa0011},
\begin{eqnarray*}
\left(1+\frac{\ve}{\lambda_k}\right)\mu_\ve^{2}\left(\nabla\mathcal{E}_{\ve}^*(\xi_\ve)\cdot(\xi_\ve-\xi_\ve' )\right)=(\alpha_4+o(1))\mu_\ve^{3}\left(\nabla\varphi(\xi_\ve)\cdot(\xi_\ve-\xi_\ve' )\right)<0,
\end{eqnarray*}
where $\xi_\ve' \in \partial \Omega$ is the unique point satisfying $d(\xi_\ve,\partial \Omega)=\left|\xi_\ve-\xi_\ve'\right|$.  Thus, by \eqref{exp-rob}, the Taylor expansion and $(1)$ and $(2)$ of Proposition~\ref{exp-orth}, we have
\begin{eqnarray*}
-\mu_\ve^2\left|\log\kappa_\ve\right|\gtrsim\left(1+\frac{\ve}{\lambda_k}\right)\mu_\ve\beta_\ve \mathcal{E}_\ve^*(\xi_\ve)-\alpha_4\mu_\ve^{2}\varphi(\xi_\ve)\sim\kappa_\ve^2,
\end{eqnarray*}
which is impossible.  Thus, if $\xi_\ve\to\xi_0\in\partial\Omega$ then $\xi_0$ must be a singular point of $\mathcal{E}_0^*$.
\hfill$\Box$

\vskip0.12in

It remains to consider the case that \eqref{cancelation} holds true for which cancelations occur in the leading order terms in $(3)$ of Proposition~\ref{exp-orth}.  In this case, the refined expansion obtained in Proposition~\ref{exp-orth} is not good enough for the classification and we need to refine it once more.  Recall that by \eqref{cancelation}, we set
\begin{eqnarray}\label{equa0021}
\beta_\ve^*=(1+o(1))\overline{\beta}_\ve\quad\text{and}\quad\tau_{k,l,\ve}'=(t_{k,l,0}+o(1))\overline{\beta}_\ve,
\end{eqnarray}
where $\overline{\beta}_\ve=\left(\frac{\|\mathcal{E}_0^*\|_{L^{2}(\Omega)}^{2}}{\|\mathcal{E}_0^*\|_{L^{p+1}(\Omega)}^{p+1}}|\ve|\right)^{\frac{1}{p-1}}$.
\begin{proposition}\label{new-decomp}
Suppose the matrix
\begin{eqnarray*}
\left(\int_{\Omega}\left(f'\left(\widetilde{\mathcal{E}}_{\ve,0}^*\right)-1\right)e_{k,l}e_{k,t}dx\right)_{m_k\times m_k}
\end{eqnarray*}
is regular, where $\widetilde{\mathcal{E}}_{\ve,0}^*=\sum_{l=1}^{m_k}\left(\frac{\|\mathcal{E}_0^*\|_{L^{2}(\Omega)}^{2}}{\|\mathcal{E}_0^*\|_{L^{p+1}(\Omega)}^{p+1}}\right)^{\frac{1}{p-1}}t_{k,l,0} e_{k,l}$, then for $\ve<0$ such that $|\ve|>0$ sufficiently small, there exists a unique $\overline{\rho}_\ve\in C^{2,\alpha}(\overline{\Omega})$ satisfying the following equation
\begin{eqnarray*}
\left\{\aligned
&-\Delta \overline{\rho}_\ve-\lambda \overline{\rho}_\ve=f\left(\overline{\mathcal{E}}_{\ve}^*+\overline{\rho}_\ve\right)+\ve \overline{\mathcal{E}}_{\ve}^*,\quad\text{in }\Omega,\\
&\langle\overline{\rho}_\ve, e_{k,l}\rangle=0,\quad 1\leq l\leq m_k,\\
&\overline{\rho}_\ve=0,\quad\text{on }\partial\Omega,
\endaligned\right.
\end{eqnarray*}
such that $\|\overline{\rho}_\ve\|_{C^2(\overline{\Omega})}\lesssim\beta_\ve^{p}$,
where $\overline{\mathcal{E}}_{\ve}^*=\sum_{l=1}^{m_k}\overline{\beta}_\ve t_{k,l} e_{k,l}$ with $\left|t_{k,l}-t_{k,l,0}\right|=\mathcal{O}\left(\beta_\ve^{2p}\right)$ for all $1\leq l\leq m_k$.
\end{proposition}
\begin{proof}
Let us first consider the following equation
\begin{eqnarray}\label{equa0016}
\left\{\aligned
&-\Delta \overline{\rho}_\ve-\lambda \overline{\rho}_\ve=f\left(\overline{\mathcal{E}}_{\ve}^*+\overline{\rho}_\ve\right)+\ve \overline{\mathcal{E}}_{\ve}^*+\sum_{l=1}^{m_k}\overline{c}_{k,l,\ve}e_{k,l},\quad\text{in }\Omega,\\
&\langle\overline{\rho}_\ve, e_{k,l}\rangle=0,\quad 1\leq l\leq m_k,\\
&\overline{\rho}_\ve=0,\quad\text{on }\partial\Omega,
\endaligned\right.
\end{eqnarray}
where $\overline{c}_{k,l,\ve}\in\mathbb{R}$ are a part of unknowns serving as the Lagrange multipliers.  Since $\lambda=\lambda_k+\ve$ with $\ve\to0^-$ and $\langle\overline{\rho}_\ve, e_{k,l}\rangle=0$ for all $1\leq l\leq m_k$, the existence and uniqueness of $\overline{\rho}_\ve\in C^2(\Omega)$ and $\overline{c}_{k,l,\ve}\in\mathbb{R}$ satisfying \eqref{equa0016} is a direct conclusion of the application of the blow-up and fix-point arguments and the Fredholm alternative, where
\begin{eqnarray*}
\overline{c}_{k,l,\ve}=-\int_{\Omega}\left(f\left(\overline{\mathcal{E}}_{\ve}^*+\overline{\rho}_\ve\right)+\ve \overline{\mathcal{E}}_{\ve}^*\right)e_{k,l}dx
\end{eqnarray*}
for all $1\leq l\leq m_k$.  Moreover, by the standard elliptic regularity theory and the choice of $\overline{\beta}_\ve$, $\|\overline{\rho}_\ve\|_{C^2(\overline{\Omega})}\lesssim\beta_\ve^{p}$.  Since the matrix
\begin{eqnarray*}
\left(\int_{\Omega}\left(f'\left(\widetilde{\mathcal{E}}_{\ve,0}^*\right)-1\right)e_{k,l}e_{k,t}dx\right)_{m_k\times m_k}
\end{eqnarray*}
is regular, it is standard to show, by the fix-point argument, that the system
\begin{eqnarray*}
0=\int_{\Omega}\left(f\left(\overline{\mathcal{E}}_{\ve}^*+\overline{\rho}_\ve\right)+\ve \overline{\mathcal{E}}_{\ve}^*\right)e_{k,l}dx,\quad 1\leq l\leq m_k,
\end{eqnarray*}
has a solution $\pmb{t}_{k}=(t_{k,1}, t_{k,2}, \cdots, t_{k,m_k})$ satisfying $\left|\pmb{t}_{k}-\pmb{t}_{k,0}\right|=\mathcal{O}\left(\beta_\ve^{2p}\right)$, where 
\begin{eqnarray*}
\pmb{t}_{k,0}=(t_{k,1,0}, t_{k,2,0}, \cdots, t_{k,m_k,0}).
\end{eqnarray*}
It completes the proof.
\end{proof}

Recall that by \eqref{v} and \eqref{orthogonal-decomp}, we have
\begin{eqnarray*}
u_\ve=\mathcal{W}_{\mu_\ve,\xi_\ve}+\sum_{l=1}^{m_k}\tau_{k,l,\ve}'e_{k,l}+\sum_{j=0}^{N}\alpha_{j,\ve}Z_{\mu_\ve,\xi_\ve}^j+v_\ve^*,
\end{eqnarray*}
where $v_\ve^*$ satisfies \eqref{v*-equ}.  We write
\begin{eqnarray}\label{new-v**}
v_\ve^{**}=v_\ve^*-\overline{\rho}_\ve-\sum_{l=1}^{m_k}\overline{\tau}_{k,l,\ve}'e_{k,l}-\sum_{j=0}^{N}\overline{\alpha}_{j,\ve}Z_{\mu_\ve,\xi_\ve}^j,
\end{eqnarray}
such that by Proposition~\ref{new-decomp}, $v_\ve^{**}$ satisfies
\begin{eqnarray}\label{v**-equ}
\left\{\aligned
&-\Delta v_{\ve}^{**}-\lambda v_{\ve}^{**}-f'(\mathcal{W}_{\mu_\ve,\xi_\ve}^{**})v_{\ve}^{**}=\overline{\mathcal{R}}_0(v_{\ve}^{**})+\mathcal{R}_1+\overline{\mathcal{R}}_2,\quad\text{in }\Omega,\\
&\langle v_\ve^{**}, \mathcal{Z}_{\mu_\ve, \xi_\ve}^i\rangle=\langle v_\ve^{**}, e_{k,l}\rangle=0,\quad 0\leq i\leq N,\quad1\leq l\leq m_k,\\
&v_\ve^{**}=0,\quad\text{on }\partial\Omega,
\endaligned\right.
\end{eqnarray}
where $\mathcal{R}_1$ is given in \eqref{R},
\begin{eqnarray}\label{Rv**}
\overline{\mathcal{R}}_0(v_{\ve}^{**})&=&f\left(\mathcal{W}_{\mu_\ve,\xi_\ve}^{**}+\mathcal{E}_{\ve}^{**}+v_{\ve}^{**}\right)-f\left(\mathcal{W}_{\mu_\ve,\xi_\ve}^{**}+\mathcal{E}_{\ve}^{**}\right)-f'\left(\mathcal{W}_{\mu_\ve,\xi_\ve}^{**}+\mathcal{E}_{\ve}^{**}\right)v_{\ve}^{**}\notag\\
&&+\left(f'\left(\mathcal{W}_{\mu_\ve,\xi_\ve}^{**}+\mathcal{E}_{\ve}^{**}\right)-f'\left(\mathcal{W}_{\mu_\ve,\xi_\ve}^{**}\right)\right)v_{\ve}^{**}
\end{eqnarray}
and
\begin{eqnarray}\label{R2*}
\overline{\mathcal{R}}_2&=&f\left(\mathcal{W}_{\mu_\ve,\xi_\ve}^{**}+\mathcal{E}_{\ve}^{**}\right)-f\left(\mathcal{W}_{\mu_\ve,\xi_\ve}\right)-f\left(\mathcal{E}_{\ve}^{**}+\widehat{\mathcal{E}}_{\ve}^{**}\right)\notag\\
&&-\ve\widehat{\mathcal{E}}_{\ve}^{**}+\lambda\mathcal{Z}_\ve^*-\sum_{j=0}^{N}\overline{\alpha}^{*}_{j,\ve}f'\left(U_{\mu_\ve,\xi_\ve}\right)\Phi_{\mu_\ve,\xi_\ve}^j
\end{eqnarray}
with $\mathcal{E}_{\ve}^{**}=\sum_{l=1}^{m_k}\overline{\tau}^{*}_{k,l,\ve}e_{k,l}+\overline{\rho}_\ve$, $\widehat{\mathcal{E}}_{\ve}^{**}=\sum_{l=1}^{m_k}\widehat{\tau}^{**}_{k,l,\ve}e_{k,l}$,
\begin{eqnarray}\label{new**}
\mathcal{Z}^*_\ve=\sum_{j=0}^{N}\overline{\alpha}^{*}_{j,\ve}Z_{\mu_\ve,\xi_\ve}^j,\quad\mathcal{W}_{\mu_\ve,\xi_\ve}^{**}=\mathcal{W}_{\mu_\ve,\xi_\ve}+\mathcal{Z}^*_\ve,\quad\overline{\mathcal{E}}_{\ve}^{**}=\sum_{l=1}^{m_k}\overline{c}_{k,l,\ve}e_{k,l},
\end{eqnarray}
and 
\begin{eqnarray}\label{pare}
\left\{\aligned
&\overline{\tau}^{*}_{k,l,\ve}=\tau_{k,l,\ve}'+\overline{\tau}_{k,l,\ve}',\\
&\widehat{\tau}^{**}_{k,l,\ve}=\overline{\beta}_\ve t_{k,l}-\overline{\tau}^*_{k,l,\ve},\\
&\overline{\alpha}^{*}_{j,\ve}=\alpha_{j,\ve}+\overline{\alpha}_{j,\ve}.
\endaligned\right.
\end{eqnarray}
As \eqref{system}, by the orthogonal conditions of $v_\ve^{*}$ and $v_\ve^{**}$ in \eqref{v*-equ} and \eqref{v**-equ}, respectively, \eqref{equa0016} and \eqref{new-v**}, we have
\begin{eqnarray}\label{system*}
\left\{\aligned
&0=\overline{\tau}_{k,l,\ve}'+\sum_{i=0}^{N}\overline{\alpha}_{i,\ve}\langle Z_{\mu_\ve,\xi_\ve}^i, e_{k,l}\rangle,\quad 1\leq l\leq m_k,\\
&0=\langle \overline{\rho}_\ve, Z_{\mu_\ve,\xi_\ve}^j\rangle+\sum_{l=1}^{m_k}\overline{\tau}_{k,l,\ve}'\langle e_{k,l}, Z_{\mu_\ve,\xi_\ve}^j\rangle+\sum_{i=0}^{N}\overline{\alpha}_{i,\ve}\langle Z_{\mu_\ve,\xi_\ve}^i, Z_{\mu_\ve,\xi_\ve}^j\rangle,\quad 0\leq j\leq N,
\endaligned\right.
\end{eqnarray}
which, together with Lemma~\ref{compute-ez}, \eqref{zi-zj} and Proposition~\ref{new-decomp}, implies that 
\begin{eqnarray}\label{new-esti*}
\sum_{i=0}^{N}\left|\overline{\alpha}_{i,\ve}\right|=\left\{\aligned
&\mathcal{O}\left(\mu_\ve^{2}\beta_\ve^{p}\right),\quad N=4,\\
&\mathcal{O}\left(\mu_\ve^{3}|\log\mu_\ve|\beta_\ve^{p}\right),\quad N=5,\\
&\mathcal{O}\left(\mu_\ve^{3}\beta_\ve^{p}\right),\quad N\geq6.
\endaligned\right.
\quad\text{and}\quad
\sum_{l=1}^{m_k}\left|\overline{\tau}_{k,l,\ve}'\right|=\left\{\aligned
&\mathcal{O}\left(\mu_\ve^{3}\beta_\ve^{p}\right),\quad N=4,\\
&\mathcal{O}\left(\mu_\ve^{\frac{9}{2}}|\log\mu_\ve|\beta_\ve^{p}\right),\quad N=5,\\
&\mathcal{O}\left(\mu_\ve^{\frac{N+4}{2}}\beta_\ve^{p}\right),\quad N\geq6.
\endaligned\right.
\end{eqnarray}
It follows from \eqref{solution-sys}, \eqref{pare} and \eqref{new-esti*} that
\begin{eqnarray}\label{new-esti**}
\sum_{l=1}^{m_k}\left|\overline{\tau}^{*}_{k,l,\ve}\right|=\mathcal{O}\left(\beta_\ve\right)\quad\text{and}\quad\sum_{i=0}^{N}\left|\overline{\alpha}^{*}_{i,\ve}\right|=\left\{\aligned
&\mathcal{O}\left(\mu_\ve^{\frac{N-2}{2}}\beta_\ve\right),\quad 4\leq N\leq8,\\
&\mathcal{O}\left(\mu_\ve^{3}\beta_\ve^{p}+\mu_\ve^{\frac{N-2}{2}}\beta_\ve\right),\quad N\geq9.
\endaligned\right.
\end{eqnarray}
Moreover, by \eqref{solution-sys}, \eqref{beta-new}, \eqref{equa0021} and \eqref{new-esti*},
\begin{eqnarray*}\label{tau**}
\sum_{l=1}^{m_k}\left|\widehat{\tau}^{**}_{k,l,\ve}\right|=o(\beta_\ve)+\left\{\aligned
&\mathcal{O}\left(\kappa_\ve^{N-2}\right),\quad N=4,5,\\
&\mathcal{O}\left(\kappa_\ve^4|\log\kappa_\ve|^{\frac{2}{3}}+\mu_\ve^2\left|\log\mu_\ve\right|^{\frac{2}{3}}\right),\quad N=6,\\
&\mathcal{O}\left(\kappa_\ve^{\frac{N+2}{2}}+\mu_\ve^2\right),\quad N\geq7.
\endaligned\right.
\end{eqnarray*}
\begin{lemma}\label{esti-v**}
Let $N\geq4$ and $v_\ve^{**}$ is the solution of \eqref{v**-equ}.  Then
\begin{eqnarray*}
\|v_\ve^{**}\|_{H^1_0(\Omega)}\lesssim\sum_{l=1}^{m_k}\ve\left|\widehat{\tau}^{**}_{k,l,\ve}\right|+\left\{\aligned
&\kappa_\ve^{N-2}+\mu_\ve^{\frac{N-2}{2}},\quad N=4,5,\\
&\kappa_\ve^4|\log\kappa_\ve|^{\frac{2}{3}}+\mu_\ve^2\left|\log\mu_\ve\right|^{\frac{2}{3}},\quad N=6,\\
&\kappa_\ve^{\frac{N+2}{2}}+\mu_\ve^2,\quad N\geq7
\endaligned\right.
\end{eqnarray*}
and
\begin{eqnarray*}
\|v_\ve^{**}\|_{L^\infty (\Omega)} \lesssim  \beta_\ve+\mu_\ve^{-\frac{N-2}{2}}\kappa_\ve^{N-2}+\mu_\ve^{\frac{6-N}{2}}.
\end{eqnarray*}
\end{lemma}
\begin{proof}
The proof is similar to that of Lemma~\ref{esti-v*}, so we only sketch it and point out the differences.  By Proposition~\ref{new-decomp}, \eqref{new-esti**} and \eqref{new**}, we have
\begin{eqnarray}\label{equa0017}
\left\|\mathcal{E}_{\ve}^{**}\right\|_{L^{\infty}(\Omega)}\lesssim\beta_\ve\quad\text{and}\quad 0<\mathcal{W}_{\mu_\ve,\xi_\ve}^{**}\sim U_{\mu_\ve,\xi_\ve}.
\end{eqnarray}
Thus, similar to \eqref{R0-1}, by \eqref{Rv**} and \eqref{equa0017}, we still have
\begin{eqnarray}\label{R0-1**}
\left|\overline{\mathcal{R}}_0(v_{\ve}^{**})\right|
\lesssim\beta_\ve U_{\mu_\ve,\xi_\ve}^{p-2}\left|v_{\ve}^{**}\right|{\bf 1}_{(\Omega\backslash\mathcal{A}_4)\cap (\Omega\backslash\mathcal{A}_4')}+U_{\mu_\ve,\xi_\ve}^{p-2}\left|v_{\ve}^{**}\right|^2{\bf 1}_{\Omega\backslash\mathcal{A}_4'}+\beta_\ve^{p-1}\left|v_{\ve}^{**}\right|{\bf 1}_{\mathcal{A}_4'}+\left|v_{\ve}^{**}\right|^{p}{\bf 1}_{\mathcal{A}_4}
\end{eqnarray}
where
\begin{eqnarray*}\label{A4}
\left\{\aligned
&\mathcal{A}_4=\left\{x \in \Omega\mid \left|\mathcal{W}_{\mu_\ve,\xi_\ve}^{**}(x)+\mathcal{E}_\ve^{**}(x)\right|\lesssim\left|v_{\ve}^{**}(x)\right|\right\},\\
&\mathcal{A}_4'=\left\{x \in \Omega\mid \left|\mathcal{W}_{\mu_\ve,\xi_\ve}^{**}(x)\right|\lesssim\left|\mathcal{E}_\ve^{**}(x)\right|\right\}.
\endaligned\right.
\end{eqnarray*}
It follows from \eqref{R0-1**} that
\begin{eqnarray}\label{R0**}
\left\|\mathcal{R}_0(v_{\ve}^{**})\right\|_{L^{\frac{p+1}{p}}(\Omega)}=o\left(\left\|v_\ve^{**}\right\|_{H^1_0(\Omega)}\right).
\end{eqnarray}
By \eqref{solution-sys}, Proposition~\ref{new-decomp}, \eqref{new-esti*}, 
\begin{eqnarray}\label{equa0018}
\left\|\ve\widehat{\mathcal{E}}_{\ve}^{**}\right\|_{L^{\infty}(\Omega)}=\mathcal{O}\left(\sum_{l=1}^{m_k}\ve\left|\widehat{\tau}^{**}_{k,l,\ve}\right|\right)
\end{eqnarray}
and
\begin{eqnarray}\label{equa0019}
\left|f\left(\mathcal{E}_{\ve}^{**}+\widehat{\mathcal{E}}_{\ve}^{**}\right)-f\left(\mathcal{E}_{\ve}^{**}\right)\right|=\mathcal{O}\left(\sum_{l=1}^{m_k}\beta_\ve^{p-1}\left|\widehat{\tau}^{**}_{k,l,\ve}\right|\right).
\end{eqnarray}
Thus, similar to \eqref{R2-exp}, by \eqref{R2*}, \eqref{new**}, \eqref{new-esti**}, \eqref{equa0018} and \eqref{equa0019}, 
\begin{eqnarray}\label{R2**-exp}
\left|\overline{\mathcal{R}}_2\right|=\mathcal{O}\left(\beta_\ve U_{\mu_\ve,\xi_\ve}^{p-1}+\beta_\ve^{p-1} U_{\mu_\ve,\xi_\ve}+\left(\left(\mu_\ve^{\frac{N-2}{2}}\beta_\ve\right)\vee\left(\mu_\ve^{3}\beta_\ve^p\right)\right) U_{\mu_\ve,\xi_\ve}+\ve\overline{\tau}^{**}_{k,l,\ve}\right),
\end{eqnarray}
which implies that
\begin{eqnarray}\label{R2**}
\left\|\overline{\mathcal{R}}_2\right\|_{L^{\frac{p+1}{p}}(\Omega)}\lesssim\sum_{l=1}^{m_k}\ve\overline{\tau}^{**}_{k,l,\ve}+\left\{\aligned
&\beta_\ve\mu_\ve^{\frac{N-2}{2}}+\beta_\ve^{p-1}\mu_\ve^{\frac{N-2}{2}}+\mu_\ve^{N-2}\beta_\ve,\quad N=4,5,\\
&\beta_\ve\mu_\ve^2|\log\mu_\ve|^{\frac{2}{3}}+\beta_\ve\mu_\ve^2\left|\log\mu_\ve\right|^{\frac{2}{3}},\quad N=6,\\
&\beta_\ve\mu_\ve^2+\beta_\ve^{p-1}\mu_\ve^2,\quad N\geq7.
\endaligned\right.
\end{eqnarray}
Here, $a\vee b=\min\{a,b\}$.  Combining \eqref{R1-esti-0}, \eqref{R0**}, \eqref{R2**} and Lemma~\ref{Lem2.1}, we have
 \begin{eqnarray*}
\|v_\ve^{**}\|_{H^1_0(\Omega)}\lesssim\sum_{l=1}^{m_k}\ve\overline{\tau}^{**}_{k,l,\ve}+\left\{\aligned
&\kappa_\ve^{N-2}+\mu_\ve^{\frac{N-2}{2}},\quad N=4,5,\\
&\kappa_\ve^4|\log\kappa_\ve|^{\frac{2}{3}}+\mu_\ve^2\left|\log\mu_\ve\right|^{\frac{2}{3}},\quad N=6,\\
&\kappa_\ve^{\frac{N+2}{2}}+\mu_\ve^2,\quad N\geq7.
\endaligned\right.
\end{eqnarray*}
The $L^{\infty}$ estimate of $v_\ve^{**}$ follows from Lemma~\ref{esti-v*}, \eqref{new-v**} and \eqref{new-esti*}.
\end{proof}

With Lemma~\ref{esti-v**} in hands, we could refine Proposition~\ref{exp-orth} once more to exclude the cancelation in the leading order terms.
\begin{proposition}\label{refine-exp-orth}
Let $N\geq4$, then 
\begin{enumerate}
\item[$(1)$] we have
\begin{eqnarray*}\label{z0**-}
0&\geq&\left\{\aligned
&\lambda D_{N,1}\mu_\ve^2-\alpha_N\mu_\ve^{N-2}D_{N,3}\varphi(\xi_\ve)+D_{N,3}\mu_\ve^{\frac{N-2}{2}}\mathcal{E}_{\ve}^{**}(\xi_\ve)+o\left(\mu_\ve^2+\kappa_\ve^{N-2}\right),\quad N\geq5,\\
&\lambda D_{4,2}\mu_\ve^2\left|\log\kappa_\ve\right|-\alpha_4\mu_\ve^{2}D_{4,3}\varphi(\xi_\ve)+D_{N,3}\mu_\ve^{\frac{N-2}{2}}\mathcal{E}_{\ve}^{**}(\xi_\ve)+o\left(\kappa_\ve^{2}\right)+\mathcal{O}\left(\mu_\ve^2\right), \quad N=4,
\endaligned\right.\\
&&+\mathcal{O}\left(\left(\ve\beta_{1,\ve}\right)^2\right)
\end{eqnarray*}
and
\begin{eqnarray*}\label{z0**+}
0&\leq&\left\{\aligned
&\lambda D_{N,1}\mu_\ve^2-\alpha_N\mu_\ve^{N-2}D_{N,3}\varphi(\xi_\ve)+D_{N,3}\mu_\ve^{\frac{N-2}{2}}\mathcal{E}^{**}(\xi_\ve)+o\left(\mu_\ve^2+\kappa_\ve^{N-2}\right),\quad N\geq5,\\
&\lambda D_{4,2}\mu_\ve^2\left|\log\mu_\ve\right|-\alpha_4\mu_\ve^{2}D_{4,3}\varphi(\xi_\ve)+D_{N,3}\mu_\ve^{\frac{N-2}{2}}\mathcal{E}_{\ve}^{**}(\xi_\ve)+o\left(\kappa_\ve^{2}\right)+\mathcal{O}\left(\mu_\ve^2\right), \quad N=4,
\endaligned\right.\\
&&+\mathcal{O}\left(\left(\ve\beta_{1,\ve}\right)^2\right),
\end{eqnarray*}
where $\mathcal{E}_{\ve}^{**}=\sum_{l=1}^{m_k}\overline{\tau}^{*}_{k,l,\ve}e_{k,l}$.
\item[$(2)$] For all $1\leq j\leq N$, we have
\begin{eqnarray*}\label{zj**}
0&=&D_{N,4}\mu_\ve^{\frac{N}{2}}\frac{\partial \mathcal{E}_{\ve}^{**}(\xi_\ve)}{\partial \xi_{i,\ve}}-\alpha_ND_{N,4}\mu_\ve^{N-1}\frac{\partial \varphi(\xi_\ve)}{\partial \xi_{i,\ve}}+\mathcal{O}\left(\beta_\ve\mu_\ve^{\frac{N-2}{2}}\kappa_\ve^3\left|\log\kappa_\ve\right|+\mu_\ve^{2}\kappa_\ve^{N-3}\right)\\
&&+o\left(\kappa_\ve^{N-1}+\beta_\ve\mu_\ve^{\frac{N}{2}}\right)
+\left\{\aligned
&\mathcal{O}\left(\left(\ve\beta_{1,\ve}\right)^3+\kappa_\ve^5|\log d(\xi_\ve,\partial \Omega)|\right),\quad N=4,\\
&o\left(\left(\ve\beta_{1,\ve}\right)^2+\mu_\ve^3\right)+\mathcal{O}\left(\kappa_\ve^7\mu_\ve^{-1}\right),\quad N=5,\\
&\mathcal{O}\left(\left(\ve\beta_{1,\ve}\right)^2\right)+o\left(\mu_\ve^3\right),\quad N\geq6.
\endaligned\right.
\end{eqnarray*}
\item[$(3)$] For all $1\leq l\leq m_k$, we have
\begin{eqnarray*}\label{el**}
0&=&\left(1+\frac{\ve}{\lambda_k}\right)D_{N,6}\mu_\ve^{\frac{N-2}{2}}e_{k,l}(\xi_\ve)-\left(\int_{\Omega}f'\left(\mathcal{E}_{\ve}^{**}\right)\widehat{\mathcal{E}}_{\ve}^{**}e_{k,l}dx+\int_{\Omega}\ve\widehat{\mathcal{E}}_{\ve}^{**}e_{k,l}dx\right)\notag\\
&&+\mathcal{O}\left(\mu_\ve^{\frac{N-2}{2}}\kappa_\ve^2|\log\mu_\ve|\right)+o\left(\ve\beta_{1,\ve}+\kappa_\ve^{N-2}\right)+\mathcal{O}\left(\beta_\ve^{p-1}\mu_\ve^{\frac{N-2}{2}}\right)\\
&&+\left\{\aligned
&o\left(\mu_\ve^{2}|\log\mu_\ve|\right),\quad N=4,\\
&o\left(\mu_\ve^{2}\right),\quad N=5,6,\\
&o\left(\mu_\ve^{2}\right)+\mathcal{O}\left(\kappa_\ve^{\frac{(N+2)p}{2}}\right),\quad N\geq7.
\endaligned\right.
\end{eqnarray*}
\end{enumerate}
Here, for the sake of simplicity, we denote
\begin{eqnarray*}
\beta_{1,\ve}=\max_{1\leq l\leq m_k}\left|\widehat{\tau}^{**}_{k,l,\ve}\right|.
\end{eqnarray*}
\end{proposition}
\begin{proof}
The estimates are similar to that in the proof of Proposition~\ref{exp-orth}, so we only sketch it and point out the differences.  

\vskip0.12in

For the conclusion~$(1)$, we still test \eqref{v**-equ} with $Z_{\mu_\ve,\xi_\ve}^0$, then by the orthogonality of $v_\ve^{**}$, we have
\begin{eqnarray*}\label{equa0002**}
0&=&\int_{\Omega}\mathcal{R}_1Z_{\mu_\ve,\xi_\ve}^0dx+\int_{\Omega}\overline{\mathcal{R}}_2Z_{\mu_\ve,\xi_\ve}^0dx+\int_{\Omega}\overline{\mathcal{R}}_0(v_\ve^{**})Z_{\mu_\ve,\xi_\ve}^0dx\notag\\
&&+\lambda\int_{\Omega}v_\ve^{**}Z_{\mu_\ve,\xi_\ve}^0dx+\int_{\Omega}f'\left(\mathcal{W}_{\mu_\ve,\xi_\ve}^{**}\right)v_{\ve}^{**}Z_{\mu_\ve,\xi_\ve}^0dx.
\end{eqnarray*}
Then by \eqref{system}, \eqref{system*}, \eqref{new-esti**}, \eqref{R2**-exp} and similar estimates for $\int_{\Omega}\mathcal{R}_2Z_{\mu_\ve,\xi_\ve}^0dx$ in the proof of $(1)$ of Proposition~\ref{exp-orth}, we have
\begin{eqnarray*}\label{R2**-esti}
\int_{\Omega}\overline{\mathcal{R}}_2Z_{\mu_\ve,\xi_\ve}^0dx&=&\int_{\Omega}f'\left(\mathcal{W}_{\mu_\ve,\xi_\ve}\right)\left(\mathcal{E}_{\ve}^{**}+\mathcal{Z}_{\ve}^{*}\right)Z_{\mu_\ve,\xi_\ve}^0dx-\sum_{j=0}^{N}\overline{\alpha}^*_{j,\ve}\int_{\Omega}f'\left(U_{\mu_\ve,\xi_\ve}\right)\Phi_{\mu_\ve,\xi_\ve}^jZ_{\mu_\ve,\xi_\ve}^0dx\notag\\
&&+\int_{\Omega}\left(\lambda\mathcal{Z}_\ve^*-\ve\widehat{\mathcal{E}}_{\ve}^{**}\right)Z_{\mu_\ve,\xi_\ve}^0dx+\mathcal{O}\left(\ve\beta_{1,\ve}\mu_\ve^{\frac{N-2}{2}}\right)\notag\\
&&+\left\{\aligned
&o\left(\mu_\ve^2|\log\mu_\ve|\right),\quad N=4,\\
&o\left(\mu_\ve^2\right),\quad N\geq5,
\endaligned\right.\notag\\
&=&D_{N,3}\mu_\ve^{\frac{N-2}{2}}\widehat{\mathcal{E}}_{\ve}^{**}(\xi_\ve)+\mathcal{O}\left(\left(\ve\beta_{1,\ve}\right)^2\right)+o\left(\kappa_\ve^{N-2}\right)+\left\{\aligned
&o\left(\mu_\ve^2|\log\mu_\ve|\right),\quad N=4,\\
&o\left(\mu_\ve^2\right),\quad N\geq5,
\endaligned\right.
\end{eqnarray*}
where $\widehat{\mathcal{E}}_{\ve}^{**}=\sum_{l=1}^{m_k}\widehat{\tau}^{**}_{k,l,\ve}e_{k,l}$.
Since \eqref{equa0017} holds true, the rest estimates for $(1)$ is the same as that for $(1)$ of Proposition~\ref{exp-orth}, the only difference is to use the estimate on $\|v_\ve^{**}\|_{H^1_0(\Omega)}$ in Lemma~\ref{esti-v**} to replace the usage of $\|v_\ve^{*}\|_{H^1_0(\Omega)}$ in Lemma~\ref{esti-v*}.  Thus, we omit the details of the remaining estimates.

\vskip0.12in

For the conclusion~$(2)$, we still test \eqref{v**-equ} with $Z_{\mu_\ve,\xi_\ve}^i$, then by the orthogonality of $v_\ve^{**}$, we have
\begin{eqnarray*}\label{equa0002**}
0&=&\int_{\Omega}\mathcal{R}_1Z_{\mu_\ve,\xi_\ve}^idx+\int_{\Omega}\overline{\mathcal{R}}_2Z_{\mu_\ve,\xi_\ve}^idx+\int_{\Omega}\overline{\mathcal{R}}_0(v_\ve^{**})Z_{\mu_\ve,\xi_\ve}^idx\notag\\
&&+\lambda\int_{\Omega}v_\ve^{**}Z_{\mu_\ve,\xi_\ve}^idx+\int_{\Omega}f'\left(\mathcal{W}_{\mu_\ve,\xi_\ve}^{**}\right)v_{\ve}^{**}Z_{\mu_\ve,\xi_\ve}^idx.
\end{eqnarray*}
Again, by \eqref{system}, \eqref{system*}, \eqref{new-esti**}, \eqref{R2**-exp} and similar estimates for $\int_{\Omega}\mathcal{R}_2Z_{\mu_\ve,\xi_\ve}^0dx$ in the proof of $(1)$ of Proposition~\ref{exp-orth}, we have
\begin{eqnarray*}
\int_{\Omega}\overline{\mathcal{R}}_2Z_{\mu_\ve,\xi_\ve}^idx&=&\int_{\Omega}f'\left(\mathcal{W}_{\mu_\ve,\xi_\ve}\right)\left(\mathcal{E}_{\ve}^{**}+\mathcal{Z}^*_{\ve}\right)Z_{\mu_\ve,\xi_\ve}^idx-\sum_{j=0}^{N}\overline{\alpha}^*_{j,\ve}\int_{\Omega}f'\left(U_{\mu_\ve,\xi_\ve}\right)\Phi_{\mu_\ve,\xi_\ve}^jZ_{\mu_\ve,\xi_\ve}^idx\\
&&+\int_{\Omega}\left(\lambda\mathcal{Z}^*_\ve-\ve\widehat{\mathcal{E}}_{\ve}^{**}\right)Z_{\mu_\ve,\xi_\ve}^idx+\mathcal{O}\left(\ve\beta_{1,\ve}\mu_\ve^{\frac{N+2}{2}}\kappa_\ve^{-1}\right)\\
&&+\left\{\aligned
&o\left(\kappa_\ve^3\right),\quad N=4,\\
&\mathcal{O}\left(\beta_{\ve}^2\mu_\ve^2(\mu_\ve|\log\kappa_\ve|+\kappa_\ve)\right),\quad N=5,\\
&\mathcal{O}\left(\beta_{\ve}^2\mu_\ve^2\kappa_\ve\right),\quad N=6.
\endaligned\right.\\
&=&D_{N,4}\mu_\ve^{\frac{N}{2}}\frac{\partial \widehat{\mathcal{E}}_{\ve}^{**}(\xi_\ve)}{\partial \xi_{i,\ve}}+\mathcal{O}\left(\beta_\ve\mu_\ve^{\frac{N-2}{2}}\kappa_\ve^3\left|\log\kappa_\ve\right|+\beta_\ve\mu_\ve^{\frac{N+2}{2}}\right)\\
&&+\mathcal{O}\left(\ve\beta_{1,\ve}\mu_\ve^{\frac{N+2}{2}}\kappa_\ve^{-1}\right)+o\left(\kappa_\ve^{N-1}+\beta_\ve\mu_\ve^{\frac N2}\right)+\left\{\aligned
&\mathcal{O}\left(\beta_{\ve}^2\mu_\ve^3|\log\kappa_\ve|\right),\quad N=5,\\
&\mathcal{O}\left(\beta_{\ve}^2\mu_\ve^2\kappa_\ve\right),\quad N=6.
\endaligned\right.
\end{eqnarray*}
Again, since \eqref{equa0017} holds true, the rest estimates for $(1)$ is the same as that for $(2)$ of Proposition~\ref{exp-orth}, the only difference is to use the estimate on $\|v_\ve^{**}\|_{H^1_0(\Omega)}$ in Lemma~\ref{esti-v**} to replace the usage of $\|v_\ve^{*}\|_{H^1_0(\Omega)}$ in Lemma~\ref{esti-v*}.  Thus, we omit the details of the remaining estimates.

\vskip0.12in

For the conclusion~$(3)$, we test \eqref{v**-equ} with $e_{k,l}$, then by the orthogonality of $v_\ve^{**}$, we have
\begin{eqnarray*}\label{equa0025}
0&=&\int_{\Omega}\mathcal{R}_1e_{k,l}dx+\int_{\Omega}\overline{\mathcal{R}}_2e_{k,l}dx+\int_{\Omega}\overline{\mathcal{R}}_0(v_\ve^{**})e_{k,l}dx+\int_{\Omega}f'(\mathcal{W}_{\mu_\ve,\xi_\ve}^{**})v_{\ve}^{**}e_{k,l}dx.
\end{eqnarray*}
Similar to \eqref{R2-e}, by \eqref{R2*} and \eqref{R2**-exp}, we have
\begin{eqnarray*}\label{R2**-e}
\int_{\Omega}\overline{\mathcal{R}}_2e_{k,l}dx
&=&\int_{\Omega}\left(f\left(\mathcal{E}_{\ve}^{**}\right)-f\left(\mathcal{E}_{\ve}^{**}+\widehat{\mathcal{E}}_{\ve}^{**}\right)\right)e_{k,l}dx-\int_{\Omega}\ve\widehat{\mathcal{E}}_{\ve}^{**}e_{k,l}dx\notag\\
&&+\mathcal{O}\left(\beta_\ve^{p-1}\mu_\ve^{\frac{N-2}{2}}\right)+\left\{\aligned
&o\left(\mu_\ve^2|\log\mu_\ve|\right),\quad N=4,\\
&o\left(\mu_\ve^{2}\right),\quad N\geq 5.
\endaligned\right.\notag\\
&=&-\left(\int_{\Omega}f'\left(\mathcal{E}_{\ve}^{**}\right)\widehat{\mathcal{E}}_{\ve}^{**}e_{k,l}dx+\int_{\Omega}\ve\widehat{\mathcal{E}}_{\ve}^{**}e_{k,l}dx\right)\notag\\
&&+o\left(\ve\beta_{1,\ve}\right)+\mathcal{O}\left(\beta_\ve^{p-1}\mu_\ve^{\frac{N-2}{2}}\right)+\left\{\aligned
&o\left(\mu_\ve^2|\log\mu_\ve|\right),\quad N=4,\\
&o\left(\mu_\ve^{2}\right),\quad N\geq 5.
\endaligned\right.
\end{eqnarray*}
Again, since \eqref{equa0017} holds true, the rest estimates for $(1)$ is the same as that for $(3)$ of Proposition~\ref{exp-orth}, the only difference is to use the estimate on $\|v_\ve^{**}\|_{H^1_0(\Omega)}$ in Lemma~\ref{esti-v**} to replace the usage of $\|v_\ve^{*}\|_{H^1_0(\Omega)}$ in Lemma~\ref{esti-v*}.  Thus, we omit the details of the remaining estimates.
\end{proof}

With Proposition~\ref{refine-exp-orth} in hands, we could complete the classification in the case $\lambda\to\lambda_k^-$.

\vskip0.12in

\noindent\textbf{Proof of Theorem~\ref{zero weak limit-2}:} \quad
Since $p>1$, we multiply $(3)$ of Proposition~\ref{refine-exp-orth} with $\widehat{\tau}^{**}_{k,l,\ve}$ and sum them from $l=1$ to $l=m_k$, then we have
\begin{eqnarray}\label{beta1}
\ve\beta_{1,\ve}\lesssim\mu_\ve^{\frac{N-2}{2}}+o\left(\kappa_\ve^{N-2}\right)+\left\{\aligned
&o\left(\mu_\ve^{2}|\log\mu_\ve|\right),\quad N=4,\\
&o\left(\mu_\ve^{2}\right),\quad N\geq5.
\endaligned\right.
\end{eqnarray}
Inserting \eqref{beta1} into $(1)$ of Proposition~\ref{refine-exp-orth} implies that
\begin{eqnarray}\label{z0**-}
0\geq\left\{\aligned
&\left(\lambda D_{N,1}+o(1)\right)\mu_\ve^2-\left(\alpha_ND_{N,3}+o(1)\right)\mu_\ve^{N-2}\varphi(\xi_\ve)+D_{N,3}\mu_\ve^{\frac{N-2}{2}}\mathcal{E}_{\ve}^{**}(\xi_\ve),\quad N\geq5,\\
&\left(\lambda D_{4,2}+o(1)\right)\mu_\ve^2\left|\log\kappa_\ve\right|-\left(\alpha_4D_{4,3}+o(1)\right)\mu_\ve^{2}\varphi(\xi_\ve)+D_{N,3}\mu_\ve^{\frac{N-2}{2}}\mathcal{E}_{\ve}^{**}(\xi_\ve), \quad N=4
\endaligned\right.
\end{eqnarray}
and
\begin{eqnarray}\label{z0**+}
0\leq\left\{\aligned
&\left(\lambda D_{N,1}+o(1)\right)\mu_\ve^2-\left(\alpha_ND_{N,3}+o(1)\right)\mu_\ve^{N-2}\varphi(\xi_\ve)+D_{N,3}\mu_\ve^{\frac{N-2}{2}}\mathcal{E}_{\ve}^{**}(\xi_\ve),\quad N\geq5,\\
&\left(\lambda D_{4,2}+o(1)\right)\mu_\ve^2\left|\log\mu_\ve\right|-\left(\alpha_4D_{4,3}+o(1)\right)\mu_\ve^{2}\varphi(\xi_\ve)+D_{N,3}\mu_\ve^{\frac{N-2}{2}}\mathcal{E}_{\ve}^{**}(\xi_\ve), \quad N=4.
\endaligned\right.
\end{eqnarray}
Inserting \eqref{beta1} into $(2)$ of Proposition~\ref{refine-exp-orth} also implies that
\begin{eqnarray}\label{zj**}
0&=&D_{N,4}\mu_\ve^{\frac{N}{2}}\frac{\partial \mathcal{E}_{\ve}^{**}(\xi_\ve)}{\partial \xi_{i,\ve}}-\alpha_ND_{N,4}\mu_\ve^{N-1}\frac{\partial \varphi(\xi_\ve)}{\partial \xi_{i,\ve}}+\mathcal{O}\left(\beta_\ve\mu_\ve^{\frac{N-2}{2}}\kappa_\ve^3\left|\log\kappa_\ve\right|+\mu_\ve^{2}\kappa_\ve^{N-3}\right)\notag\\
&&+o\left(\kappa_\ve^{N-1}+\beta_\ve\mu_\ve^{\frac{N}{2}}\right)
+\left\{\aligned
&\mathcal{O}\left(\kappa_\ve^5|\log d(\xi_\ve,\partial \Omega)|\right),\quad N=4,\\
&o\left(\mu_\ve^3\right)+\mathcal{O}\left(\kappa_\ve^7\mu_\ve^{-1}\right),\quad N=5,\\
&o\left(\mu_\ve^3\right),\quad N\geq6.
\endaligned\right.
\end{eqnarray}

\vskip0.12in

{\bf The case~$N\geq6$.}

\vskip0.06in

By \eqref{z0**-} and \eqref{z0**+}, for $N\geq6$, we must have $\xi_\ve\to\xi_0\in\partial\Omega$ and $\mu_\ve^2\sim\kappa_\ve^{N-2}$ as $\ve\to0^-$.  Thus, by \eqref{exp-rob} and \eqref{zj**}, we still have $\kappa_\ve^{N-1}\lesssim\mu_\ve^{\frac{N}{2}}\beta_\ve$, which is impossible for $|\ve|>0$ sufficiently small.

\vskip0.12in

{\bf The case~$N=5$.}

\vskip0.06in

If $\xi_\ve\to\xi_0\in\partial\Omega$ as $\ve\to0^-$ such that $\xi_0\in\partial\Omega$ is a regular point of $\mathcal{E}_0^*$, that is $\nabla \mathcal{E}_0^*(\xi_0)\not=0$, then by \eqref{exp-rob}, $(1)$ and $(3)$ of Proposition~\ref{refine-exp-orth}, 
\begin{eqnarray}\label{equa0031}
D_{5,3}\mu_\ve^{\frac{3}{2}}\nabla\mathcal{E}_{\ve}^{**}(\xi_\ve)\cdot(\xi_\ve-\xi_\ve')=\left(\lambda D_{5,1}+o(1)\right)\mu_\ve^2-\left(\alpha_5D_{5,3}+o(1)\right)\mu_\ve^{N-2}\varphi(\xi_\ve)
\end{eqnarray}
and
\begin{eqnarray*}
-D_{5,3}\mu_\ve^{\frac{3}{2}}\nabla\widehat{\mathcal{E}}_{\ve}^{**}(\xi_\ve)\cdot(\xi_\ve-\xi_\ve')\sim-\ve\beta_{1,\ve}^2+o\left(\kappa_\ve^3\beta_{1,\ve}+\mu_\ve^3\beta_{1,\ve}\right),
\end{eqnarray*}
where $\xi_\ve' \in \partial \Omega$ is the unique point satisfying $d(\xi_\ve,\partial \Omega)=\left|\xi_\ve-\xi_\ve'\right|$.  It follows that $\mu_\ve^2\sim\kappa_\ve^3+\ve\beta_{1,\ve}\beta_\ve$.  Thus, by \eqref{exp-rob} and \eqref{zj**}, we still have 
\begin{eqnarray}\label{equa0025}
\mu_\ve^{\frac{5}{2}}\nabla\mathcal{E}_{\ve}^{**}(\xi_\ve)\cdot(\xi_\ve-\xi_\ve')=-\left(\frac{3\alpha_5}{2^3}+o(1)\right)\kappa_\ve^{4}\left|\xi_\ve-\xi_\ve'\right|
\end{eqnarray}
as $\ve\to0^-$,   It follows from \eqref{exp-rob}, \eqref{equa0031} and \eqref{equa0025} that $-\mu_\ve^2\sim\kappa_\ve^3$ as $\ve\to0^-$, which is impossible.  Thus, if $\xi_\ve\to\xi_0\in\partial\Omega$ as $\ve\to0^-$ then $\xi_0\in\partial\Omega$ must be a singular point of $\mathcal{E}_0^*$.

If $\xi_\ve\to\xi_0\in\Omega$ as $\ve\to0^-$ then by \eqref{z0**-} and \eqref{z0**+},
\begin{eqnarray*}
\left(\lambda D_{5,1}+o(1)\right)\mu_\ve^2=D_{5,3}\mu_\ve^{\frac{3}{2}}\mathcal{E}_{\ve}^{**}(\xi_\ve),
\end{eqnarray*}
which implies that $\beta_\ve\gtrsim\mu_\ve^{\frac12}$.  It follows from \eqref{zj**} that $\nabla \mathcal{E}_0^*(\xi_0)=0$.  Thus, without loss of generality, we may assume that  $\xi_0\in\Omega$ is a regular point of $\mathcal{E}_0^*$, that is $\mathcal{E}_0^*(\xi_0)\not=0$ and $\nabla \mathcal{E}_0^*(\xi_0)=0$.  It follows from  \eqref{cancelation} that
\begin{eqnarray*}
\beta_\ve\sim|\ve|^{\frac{3}{4}}\quad\text{and}\quad\mu_\ve\sim|\ve|^{\frac{3}{2}}.
\end{eqnarray*}
Moreover, all the estimates can be precisely computed up to the leading order terms.

\vskip0.12in

{\bf The case~$N=4$.}

\vskip0.06in

If $\xi_\ve\to\xi_0\in\partial\Omega$ as $\ve\to0^-$ such that $\xi_0\in\partial\Omega$ is a regular point of $\mathcal{E}_0^*$, that is $\nabla \mathcal{E}_0^*(\xi_0)\not=0$, then by \eqref{exp-rob}, we can also get the same contradiction in the case $N=5$, that is, $-\mu_\ve^2|\log\kappa_\ve|\gtrsim\kappa_\ve^2$ as $\ve\to0^-$.  Thus, if $\xi_\ve\to\xi_0\in\partial\Omega$ as $\ve\to0^-$ then $\xi_0\in\partial\Omega$ must be a singular point of $\mathcal{E}_0^*$.

If $\xi_\ve\to\xi_0\in\Omega$ as $\ve\to0^-$ then by \eqref{z0**-} and \eqref{z0**+},
\begin{eqnarray}\label{equa0026}
\left(\lambda D_{4,1}+o(1)\right)\mu_\ve^2|\log\mu_\ve|=D_{4,3}\mu_\ve\mathcal{E}^{**}(\xi_\ve),
\end{eqnarray}
which implies that $\beta_\ve\gtrsim\mu_\ve|\log\mu_\ve|$.  It follows from \eqref{zj**} that $\nabla \mathcal{E}_0^*(\xi_0)=0$.  If we further have that $\xi_0\in\Omega$ is a regular point of $\mathcal{E}_0^*$, then by $(3)$ of Proposition~\ref{refine-exp-orth}, $\mu_\ve\beta_\ve=o(\ve\beta_\ve)$ as $\ve\to0^-$.  However, by \eqref{equa0026}, $\mu_\ve\sim\frac{\beta_\ve}{|\log\beta_\ve|}$ as $\ve\to0^-$.  Thus, by \eqref{cancelation}, $\frac{|\ve|}{|\log|\ve||}=o(|\ve|^{\frac{3}{2}})$ as $\ve\to0^-$, which is impossible for $|\ve|>0$ sufficiently small.  Thus, if $\xi_\ve\to\xi_0\in\Omega$ as $\ve\to0^-$ then $\xi_0\in\Omega$ also must be a singular point of $\mathcal{E}_0^*$.  It completes the proof.
\hfill$\Box$

\section{The existence theory of $4d$ Brezis-Nirenberg equation for $\lambda=\lambda_k$}
We recall that it has been proved by Cerami, Solimini and Struwe \cite{CSS1986}, Chen, Lin and Zou \cite{CLZ2014}, Chen, Shioji and Zou \cite{CSZ2012}, Clapp and Weth \cite{CW2005}, Tavares, You and Zou \cite{TYZ2022}, and Szulkin, Weth and Willem \cite{SWW2009} that the least energy sign-changing solutions of \eqref{BN}, denoted by $u_{\lambda, sg}$, satisfies
\begin{eqnarray}\label{energy-1}
E_\lambda(u_{\lambda,sg})=\left\{\aligned
&\inf_{\mathcal{N}_{sg}}E_\lambda(u)<\frac{2}{N}S^{\frac{N}{2}},\quad0<\lambda<\lambda_1,\\
&\inf_{\mathcal{P}_i}E_\lambda(u)<\frac{1}{N}S^{\frac{N}{2}},\quad\lambda_i<\lambda<\lambda_{i+1}\text{ and }i\geq1,
\endaligned\right.
\end{eqnarray}
where $\lambda_i$ are the eigenvalues of $-\Delta$ in $H_0^1(\Omega)$, 
\begin{eqnarray*}
E_\lambda(u)=\frac{1}{2}\|u\|^2_{H_0^1(\Omega)}-\frac{\lambda}{2}\|u\|^2_{L^2(\Omega)}-\frac{1}{p+1}\|u\|^{p+1}_{L^{p+1}(\Omega)}
\end{eqnarray*}
is the corresponding energy functional of the Brezis-Nirenberg equation~\eqref{BN}, 
\begin{eqnarray*}
\mathcal{N}_{sg}=\left\{u\in H_0^1(\Omega)\backslash\{0\}\mid E'_\lambda(u^{\pm})u^{\pm}=0\text{ and }u^\pm\not=0\right\}
\end{eqnarray*}
and
\begin{eqnarray*}
\mathcal{P}_i=\left\{u\in H_0^1(\Omega)\backslash\{0\}\mid E'_\lambda(u)u=0\quad\text{and}\quad Q_jE'_\lambda(u)=0, 1\leq j\leq i\right\}
\end{eqnarray*}
are its associated sign-changing Nehari manifold and Pankov-Nehari manifold, respectively, with $Q_j$ the projection from $H_0^1(\Omega)$ onto 
\begin{eqnarray}\label{xij}
\Xi_j=\bbr e_{j,1}\oplus\bbr e_{j,2}\oplus\cdots\oplus\bbr e_{j,m_j}
\end{eqnarray}
with
$m_j\in\bbn$ being the multiplicity of $\lambda_{j}$ and $\{e_{j,i}\}$ being its orthogonal system.

\vskip0.12in

\noindent\textbf{Proof of Theorem~\ref{existence}:} \quad
We divide the proof into three steps.

\vskip0.06in

{\bf Step~1}\quad We prove that the $4d$ Brezis-Nirenberg equation~\eqref{BN} has a least energy sign-changing solution for $\lambda=\lambda_1$.

\vskip0.06in

For this purpose, let $u_{\lambda,sg}$ be a sequence of least energy sign-changing solutions of the $4d$ Brezis-Nirenberg equation~\eqref{BN} for $\lambda<\lambda_1$ and $\lambda\to\lambda_1^-$.  Since it is known \cite{AP2025, SWW2009, TYZ2022} that 
$m_{\lambda,sg}:=E_{\lambda}(u_{\lambda,sg})$ is strictly decreasing for $\lambda\in(0, \lambda_1)$ and $\lim_{\lambda\to\lambda_1^-}m_{\lambda,sg}\leq\frac{1}{4}S^{2}$.  Thus, it is standard to show that $\{u_{\lambda,sg}\}$ is bounded in $H_0^1(\Omega)$, which, together with the Struwe decomposition \cite{Str1984} and $\lim_{\lambda\to\lambda_1^-}m_{\lambda,sg}\leq\frac{1}{4}S^{2}$, implies that \eqref{condition} holds true for $\{u_{\lambda,sg}\}$ up to a subsequence, provided $\{u_{\lambda,sg}\}$ is noncompact in $H_0^1(\Omega)$ up to a subsequence.  By Theorems~\ref{zero weak limit-1} and \ref{zero weak limit-2}, the concentration point must be a singular point of the principal eigenfunction $e_{1,1}$.  However, as the principal eigenfunction, $e_{1,1}$ is regular on $\overline{\Omega}$ by the strong maximum principle since $\partial\Omega$ is smooth.  It follows that $\{u_{\lambda,sg}\}$ is compact in $H_0^1(\Omega)$, which implies that the limit, say $u_{\lambda_1,sg}$ is a sign-changing solution of the $4d$ Brezis-Nirenberg equation~\eqref{BN} for $\lambda=\lambda_1$ satisfying $E_{\lambda_1}(u_{\lambda_1,sg})\leq\frac{1}{4}S^{2}$.  Since it is known \cite{AP2025, SWW2009, TYZ2022} that 
\begin{eqnarray*}
\inf_{\mathcal{P}_1}E_{\lambda_1}(u)\leq\frac{1}{4}S^{2},
\end{eqnarray*}
either $u_{\lambda_1,sg}$ is a least energy sign-changing solution of the $4d$ Brezis-Nirenberg equation~\eqref{BN} for $\lambda=\lambda_1$ or $\inf_{\mathcal{P}_1}E_{\lambda_1}(u)<\frac{1}{4}S^{2}$, which also implies the existence of least energy sign-changing solution by classical variational methods used in \cite{SWW2009}.

\vskip0.06in

{\bf Step~2}\quad We prove that the $4d$ Brezis-Nirenberg equation~\eqref{BN} has a least energy sign-changing solution for $\lambda=\lambda_k$ with $k\geq2$.

\vskip0.06in

For this purpose, let $u_{\lambda,sg}$ be a sequence of least energy sign-changing solutions of the $4d$ Brezis-Nirenberg equation~\eqref{BN} for $\lambda_k<\lambda<\lambda_{k+1}$ and $\lambda\to\lambda_k^+$.  By \eqref{energy-1}, we can use the same argument in Step~1 to show that \eqref{condition} holds true for $\{u_{\lambda,sg}\}$ up to a subsequence, provided $\{u_{\lambda,sg}\}$ is noncompact in $H_0^1(\Omega)$ up to a subsequence.  Thus, by Theorem~\ref{zero weak limit+}, Proposition~\ref{exp-orth}, \eqref{orthogonal-decomp}, \eqref{solution-sys} and \eqref{beta-mu-2},
\begin{eqnarray*}
u_\ve=\mathcal{W}_{\xi_\ve, \mu_\ve}^*+\mathcal{E}_{\ve}^*+v_{\ve}^*
\end{eqnarray*}
and
\begin{eqnarray}\label{equa0035}
\left\{\aligned
&\lambda D_{4,2}\mu_\ve^2\left|\log\mu_\ve\right|-\alpha_4\mu_\ve^{2}D_{4,3}\varphi(\xi_\ve)+D_{4,3}\mu_\ve\mathcal{E}_{\ve}^*(\xi_\ve)= o\left(\mu_\ve^2|\log\mu_\ve|\right),\\
&D_{4,6}\mu_\ve \mathcal{E}_{\ve}^*(\xi_\ve)+\left\|\mathcal{E}_\ve^*\right\|_{L^{4}(\Omega)}^4+\ve\left\|\mathcal{E}_\ve^*\right\|_{L^{2}(\Omega)}^2=o\left(\mu_\ve^2|\log\mu_\ve|\right),
\endaligned\right.
\end{eqnarray}
where $\mathcal{W}_{\xi_\ve, \mu_\ve}^*$ and $\mathcal{E}_{\ve}^*$ are given by \eqref{new*}, $v_\ve^*$ is given in Lemma~\ref{esti-v*},
\begin{eqnarray*}
D_{4,3}=3\int_{\mathbb{R}^4}U_{1,0}^2\Phi_{1,0}^0dx\quad\text{and}\quad D_{4,6}=\int_{\mathbb{R}^4}U_{1,0}^3dx,
\end{eqnarray*}
and
\begin{eqnarray*}\label{equa0015}
\xi_\ve\to\xi_0\in\overline{\Omega},\quad \mu_{\ve}\to0,\quad \beta_\ve^*=\max_{1\leq i\leq m_{k}}|\tau_{k,i,\ve}|\to0\quad\text{and}\quad\frac{\|\varphi_{\ve}\|_{H_0^1(\Omega)}}{\max_{1\leq i\leq m_{k}}|\tau_{k,i,\ve}|}\to0
\end{eqnarray*}
as $\ve=\lambda-\lambda_k\to0^+$.  It follows from \eqref{beta-mu-2}, \eqref{equa0035} and Lemma~\ref{App-esti} that
\begin{eqnarray}\label{equa0036}
E_{\lambda}\left(u_{\lambda,sg}\right)&=&\frac{1}{4}S^2-\frac{\lambda}{2}\left\|\mathcal{W}_{\mu_\ve,\xi_\ve}\right\|_{L^2(\Omega)}^2-\lambda\int_{\Omega}\mathcal{W}_{\mu_\ve,\xi_\ve}\mathcal{E}_{\ve}^*dx+\frac{\alpha_4\mu_\ve}{2}\int_{\Omega}U_{\mu_\ve,\xi_\ve}^3H(x,\xi_\ve)dx\notag\\
&&-\frac{\ve}{2}\left\|\mathcal{E}_{\ve}^*\right\|_{L^2(\Omega)}^2-\frac{1}{4}\left\|\mathcal{E}_{\ve}^*\right\|_{L^4(\Omega)}^4+o(\mu_\ve^2|\log\mu_\ve|)\notag\\
&=&\frac{1}{4}S^2-\frac{1}{2}\left(\lambda D_{4,2}\mu_\ve^2\left|\log\mu_\ve\right|-\alpha_4\mu_\ve^{2}D_{4,3}\varphi(\xi_\ve)+D_{4,6}\mu_\ve \mathcal{E}_{\ve}^*(\xi_\ve)\right)\notag\\
&&-\frac{1}{2}\left(\ve\left\|\mathcal{E}_{\ve}^*\right\|_{L^2(\Omega)}^2+\left\|\mathcal{E}_{\ve}^*\right\|_{L^4(\Omega)}^4+D_{4,6}\mu_\ve \mathcal{E}_{\ve}^*(\xi_\ve)\right)+\frac{1}{4}\left\|\mathcal{E}_{\ve}^*\right\|_{L^4(\Omega)}^4+o(\mu_\ve^2|\log\mu_\ve|)\notag\\
&\geq&\frac{1}{4}S^2+(D_{4,3}-D_{4,6})\mu_\ve \mathcal{E}_{\ve}^*(\xi_\ve)+o(\mu_\ve^2|\log\mu_\ve|).
\end{eqnarray}
Since by the standard inverse stereographic projection $\mathbb{R}^4\to\mathbb{S}^4$, i.e.,
\begin{eqnarray*}
\omega_i=\left\{\aligned
&\frac{2x_i}{1+|x|^2},\quad i=1,2,3,4,\\
&\frac{1-|x|^2}{1+|x|^2},\quad i=5,
\endaligned\right.
\end{eqnarray*}
we have
\begin{eqnarray*}
D_{4,3}=\frac{3\alpha_4^3}{16}\int_{\mathbb{R}^4}(|x|^2-1)\mathcal{J}(x)dx=-\frac{6\alpha_4^3}{16}\int_{\mathbb{S}^4}\frac{\omega_5}{1+\omega_5}d\omega=\frac{\alpha_4^3}{4}\left|\mathbb{S}^3\right|
\end{eqnarray*}
and
\begin{eqnarray*}
D_{4,6}=\frac{\alpha_4^3}{16}\int_{\mathbb{R}^4}(|x|^2+1)\mathcal{J}(x)dx=\frac{2\alpha_4^3}{16}\int_{\mathbb{S}^4}\frac{1}{1+\omega_5}d\omega=\frac{\alpha_4^3}{4}\left|\mathbb{S}^3\right|,
\end{eqnarray*}
where $\mathcal{J}(x)=\left(\frac{2}{1+|x|^2}\right)^4$ is the Jacobian of the standard inverse stereographic projection.  Thus, by \eqref{equa0036},
we must have 
\begin{eqnarray}\label{low}
E_{\lambda}\left(u_{\lambda,sg}\right)\geq\frac{1}{4}S^2+o(\mu_\ve^2|\log\mu_\ve|)
\end{eqnarray}
as $\ve\to0^+$.  Thus, by the classical estimates in \cite{SWW2009}, we know that
\begin{eqnarray}\label{minimax-sg}
E_\lambda(u_{\lambda,sg})=\inf_{v\in\mathbb{Y}_k\atop v\not=0}\max_{w\in\mathbb{X}_k\atop t>0}E_\lambda(tv+w)\to\frac{1}{4}S^2
\end{eqnarray}
as $\lambda\to\lambda_k^+$,
where $\mathbb{X}_k=\oplus_{j=1}^{k}\Xi_j$ and $\mathbb{Y}_k=\oplus_{j=k+1}^{\infty}\Xi_j$ with $\Xi_j$ given by \eqref{xij}.  Let
\begin{eqnarray*}
v_{\ve}&=&\mathcal{W}_{\mu_\ve,\eta_\ve}-\sum_{j=1}^{k-1}\sum_{l=1}^{m_j}\varrho_{j,l,\ve}e_{j,l}-\sum_{l=1}^{m_k}\varrho_{k,l,\ve}e_{k,l}\\
&:=&\mathcal{W}_{\mu_\ve,\eta_\ve}-\sum_{j=1}^{k}v_{j,\ve},
\end{eqnarray*}
where $\varrho_{k,l,\ve}$ are chosen such that $\langle v_{\ve}, e_{j,l}\rangle=0$ for all $j\geq k+1$ and $1\leq l\leq m_{j}$ and $\eta_\ve\in\Omega$ will be chosen later.
Moreover, similar to Lemma~\ref{compute-ez}, by  \eqref{beta-mu}, 
\begin{eqnarray*}
\varrho_{j,l,\ve}=D_{4,6}\mu_\ve e_{k,l}(\eta_\ve)+o\left(\mu_\ve^2\left|\log\mu_\ve\right|\right)\quad\text{for all }1\leq j\leq k\text{ and }1\leq l\leq m_j.
\end{eqnarray*}
It follows from \eqref{minimax-sg} that 
\begin{eqnarray*}
E_\lambda(u_{\lambda,sg})\leq E_\lambda\left(t_\ve v_\ve+\sum_{j=1}^{k}w_{j,\ve}\right),
\end{eqnarray*}
where $t_\ve\in\mathbb{R}$ and $w_{j,\ve}=\sum_{l=1}^{m_j}s_{j,l,\ve}e_{j,l}\in \Xi_j$ satisfy $t_\ve\to1$ and $\|w_{j,\ve}\|_{H^1_0(\Omega)}\to0$ as $\ve\to0^+$.  
Since 
\begin{eqnarray}\label{test-10}
\frac{\sum_{l=1}^{m_k}(s_{k,l,\ve}-\varrho_{k,l,\ve})e_{k,l}(x)}{\max_{1\leq l\leq m_k}|s_{k,l,\ve}-\varrho_{k,l,\ve}|}\to \sum_{l=1}^{m_k}s_{k,l}^0e_{k,l}(x)\in \Xi_j,
\end{eqnarray}
where $s_{k,l}^0=\lim_{\ve\to0^+}\frac{s_{k,l,\ve}-\varrho_{k,l,\ve}}{\max_{1\leq l\leq m_k}|s_{k,l,\ve}-\varrho_{k,l,\ve}|}$, we can choose $\eta_\ve\in\Omega$ such that $\sum_{l=1}^{m_k}(s_{k,l,\ve}-\varrho_{k,l,\ve})e_{k,l}(\eta_\ve)=0$ and $d(\eta_\ve,\partial\Omega)\gtrsim1$ as $\ve\to0^+$.  Thus, by the same computations in \eqref{compute-0},
\begin{eqnarray*}
E_\lambda\left(t_\ve v_\ve+w_{k,-}+w_k\right)&=&\left(\frac{1}{2}t_\ve^2-\frac{1}{4}t_\ve^4\right)S^2+\sum_{j=1}^{k}\sum_{l=1}^{m_j}\frac{\lambda-\lambda_j}{2\lambda_j}\left(t_\ve\alpha_{j,l,\ve}\right)^2-\lambda t_\ve^2D_{4,2}\mu_\ve^2|\log\mu_\ve^2|\\
&&+\sum_{j=1}^{k}\frac{\lambda_j-\lambda}{2\lambda_j}\|w_{j,\ve}\|_{L^2(\Omega)}^2-\int_{\Omega}\left(t_\ve\mathcal{W}_{\mu_\ve,\eta_\ve}\right)^3\left(\sum_{j=1}^k(w_{j,\ve}-v_{j,\ve})\right)dx\\
&&-\frac{3}{2}\int_{\Omega}\left(t_\ve\mathcal{W}_{\mu_\ve,\eta_\ve}\right)^2\left(\sum_{j=1}^k(w_{j,\ve}-v_{j,\ve})\right)^2dx-\frac14\left\|\sum_{j=1}^k(w_{j,\ve}-v_{j,\ve})\right\|_{L^4(\Omega)}^4\\
&&+\mathcal{O}(\mu_\ve^2)\\
&\leq&\frac{1}{4}S^2-\lambda t_\ve^2D_{4,2}\mu_\ve^2|\log\mu_\ve^2|+\sum_{j=1}^{k}\frac{\lambda_j-\lambda}{2}\|w_{j,\ve}\|_{L^2(\Omega)}^2\\
&&+\mathcal{O}\left(\sum_{j=1}^{k-1}\|w_{j,\ve}\|_{L^2(\Omega)}\mu_\ve+\mu_\ve^2\right)\\
&\leq&\frac{1}{4}S^2-\lambda t_\ve^2D_{4,2}\mu_\ve^2|\log\mu_\ve^2|+\mathcal{O}(\mu_\ve^2),
\end{eqnarray*}
which contradicts \eqref{low} for $\ve>0$ sufficiently small.  It follows that $\{u_{\lambda,sg}\}$ is compact in $H_0^1(\Omega)$, which implies that the limit, say $u_{\lambda_k,sg}$ is a sign-changing solution of the $4d$ Brezis-Nirenberg equation~\eqref{BN} for $\lambda=\lambda_k$ satisfying $E_{\lambda_k}(u_{\lambda_k,sg})=\frac{1}{4}S^{2}$.  Clearly, $u_{\lambda_k,sg}$ must be the least energy sign-changing solution of the $4d$ Brezis-Nirenberg equation~\eqref{BN} for $\lambda=\lambda_k$.

\vskip0.06in

{\bf Step~3}\quad We prove that $\{u_{\lambda,sg}\}$ is compact in $H_0^1(\Omega)$ as $\lambda\to\lambda_1^+$.

\vskip0.06in

Since $u_{\lambda_1,sg}$ must be the least energy sign-changing solution of the $4d$ Brezis-Nirenberg equation~\eqref{BN} for $\lambda=\lambda_1$, a standard perturbation argument implies that 
\begin{eqnarray*}
\inf_{\mathcal{P}_1}E_{\lambda}(u)<\frac{1}{4}S^{2}-\ve C,
\end{eqnarray*}
where $C>0$ is a constant.  Thus, $\{u_{\lambda,sg}\}$ is also compact as $\lambda\to\lambda_1^+$ by going through the same arguments used in the second step for $k\geq2$.  The monotonicity and continuity of $E_{\lambda}(u_{\lambda,sg})$ is also standard now according to the complete existence theory of the $4d$ Brezis-Nirenberg equation~\eqref{BN}, so we omit it here.
\hfill$\Box$

\section{Appendix: Some useful estimates}
In this section, we provide some estimates which will be frequently used in this paper.
\begin{lemma}\label{App-esti}
Let $\alpha>0$ and $\beta,\gamma\geq0$, then for any $\mu>0$ and $\xi\in\Omega$ such that $\kappa=\frac{\mu}{d(\xi,\partial \Omega)}<<1$, we have
\begin{eqnarray*}
\left(\int_{B_{d(\xi,\partial\Omega)}(\xi)}U_{\mu,\xi}^{\beta}\left|\Phi_{\mu,\xi}^i\right|^{\gamma}dx\right)^{\frac{1}{\alpha}}\lesssim\left\{\aligned
&\mu^{\frac{2N-(N-2)(\beta+\gamma)}{2\alpha}}, \quad N<(N-2)\beta+(N-1)\gamma\text{ and }1\leq i\leq N,\\
&\mu^{\frac{N+\gamma}{2\alpha}}\left|\log \kappa\right|^{\frac{1}{\alpha}}, \quad N=(N-2)\beta+(N-1)\gamma\text{ and }1\leq i\leq N,\\
&\frac{\mu^{\frac{2N-(N-2)(\beta+\gamma)}{2\alpha}}}{\kappa^{\frac{N-(N-2)\beta-(N-1)\gamma}{\alpha}}}, \quad N>(N-2)\beta+(N-1)\gamma\text{ and }1\leq i\leq N,\\
&\mu^{\frac{2N-(N-2)(\beta+\gamma)}{2\alpha}}, \quad N<(N-2)(\beta+\gamma)\text{ and }i=0,\\
&\mu^{\frac{N}{2\alpha}}\left|\log \kappa\right|^{\frac{1}{\alpha}}, \quad N=(N-2)(\beta+\gamma)\text{ and }i=0,\\
&\frac{\mu^{\frac{2N-(N-2)(\beta+\gamma)}{2\alpha}}}{\kappa^{\frac{N-(N-2)(\beta+\gamma)}{\alpha}}}, \quad N>(N-2)(\beta+\gamma)\text{ and }i=0
\endaligned\right.
\end{eqnarray*}
and
\begin{eqnarray*}
\left(\int_{\Omega\backslash B_{d(\xi,\partial\Omega)}(\xi)}U_{\mu,\xi}^{\beta}\left|\Phi_{\mu,\xi}^i\right|^{\gamma}dx\right)^{\frac{1}{\alpha}}\lesssim\left\{\aligned
&\frac{\mu^{\frac{2N-(N-2)(\beta+\gamma)}{2\alpha}}}{\kappa^{\frac{N-(N-2)\beta-(N-1)\gamma}{\alpha}}}, \quad N<(N-2)\beta+(N-1)\gamma\text{ and }1\leq i\leq N,\\
&\mu^{\frac{N+\gamma}{2\alpha}}\left|\log d(\xi,\partial\Omega)\right|^{\frac{1}{\alpha}}, \quad N=(N-2)\beta+(N-1)\gamma\text{ and }1\leq i\leq N,\\
&\mu^{\frac{(N-2)\beta+N\gamma}{2\alpha}}, \quad N>(N-2)\beta+(N-1)\gamma\text{ and }1\leq i\leq N,\\
&\frac{\mu^{\frac{2N-(N-2)(\beta+\gamma)}{2\alpha}}}{\kappa^{\frac{N-(N-2)(\beta+\gamma)}{\alpha}}}, \quad N<(N-2)(\beta+\gamma)\text{ and }i=0,\\
&\mu^{\frac{N}{2\alpha}}\left|\log d(\xi,\partial\Omega)\right|^{\frac{1}{\alpha}}, \quad N=(N-2)(\beta+\gamma)\text{ and }i=0,\\
&\mu^{\frac{(N-2)(\beta+\gamma)}{2\alpha}}, \quad N>(N-2)(\beta+\gamma)\text{ and }i=0,
\endaligned\right.
\end{eqnarray*}
where $U_{\mu,\xi}$ and $\Phi_{\mu,\xi}^i$ are given by \eqref{Talenti} and \eqref{kernel}, respectively.
\end{lemma}
\begin{proof}
The conclusion follows from direct calculations.  Thus, we omit it here.
\end{proof}

\end{document}